\RequirePackage{fix-cm}
\documentclass[smallextended]{svjour3}       
\smartqed  

\usepackage[toc,page,title,titletoc,header]{appendix}

\usepackage{graphicx}
\usepackage{subfigure}
\usepackage[misc]{ifsym}
\usepackage{amssymb,amsfonts,mathrsfs,amsmath,multirow,mathtools,color}
\usepackage{cite,hyperref,caption}
\allowdisplaybreaks[4]

\newcommand{\be}{\begin{equation}}
\newcommand{\ee}{\end{equation}}
\newcommand\bea{\begin{eqnarray}}
\newcommand\eea{\end{eqnarray}}
\newcommand{\bean}{\begin{eqnarray*}}
\newcommand{\eean}{\end{eqnarray*}}

\usepackage{cases}
\newcommand\bcase{\begin{numcases}{}}
\newcommand\ecase{\end{numcases}}

\begin{document}

\title{Uniform convergence and stability of linearized fourth-order conservative compact scheme for Benjamin-Bona-Mahony-Burgers' equation
\thanks{Qifeng Zhang was supported in part by Natural Science Foundation of China (No. 11501514),
in part by the Natural Sciences Foundation of Zhejiang Province under Grant LY19A010026,
in part by project funded by China Postdoctoral Science Foundation under Grant 2018M642131 when he studied in Southeast University.}
}
\titlerunning{Uniform convergence and stability of compact schemes for BBMB equation}        
\author{Qifeng Zhang \and Lingling Liu}
\institute{\Letter\; Qifeng Zhang \at Department of Mathematics, Zhejiang Sci-Tech University, Hangzhou, 310018, China \\Department of Mathematics, Southeast University, Nanjing, 210096, China \\
               \email{zhangqifeng0504@gmail.com}
		\and Lingling Liu \at Department of Mathematics, Zhejiang Sci-Tech University, Hangzhou, 310018, China\\ \email{2710751137@qq.com}
}

\date{Received: date / Accepted: date}

\maketitle

\begin{abstract}
In the paper,
a newly developed three-point fourth-order compact operator is utilized to construct an efficient compact finite difference scheme
for the Benjamin-Bona-Mahony-Burgers' (BBMB) equation.
Detailed derivation is carried out based on the reduction order method together with a three-level linearized technique.
The conservative invariant, boundedness and unique solvability are studied at length. The uniform convergence is proved by the technical
energy argument with the optimal convergence order $\mathcal{O}(\tau^2+h^4)$ in the sense of the maximum norm.
The almost unconditional stability can be achieved based on the uniform boundedness of the numerical solution.
The present scheme is very efficient in practical computation since only a system of linear equations with a symmetric circulant matrix needing to be solved
 at each time step.  The extensive numerical examples
verify our theoretical results and demonstrate the superiority of the scheme when compared with state-of-the-art those in the references.
\keywords{BBMB equation \and Reduction order method \and Linearized compact scheme \and Boundedness \and Uniform convergence}
\end{abstract}

\section{Introduction}
\label{intro}
The classical nonlinear Benjamin-Bona-Mahony (BBM) equation can describe the unidirectional propagation of weakly nonlinear
long waves in the presence of dispersion as follows
\begin{align}\label{equation_1}
  u_{t}-\frac{\mu}{6} u_{xxt}+ \frac{3\varepsilon}{2} uu_{x}+u_{x} = 0, \quad x\in  \mathbb{R},\; 0<t\leqslant T,
\end{align}
where $\varepsilon>0$ and $\mu$ are the parameters with the same order \cite{BBM1972}.
Compared with the well-known Korteweg-de Vries (KdV) equation
\begin{equation}\label{equation_2}
  u_t+u_x+\frac{3\varepsilon}{2}uu_x+\frac{\mu}{6}u_{xxx}=0,\quad x\in\mathbb{R},\; 0<t\leqslant T,
\end{equation}
\eqref{equation_1} is proposed as an analytically advantageous alternative.
\eqref{equation_1} and \eqref{equation_2} are both derived from the Green-Naghdi equations and they are
asymptotically equivalent in the limit $\varepsilon=\mu\rightarrow 0$ since $u_{xxx} = u_{xxt}+\mathcal{O}(\mu)$,
but enjoying different properties, see \cite{Lan2013} for detailed explanation.
In many applications, when the dissipation effect cannot be ignored, $-\nu u_{xx}$ have to be added and \eqref{equation_1} becomes
the known BBMB equation as
\begin{align}
u_{t}-\frac{\mu}{6} u_{xxt}+\frac{3\varepsilon}{2} uu_{x}+ u_{x}-\nu u_{xx}=0,\quad x\in\mathbb{R},\; 0<t\leqslant T,\label{equation_3}
\end{align}
which describes the propagation of small-amplitude long waves in a nonlinear dispersive media.
For the well-posedness, existence, uniqueness, regularity results, long time dynamics and the numerical simulation
for \eqref{equation_3} and its special cases are referred to \cite{MM1998, MM1999,KMO2000,Zhang1995,MS2005,Wang2014,Wang2016,PZ2002,LXC2020}.

In this paper, we are aimed to develop and analyze a high-order conservative difference
approximation for the BBMB equation as
\begin{align}
u_{t}-\mu u_{xxt}+\gamma uu_{x}+\kappa u_{x}-\nu u_{xx}=0,\quad x\in\mathbb{R},\; 0<t\leqslant T,\label{eq1}
\end{align}
with the periodic boundary condition
\begin{align}
u(x,t)=u(x+L,t),\quad x\in\mathbb{R},\; 0<t\leqslant T,\label{per}
\end{align}
and the initial-value condition
\begin{align}
u(x,0)=\varphi(x),\quad x\in\mathbb{R},\label{eq2}
\end{align}
where $\mu$ and $\gamma$ are non-negative constants, $\kappa$ and $\nu$ are parameters and $L$ denotes the spatial period.

In order to explore the solutions and their properties of the BBMB equation,
researchers racked their wits to develop various analytical methods for seeking the exact solutions of the BBMB equation.
For instance, Yin \emph{et al.} \cite{YCJ2008} employed the weighted energy method to investigate the time decay rate
of traveling waves of Cauchy problem of the BBMB equation. Est\'{e}vez \emph{et al.} \cite{EKNN2009}
studied the travelling wave solutions for the generalized BBMB equation systematically by using
the factorization technique. Besse \emph{et al.} \cite{BMN2018} developed the exact artificial boundary
conditions for the linearized BBM equation. Al-Khaled \emph{et al.} \cite{KMA2005}
considered solitary wave solutions of the BBMB equation by using the  decomposition method.
Fakhari \emph{et al.} \cite{FDE2007} approximated the explicit solutions of the nonlinear BBMB equation with high-order nonlinear term
via the homotopy analysis method.  Tari \emph{et al.} \cite{TG2007} used He's methods to obtain the explicit solutions
of the BBMB equation and compared with the exact solutions. Ganji \emph{et al.} \cite{GGB2009} solved the special form solutions
of the BBMB equation by the Exp-Function method. Based on the well-known tanh-coth method,
Cesar \emph{et al.} \cite{SSF2010} obtained new periodic soliton solutions for the generalized BBMB equation.
Noor \emph{et al.} \cite{NNA2011} constructed some new solitary solutions of the BBM equation by using the exp-function method.
Abbasbandy \cite{AS2010}  used the first integral method  to find some new exact solutions for the BBMB equation and
Bruz\'{o}n \cite{BGR2016}  studied some nontrivial conservation laws for the BBMB equation with the help of the multiplier's method.

On the other hand, there have been many attempts to approximate the solutions for the BBMB equation and its simplified version numerically.
For example, Guo \cite{GS2000} proposed a Laguerre-Galerkin method to solve the BBM equation on a semi-infinite interval.
Omrani \cite{KO2006} considered a fully discrete Galerkin method for the BBM equation. Soon afterwards, Omrani \emph{et al.} \cite{OA2008} used
Crank-Nicolson-type implicit finite difference method to solve the BBMB equation with the second-order accuracy in maximum norm.
They \cite{KKAO2008} further employed Galerkin finite element method in space combined with the implicit Euler method in
time for solving the generalized BBMB equation. Berikelashvili \emph{et al.} \cite{BM2011} explored a linearized difference scheme for solving
the regularized long-wave equation, which can be viewed as a special case of the BBMB equation with $\nu=0$.
They \cite{BM2014} also analyzed the convergence of a type of the difference scheme for the generalized BBMB equation.
Based on the meshless method of radial basis functions, Dehghan \emph{et al.} \cite{DAM2014} solved a high dimensional generalized BBMB equation.
They \cite{DAM2015} further considered the interpolating element-free Galerkin technique for the high dimensional BBMB equation.
Zarebnia \emph{et al.} \cite{ZP2016,SJ2016} used the collocation method and spectral meshless radial point interpolation, respectively, to solve the
BBMB equation. Based on hybridization of Lucas and Fibonacci polynomials,  Oru\c{c} \emph{et al.} \cite{OO2017} solved the generalized BBMB equation
in one and two dimension, respectively. Kundu \emph{et al.} \cite{KPK2018} proposed a semidiscrete Galerkin method and discussed
stabilization results for the semidiscrete  scheme with optimal error estimate. Zhang  \emph{et al.} \cite{ZLZ2020} established two linearized implicit difference schemes for the BBMB equation, in which the convergence orders both were two.

A review of all the above numerical methods reveals that higher-order algorithms are still scarce,
let alone the uniform error estimate of the higher-order algorithms.
To the authors' best knowledge, only Mohebbi and Faraz \cite{MF2017} propose a fourth-order algorithm for solving the BBMB equation with five points
in space and obtains the infinite error estimate.
However, when deal with the points near boundary, ghost points or fictitious points are requisite.
In addition, the stability in \cite{MF2017} is also missing.
In order to avoid the difficulty caused by the discretization near the boundary points,
we first developed three-point fourth-order compact technique for the Burgers' equation in \cite{WZS2019}
and further extended it to the BBMB equation in current paper.
Moreover, we extensively and deeply studied the convergence and stability of the compact difference scheme for the BBMB equation.

The compact difference scheme as one of the most practical numerical techniques
has the significant advantages over standard finite difference methods. Specifically:
I) A smaller matrix stencil generates higher order accuracy;
II) A larger stability domain allows larger spatial and temporal step sizes;
III) It owns a better resolution for high frequency waves;
IV) It is more suitable for long time integrations;
V) Fewer boundary point makes the discretization of the boundary easier.

The compact difference scheme in the present paper not only possess all of these advantages, but also does not incur extra computational
cost. Furthermore, our scheme is linearly implicit with the exact well-defined conservative invariant.
The main difficulties for the high-order approximation of the strong nonlinear term $uu_x$ involving the optimal convergence and stability are completely overcome based on the newly discovered compact operator, which makes the numerical analysis feasible and toilless.

 The main contribution lies in that the maximum error estimate and
 the optimal convergence order $O(\tau^2+h^4)$ are obtained. The proof of the convergence in
 pointwise sense is novel and technical. Compared our numerical results with those calculated in
 \cite{MF2017} is carried out, which demonstrates the effectiveness and advantage of the present algorithm.
 Moreover, the almost unconditional stability in the maximum norm is also proved in detail.

The rest of the paper is organized as follows. In Sect. \ref{Sec:2}, some requisite notations and useful lemmas are presented.
A three-level linearized compact difference scheme is derived in Sect. \ref{Sec:3} based on the reduction order method.
Conservative invariant and boundedness are obtained in Sect. \ref{Sec:4}.
The unique solvability is proved strictly in Sect. \ref{Sec:5}.
The uniform convergence and stability are proved at length in Sect. \ref{Sec:6}, which are the main part of the paper.
Several numerical experiments are presented in Sect. \ref{Sec:7} followed by a conclusion in Sect. \ref{Sec:8}.

\section{Notations and Lemmas}
\label{Sec:2}
\setcounter{equation}{0}
We firstly introduce some useful notations.
Take two positive integers $M$, $N$, let $h=L/M$, $\tau=T/N$.
Denote $x_{i}= ih,\; i\in Z$,
$t_{k}=k\tau, \;0\leqslant k \leqslant N$, $t_{k+\frac{1}{2}}=(t_{k}+t_{k+1})/2$;
$\Omega_{h}=\{x_{i}\,|\,x_{i}=ih,\; i\in Z\}$, $\Omega_{\tau}=\{t_{k}\,|\,t_{k}=k\tau,\; 0\leqslant k\leqslant N\}$,
$\Omega_{h\tau}=\Omega_{h}\times\Omega_{\tau}$. For any grid function $u=\{u_{i}^{k}\,|\,i\in Z,\;0\leqslant k\leqslant N\}$
defined on $\Omega_{h\tau}$, introduce the following notations
\begin{align*}
& \delta_{x}u_{i+\frac{1}{2}}^{k}=\frac{1}{h}(u_{i+1}^{k}-u_{i}^{k}), \; \delta^{2}_{x}u_{i}^{k}=\frac{1}{h}(\delta_{x}u_{i+\frac{1}{2}}^{k}-\delta_{x}u_{i-\frac{1}{2}}^{k}), \; \Delta_{x}u_{i}^{k}=\frac{1}{2h}(u_{i+1}^{k}-u_{i-1}^{k}), \\
& u_{i}^{k+\frac{1}{2}}=\frac{1}{2}(u_{i}^{k}+u_{i}^{k+1}),\; u_{i}^{\bar{k}}=\frac{1}{2}(u_{i}^{k-1}+u_{i}^{k+1}), \; \delta_{t}u_{i}^{k+\frac{1}{2}}=\frac{1}{\tau}(u_{i}^{k+1}-u_{i}^{k}),\;  \Delta_{t}u_{i}^{k}=\frac{1}{2\tau}(u_{i}^{k+1}-u_{i}^{k-1}).
\end{align*}

Denote
\begin{align*}
\mathcal{U}_{h}=\left\{u\,|\,u=\{u_{i}\},\; u_{i+M}=u_{i}\right\}.
\end{align*}
For any grid functions $u$, $w\in\mathcal{U}_{h}$, define the discrete inner products
\begin{align*}
(u,w)=h\sum_{i=1}^{M}u_{i}w_{i},\quad \langle \delta_{x}u,\delta_{x}w\rangle=h\sum_{i=1}^{M}(\delta_{x}u_{i-\frac{1}{2}})(\delta_{x}w_{i-\frac{1}{2}})
\end{align*}
and the corresponding norms (seminorm)
\begin{align*}
\|u\|=\sqrt{(u,u)},\qquad |u|_{1}=\sqrt{\langle\delta_{x}u,\delta_{x}u\rangle},\qquad
\|u\|_{\infty}=\max\limits_{1\leqslant i\leqslant M}|u_{i}|.
\end{align*}
Moreover, define the function (see \cite{Guobook})
\begin{align*}
\psi(u,v)_{i}=\frac{1}{3}[u_{i}\Delta_{x}v_{i}+\Delta_{x}(uv)_{i}],\quad 1\leqslant i\leqslant M.
\end{align*}
\begin{lemma}\label{lemma1} {\rm \cite{Sunbook}}
For any grid functions $u$, $w\in\mathcal{U}_{h}$, we have
\begin{align*}
\|u\|_{\infty}\leqslant\frac{\sqrt{L}}{2}|u|_{1},\quad |u|_{1}\leqslant\frac{2}{h}\|u\|,\quad \|u\|\leqslant\frac{L}{\sqrt{6}}|u|_{1},
\quad (\delta_{x}^{2}u,w)=-\langle \delta_{x}u,\delta_{x}w\rangle.
\end{align*}
\end{lemma}
\begin{lemma}\label{lemma2}
 For any grid functions $u$, $w\in\mathcal{U}_{h}$, we have
\begin{align*}
(\psi(u,w),w)=0,\quad (\Delta_{x}u,u)=0,\quad (\Delta_xu, \delta_x^2 u)=0.
\end{align*}
\end{lemma}
\begin{proof}
  The first and second equalities come from \cite{SS2015}. We will prove the third equality briefly below.
 According to the notations defined previously, we have
   \begin{align*}
   \delta_{x}(\Delta_{x}u)_{i}&=\;\frac{1}{h}\left[(\Delta_{x}u)_{i+\frac{1}{2}}-(\Delta_{x}u)_{i-\frac{1}{2}}\right]\\
   &=\;\frac{1}{h}\left[\frac{1}{2h}(u_{i+\frac{3}{2}}-u_{i-\frac{1}{2}})-\frac{1}{2h}(u_{i+\frac{1}{2}}-u_{i-\frac{3}{2}})\right]\\
   &=\;\frac{1}{2h}\left[\frac{1}{h}(u_{i+\frac{3}{2}}-u_{i+\frac{1}{2}})-\frac{1}{h}(u_{i-\frac{1}{2}}-u_{i-\frac{3}{2}})\right]\\
   &=\;\frac{1}{2h}(\delta_{x}u_{i+1}-\delta_{x}u_{i-1})\\
   &=\;\Delta_{x}(\delta_{x}u)_{i}.
  \end{align*}
  From the definition of the discrete inner products and the second equality, we have
  \begin{align*}
  (\Delta_{x}u, \delta_{x}^2 u)&=\;-(\delta_{x}(\Delta_{x}u),\delta_{x}u)\\
  &=\;-h\sum_{i=1}^{M}\delta_{x}(\Delta_{x}u)_{i}\cdot\delta_{x}u_{i}\\
  &=\;-h\sum_{i=1}^{M}\Delta_{x}(\delta_{x}u)_{i}\cdot\delta_{x}u_{i}\\
  &=\;-(\Delta_{x}(\delta_{x}u),\delta_{x}u)=\;0.
  \end{align*}
  This completes the proof.
\end{proof}

\begin{lemma}\label{lemma3}
Let $f(x)\in \mathrm{C}^{5}[x_{i-1},x_{i+1}]$ and denote $F_{i}=f(x_{i})$ and $G_{i}=f''(x_{i})$, then we have
\begin{align*}
&f(x_{i})f'(x_{i})=\psi(F,F)_{i}-\frac{h^{2}}{2}\psi(G,F)_{i}+\mathcal{O}(h^{4}),\quad 1\leqslant i\leqslant M,\\
&f'(x_{i})=\Delta_{x}F_{i}-\frac{h^{2}}{6}\Delta_{x}G_{i}+\mathcal{O}(h^{4}),\quad 1\leqslant i\leqslant M,\\
&f''(x_{i})=\delta_{x}^{2}F_{i}-\frac{h^{2}}{12}\delta_{x}^{2}G_{i}+\mathcal{O}(h^{4}),\quad 1\leqslant i\leqslant M.
\end{align*}
\end{lemma}
\begin{proof}
  The first and third equalities come from \cite{WZS2019} and \cite{Sun2009}, respectively.
  The second equality is immediately obtained by Taylor expansion. We omit it here for sake of brevity.
\end{proof}
\begin{lemma}\label{lemma4}
For any grid functions $u$, $v$, $S\in\mathcal{U}_{h}$, satisfying
\begin{align}
&v_{i}=\delta_{x}^{2}u_{i}-\frac{h^{2}}{12}\delta_{x}^{2}v_{i}+S_{i},\quad 1\leqslant i\leqslant M,\label{1}\\
&u_{i}=u_{i+M},\quad 0\leqslant i\leqslant M,\label{2}\\
&v_{i}=v_{i+M},\quad 0\leqslant i\leqslant M,\label{3}
\end{align}
we have
\begin{align}
&(v,u)=-|u|_{1}^{2}-\frac{h^{2}}{12}\|v\|^{2}+\frac{h^{4}}{144}|v|_{1}^{2}
+\frac{h^{2}}{12}(v,S)+(S,u),\label{lem4-1}\\
&(v,u)\leqslant-|u|_{1}^{2}-\frac{h^{2}}{18}\|v\|^{2}+\frac{h^{2}}{12}(v,S)
+(S,u),\label{lem4-2}\\
&(\Delta_{x}v,u)=\frac{h^{2}}{12}(\Delta_{x}v,S)+(\Delta_{x}S,u).\label{lem4-3}
\end{align}
\end{lemma}
\begin{proof}
Taking the inner product of \eqref{1} with $u$ and noticing \eqref{2}--\eqref{3}, we have
\begin{align*}
(v,u)
=\;&\left(\delta_{x}^{2}u-\frac{h^{2}}{12}\delta_{x}^{2}v+S,u\right)\nonumber\\
=\;&(\delta_{x}^{2}u,u)-\frac{h^{2}}{12}(\delta_{x}^{2}v,u)+(S,u)\nonumber\\
=\;&-|u|_{1}^{2}-\frac{h^{2}}{12}(v,\delta_{x}^{2}u)+(S,u)\nonumber\\
=\;&-|u|_{1}^{2}-\frac{h^{2}}{12}\left(v,v+\frac{h^{2}}{12}\delta_{x}^{2}v-S\right)
+(S,u)\nonumber\\
=\;&-|u|_{1}^{2}-\frac{h^{2}}{12}\|v\|^{2}+\frac{h^{4}}{144}|v|_{1}^{2}
+\frac{h^{2}}{12}(v,S)+(S,u).
\end{align*}
With the help of Lemma \ref{lemma1}, we have
\begin{align*}
(v,u)\leqslant-|u|_{1}^{2}-\frac{h^{2}}{18}\|v\|^{2}+\frac{h^{2}}{12}(v,S)
+(S,u).
\end{align*}
Combining \eqref{1} with Lemmas \ref{lemma1}--\ref{lemma2}, we have
\begin{align*}
(\Delta_{x}v,u)
=\;&\left(\Delta_{x}\left(\delta_{x}^{2}u-\frac{h^{2}}{12}\delta_{x}^{2}v+S\right)
,u\right)\nonumber\\
=\;&(\Delta_{x}(\delta_{x}^{2}u),u)-\frac{h^{2}}{12}(\Delta_{x}(\delta_{x}^{2}v),u)
+(\Delta_{x}S,u)\nonumber\\
=\;&-(\Delta_{x}(\delta_{x}u),\delta_{x}u)-\frac{h^{2}}{12}(\Delta_{x}v,\delta_{x}^{2}u)
+(\Delta_{x}S,u)\nonumber\\
=\;&-\frac{h^{2}}{12}(\Delta_{x}v,\delta_{x}^{2}u)+(\Delta_{x}S,u)\nonumber\\
=\;&-\frac{h^{2}}{12}(\Delta_{x}v,v+\frac{h^{2}}{12}\delta_{x}^{2}v-S)+(\Delta_{x}S,u)\nonumber\\
=\;&-\frac{h^{2}}{12}(\Delta_{x}v,v)+\frac{h^{4}}{144}(\Delta_{x}(\delta_{x}v),\delta_{x}v)
+\frac{h^{2}}{12}(\Delta_{x}v,S)+(\Delta_{x}S,u)\nonumber\\
=\;&\frac{h^{2}}{12}(\Delta_{x}v,S)+(\Delta_{x}S,u).
\end{align*}
This completes the proof.
\end{proof}
\begin{lemma}\label{lemma5}
For any grid functions $u$, $v\in\mathcal{U}_{h}$, we have
\begin{align*}
\Delta_{x}(uv)_{i}=\frac{1}{2}\left(\delta_{x}u_{i+\frac{1}{2}}\right)v_{i+1}
+\frac{1}{2}\left(\delta_{x}u_{i-\frac{1}{2}}\right)v_{i-1}+u_{i}\Delta_{x}v_{i}.
\end{align*}
\end{lemma}
\begin{proof}
\begin{align*}
\Delta_{x}(uv)_{i}=&\;\frac{1}{2h}(u_{i+1}v_{i+1}-u_{i-1}v_{i-1})\nonumber\\
=&\;\frac{1}{2h}[(u_{i+1}-u_{i})v_{i+1}+(u_{i}-u_{i-1})v_{i-1}+u_{i}(v _{i+1}-v_{i-1})]\nonumber\\
=&\;\frac{1}{2}\left(\delta_{x}u_{i+\frac{1}{2}}\right)v_{i+1}+\frac{1}{2}\left(\delta_{x}u_{i-\frac{1}{2}}\right)v_{i-1}+u_{i}\Delta_{x}v_{i}.
\end{align*}
This completes the proof.
\end{proof}
\section{Derivation of Compact Difference Scheme}
\label{Sec:3}
\setcounter{equation}{0}
Denote
\begin{align}
c_{0}=\max\limits_{0\leqslant x\leqslant L,0\leqslant t\leqslant T}\{|u(x,t)|,\ |u_{x}(x,t)|,\ |u_{xx}(x,t)|, \ |u_{xxx}(x,t)| \}.\label{eq3}
\end{align}

Let $v=u_{xx}$, then the problem \eqref{eq1}--\eqref{eq2} is equivalent to
\begin{align}
&u_{t}-\mu v_{t}+\gamma uu_{x}+\kappa u_{x}-\nu v=0,\quad x\in \mathbb{R},\; 0<t\leqslant T,\label{eq4}\\
&v=u_{xx},\quad x\in\mathbb{R},\; 0<t\leqslant T,\label{eq5}\\
&u(x,0)=\varphi(x),\quad x\in\mathbb{R},\label{eq6}\\
& u(x,t) = u(x+L,t),\quad v(x,t) = v(x+L,t),\quad x\in \mathbb{R},\; 0 < t \leqslant T.\label{eq4a}
\end{align}
Define the grid functions
\begin{equation*}
U_{i}^{k} = u(x_i,t_k),\quad V_{i}^{k}=v(x_{i},t_{k}), \quad 1\leqslant i \leqslant M, \; 0 \leqslant k \leqslant N.
\end{equation*}
With the help of Lemma \ref{lemma3}, we have
\begin{align}
&uu_{x}(x_{i},t_{k})=\psi(U^{k},U^{k})_{i}-\frac{h^{2}}{2}\psi(V^{k},U^{k})_{i}+\mathcal{O}(h^{4}),\label{eq7}\\
&u_{x}(x_{i},t_{k})=\Delta_{x}U_{i}^{k}-\frac{h^{2}}{6}\Delta_{x}V_{i}^{k}+\mathcal{O}(h^{4}),\label{eq8}\\
&u_{xx}(x_{i},t_{k})=\delta_{x}^{2}U_{i}^{k}-\frac{h^{2}}{12}\delta_{x}^{2}V_{i}^{k}+\mathcal{O}(h^{4}).\label{eq9}
\end{align}

Considering \eqref{eq4} at the point $(x_{i},t_{\frac{1}{2}})$, with the help of Taylor expansion and \eqref{eq7}--\eqref{eq8}, we have
\begin{align}
&\delta_{t}U_{i}^{\frac{1}{2}}-\mu\delta_{t}V_{i}^{\frac{1}{2}}+\gamma\left[\psi(U^{0},U^{\frac{1}{2}})_{i}
-\frac{h^{2}}{2}\psi(V^{0},U^{\frac{1}{2}})_{i}\right]+\kappa\left(\Delta_{x}U_{i}^{\frac{1}{2}}
-\frac{h^{2}}{6}\Delta_{x}V_{i}^{\frac{1}{2}}\right)-\nu V_{i}^{\frac{1}{2}}=Q_{i}^{0},\nonumber\\
&\qquad\qquad\qquad\qquad\qquad\qquad\qquad\qquad 1\leqslant i\leqslant M,\label{eq10}
\end{align}
where
\begin{align}
|Q_{i}^0|\leqslant c_{1}(\tau+h^{4}),\quad 1\leqslant i\leqslant M,\label{eq11}
\end{align}
with $c_{1}$ being a positive constant.
Analogously, considering \eqref{eq4} at the point $(x_{i},t_{k})$, we have
\begin{align}
&\Delta_{t}U_{i}^{k}-\mu\Delta_{t}V_{i}^{k}+\gamma\left[\psi(U^{k},U^{\bar{k}})_{i}
-\frac{h^{2}}{2}\psi(V^{k},U^{\bar{k}})_{i}\right]+\kappa\left(\Delta_{x}U_{i}^{\bar{k}}
-\frac{h^{2}}{6}\Delta_{x}V_{i}^{\bar{k}}\right)-\nu V_{i}^{\bar{k}}=Q_{i}^{k},\nonumber\\ &\qquad\qquad\qquad\qquad\qquad\qquad\quad 1\leqslant i\leqslant M,\; 1\leqslant k\leqslant N-1,\label{eq12}
\end{align}
where
\begin{align}
|Q_{i}^{k}|\leqslant c_{2}(\tau^{2}+h^{4}),\quad 1\leqslant i\leqslant M,\; 1\leqslant k\leqslant N-1,\label{eq13}
\end{align}
with $c_{2}$ being a positive constant.

Again considering \eqref{eq5} at the point $(x_{i},t_{k})$, we have
\begin{align}
V_{i}^{k}=\delta_{x}^{2}U_{i}^{k}-\frac{h^{2}}{12}\delta_{x}^{2}V_{i}^{k}+R_{i}^{k},\quad 1\leqslant i\leqslant M,\; 0\leqslant k\leqslant N,\label{eq14}
\end{align}
where
\begin{align}
|R_{i}^{k}|\leqslant c_{3}h^{4},\quad 1\leqslant i\leqslant M,\; 0\leqslant k\leqslant N,\label{eq15}
\end{align}
with $c_{3}$ being a positive constant.

Noticing the initial and boundary conditions \eqref{eq6}--\eqref{eq4a}, we have
\begin{align}
&U_{i}^{0}=\varphi(x_{i}),\quad 1\leqslant i\leqslant M, \label{eq16}\\
&U_{i}^{k}=U_{i+M}^{k},\quad V_{i}^{k}=V_{i+M}^{k},\quad 1\leqslant i\leqslant M,\; 0\leqslant k\leqslant N. \label{eq16a}
\end{align}

Omitting the small terms $Q_{i}^{k}$ and $R_{i}^{k}$, replacing the grid functions $U_{i}^{k}$, $V_{i}^{k}$ by $u_{i}^{k}$, $v_{i}^{k}$ in \eqref{eq10}, \eqref{eq12}, \eqref{eq14}, respectively,
and noticing the initial and boundary conditions \eqref{eq16}--\eqref{eq16a}, then we construct a finite difference scheme for \eqref{eq4}--\eqref{eq4a} as follows
\begin{align}
&\delta_{t}u_{i}^{\frac{1}{2}}-\mu\delta_{t}v_{i}^{\frac{1}{2}}+\gamma\left[\psi(u^{0},u^{\frac{1}{2}})_{i}
-\frac{h^{2}}{2}\psi(v^{0},u^{\frac{1}{2}})_{i}\right]+\kappa\left(\Delta_{x}u_{i}^{\frac{1}{2}}
-\frac{h^{2}}{6}\Delta_{x}v_{i}^{\frac{1}{2}}\right)-\nu v_{i}^{\frac{1}{2}}=0,\nonumber\\
&\quad \qquad\qquad\qquad\qquad\qquad\qquad\qquad 1\leqslant i\leqslant M,\label{eq17}\\
&\Delta_{t}u_{i}^{k}-\mu\Delta_{t}v_{i}^{k}+\gamma\left[\psi(u^{k},u^{\bar{k}})_{i}
-\frac{h^{2}}{2}\psi(v^{k},u^{\bar{k}})_{i}\right]+\kappa\left(\Delta_{x}u_{i}^{\bar{k}}
-\frac{h^{2}}{6}\Delta_{x}v_{i}^{\bar{k}}\right)-\nu v_{i}^{\bar{k}}=0,\nonumber\\ &\qquad\qquad\qquad\qquad\qquad\qquad\quad 1\leqslant i\leqslant M,\; 1\leqslant k\leqslant N-1,\label{eq18}\\
&v_{i}^{k}=\delta_{x}^{2}u_{i}^{k}-\frac{h^{2}}{12}\delta_{x}^{2}v_{i}^{k},\quad 1\leqslant i\leqslant M,\; 0\leqslant k\leqslant N,\label{eq19}\\
&u_{i}^{0}=\varphi(x_{i}),\quad 1\leqslant i\leqslant M,\label{eq20}\\
&u_i^k = u_{i+M}^k,\quad v_i^k = v_{i+M}^k,\quad 1\leqslant i\leqslant M,\; 0\leqslant k\leqslant N.\label{eq20a}
\end{align}

\begin{remark}
  The coefficient matrix is a symmetric circulant matrix which can be solved by fast Fourier transform efficiently.
\end{remark}

\section{Conservative Invariant and Boundedness}
\label{Sec:4}
\setcounter{equation}{0}
\setcounter{theorem}{0}
\begin{theorem}\label{Theorem:1}
 Suppose $\{u_{i}^{k},\;v_{i}^{k}\,|\,1\leqslant i\leqslant M,\;0\leqslant k\leqslant N\}$
 is the solution of \eqref{eq17}--\eqref{eq20a}. Then it holds that
\begin{align}
 &\;\frac{1}{2}(\|u^{1}\|^{2}+\|u^{0}\|^{2})+\frac{\mu}{2}\left[(|u^{1}|_{1}^{2}+|u^{0}|_{1}^{2})
+\frac{h^{2}}{12}(\|v^{1}\|^{2}+\|v^{0}\|^{2})-\frac{h^{4}}{144}(|v^{1}|_{1}^{2}+|v^{0}|_{1}^{2})\right]\nonumber\\
 &\; +\nu\tau\left(|u^{\frac{1}{2}}|_{1}^{2}+\frac{h^{2}}{12}\|v^{\frac{1}{2}}\|^{2}
-\frac{h^{4}}{144}|v^{\frac{1}{2}}|_{1}^{2}\right)\nonumber\\
=&\|u^{0}\|^{2}+\mu|u^{0}|_{1}^{2}+\frac{\mu h^{2}}{12}\|v^{0}\|^{2}-\frac{\mu h^{4}}{144}|v^{0}|_{1}^{2}, \label{eq21}\\
 &\;E(u^{k+1},u^{k})=E(u^{1},u^{0}),\quad 1\leqslant k\leqslant N-1,\label{eq22}
\end{align}
where
\begin{align*}
E(u^{k+1},u^{k})
= &\; \frac{1}{2}(\|u^{k+1}\|^{2}+\|u^{k}\|^{2})+2\nu\tau\sum_{l=1}^{k}\left(|u^{\bar{l}}|_{1}^{2}+\frac{h^{2}}{12}\|v^{\bar{l}}\|^{2}-\frac{ h^{4}}{144}|v^{\bar{l}}|_{1}^{2}\right)\nonumber\\
&+\frac{\mu}{2}\left[(|u^{k+1}|_{1}^{2}+|u^{k}|_{1}^{2})
+\frac{h^{2}}{12}(\|v^{k+1}\|^{2}+\|v^{k}\|^{2})-\frac{h^{4}}{144}(|v^{k+1}|_{1}^{2}+|v^{k}|_{1}^{2})\right]
.
\end{align*}
\end{theorem}
\begin{proof}
Taking the inner product of \eqref{eq17} with $u^{\frac{1}{2}}$ and applying Lemma \ref{lemma2}, we have
\begin{align}
(\delta_{t}u^{\frac{1}{2}},u^{\frac{1}{2}})-\mu(\delta_{t}v^{\frac{1}{2}},u^{\frac{1}{2}})-\frac{\kappa h^{2}}{6}(\Delta_{x}v^{\frac{1}{2}},u^{\frac{1}{2}})-\nu(v^{\frac{1}{2}},u^{\frac{1}{2}})=0.\label{eq23}
\end{align}
Averaging \eqref{eq19} with superscripts $k=0$ and $k=1$, it holds
\begin{align}
v_{i}^{\frac{1}{2}}=\delta_{x}^{2}u_{i}^{\frac{1}{2}}-\frac{h^{2}}{12}\delta_{x}^{2}v_{i}^{\frac{1}{2}},\quad 1\leqslant i\leqslant M.\label{eq24}
\end{align}
With the help of \eqref{eq24} and summation by parts, we have
\begin{align}
(\delta_{t}v^{\frac{1}{2}},u^{\frac{1}{2}})
=\;&\left(\delta_{t}\left(\delta_{x}^{2}u^{\frac{1}{2}}-\frac{h^{2}}{12}\delta_{x}^{2}v^{\frac{1}{2}}\right)
,u^{\frac{1}{2}}\right)\nonumber\\
=\;&-(\delta_{t}(\delta_{x}u^{\frac{1}{2}}),\delta_{x}u^{\frac{1}{2}})
-\frac{h^{2}}{12}(\delta_{t}v^{\frac{1}{2}},\delta_{x}^{2}u^{\frac{1}{2}})\nonumber\\
=\;&-\frac{1}{2\tau}(|u^{1}|_{1}^{2}-|u^{0}|_{1}^{2})-\frac{h^{2}}{12}\left(\delta_{t}v^{\frac{1}{2}},v^{\frac{1}{2}}
+\frac{h^{2}}{12}\delta_{x}^{2}v^{\frac{1}{2}}\right)\nonumber\\
=\;&-\frac{1}{2\tau}(|u^{1}|_{1}^{2}-|u^{0}|_{1}^{2})-\frac{h^{2}}{12}(\delta_{t}v^{\frac{1}{2}},v^{\frac{1}{2}})
-\frac{h^{4}}{144}(\delta_{t}v^{\frac{1}{2}},\delta_{x}^{2}v^{\frac{1}{2}})\nonumber\\
=\;&-\frac{1}{2\tau}\left[(|u^{1}|_{1}^{2}-|u^{0}|_{1}^{2})+\frac{h^{2}}{12}(\|v^{1}\|^{2}-\|v^{0}\|^{2})
-\frac{h^{4}}{144}(|v^{1}|_{1}^{2}-|v^{0}|_{1}^{2})\right].\label{eq25}
\end{align}
Applying \eqref{lem4-1} and \eqref{lem4-3} in Lemma \ref{lemma4}, we have
\begin{align}
&(\Delta_{x}v^{\frac{1}{2}},u^{\frac{1}{2}})=0,\label{eq26}\\
&(v^{\frac{1}{2}},u^{\frac{1}{2}})
=-|u^{\frac{1}{2}}|_{1}^{2}-\frac{h^{2}}{12}\|v^{\frac{1}{2}}\|^{2}+\frac{h^{4}}{144}|v^{\frac{1}{2}}|_{1}^{2}.\label{eq27}
\end{align}
Substituting \eqref{eq25}--\eqref{eq27} into \eqref{eq23}, we have
\begin{align*}
&\frac{1}{2\tau}(\|u^{1}\|^{2}-\|u^{0}\|^{2})+\frac{\mu}{2\tau}\left[(|u^{1}|_{1}^{2}-|u^{0}|_{1}^{2})+\frac{h^{2}}{12}(\|v^{1}\|^{2}-\|v^{0}\|^{2})
-\frac{h^{4}}{144}(|v^{1}|_{1}^{2}-|v^{0}|_{1}^{2})\right]\nonumber\\
&+\nu\left(|u^{\frac{1}{2}}|_{1}^{2}
+\frac{h^{2}}{12}\|v^{\frac{1}{2}}\|^{2}-\frac{h^{4}}{144}|v^{\frac{1}{2}}|_{1}^{2}\right)=0.
\end{align*}
Rearranging the above formula, we have
\begin{align}
&\frac{1}{2}(\|u^{1}\|^{2}+\|u^{0}\|^{2})+\frac{\mu}{2}\left[(|u^{1}|_{1}^{2}+|u^{0}|_{1}^{2})
+\frac{h^{2}}{12}(\|v^{1}\|^{2}+\|v^{0}\|^{2})-\frac{h^{4}}{144}(|v^{1}|_{1}^{2}+|v^{0}|_{1}^{2})\right]\nonumber\\
&+\nu\tau\left(|u^{\frac{1}{2}}|_{1}^{2}+\frac{h^{2}}{12}\|v^{\frac{1}{2}}\|^{2}
-\frac{h^{4}}{144}|v^{\frac{1}{2}}|_{1}^{2}\right)\nonumber\\
=\;&\|u^{0}\|^{2}+\mu|u^{0}|_{1}^{2}+\frac{\mu h^{2}}{12}\|v^{0}\|^{2}-\frac{\mu h^{4}}{144}|v^{0}|_{1}^{2}. \label{eq28}
\end{align}
Taking the inner product of \eqref{eq18} with $u^{\bar{k}}$ and applying Lemma \ref{lemma2}, we have
\begin{align}
(\Delta_{t}u^{k},u^{\bar{k}})-\mu(\Delta_{t}v^{k},u^{\bar{k}})-\frac{\kappa h^{2}}{6}(\Delta_{x}v^{\bar{k}},u^{\bar{k}})-\nu(v^{\bar{k}},u^{\bar{k}})=0,\quad 1\leqslant k\leqslant N-1.\label{eq30}
\end{align}
Averaging \eqref{eq19} with superscripts $k-1$ and $k+1$, it holds
\begin{align}
v_{i}^{\bar{k}}=\delta_{x}^{2}u_{i}^{\bar{k}}-\frac{h^{2}}{12}\delta_{x}^{2}v_{i}^{\bar{k}},\quad 1\leqslant i\leqslant M,\; 1\leqslant k\leqslant N-1.\label{eq31}
\end{align}
With the help of \eqref{eq19}, \eqref{eq31}, summation by parts and similar to the derivation of \eqref{eq25}, we have
\begin{align}
&(\Delta_{t}v^{k},u^{\bar{k}})
= -\frac{1}{4\tau}\left[(|u^{k+1}|_{1}^{2}-|u^{k-1}|_{1}^{2}) + \frac{h^{2}}{12}(\|v^{k+1}\|^{2}-\|v^{k-1}\|^{2})
-\frac{h^{4}}{144}(|v^{k+1}|_{1}^{2}-|v^{k-1}|_{1}^{2})\right],\nonumber\\
&\qquad\qquad\qquad\qquad\qquad\qquad\qquad\qquad 1\leqslant k\leqslant N-1.\label{eq32}
\end{align}
Applying \eqref{lem4-1} and \eqref{lem4-3} in Lemma \ref{lemma4}, we have
\begin{align}
&(\Delta_{x}v^{\bar{k}},u^{\bar{k}})=0,\quad 1\leqslant k\leqslant N-1,\label{eq33}\\
&(v^{\bar{k}},u^{\bar{k}})=-|u^{\bar{k}}|_{1}^{2}-\frac{h^{2}}{12}\|v^{\bar{k}}\|^{2}
+\frac{h^{4}}{144}|v^{\bar{k}}|_{1}^{2},\quad 1\leqslant k\leqslant N-1.\label{eq34}
\end{align}
Substituting \eqref{eq32}--\eqref{eq34} into \eqref{eq30}, we have
\begin{align*}
&\frac{1}{4\tau}(\|u^{k+1}\|^{2}-\|u^{k-1}\|^{2})+\nu\left[|u^{\bar{k}}|_{1}^{2}+\frac{h^{2}}{12}\|v^{\bar{k}}\|^{2}
-\frac{h^{4}}{144}|v^{\bar{k}}|_{1}^{2}\right]\nonumber\\
&+\frac{\mu}{4\tau}\left[(|u^{k+1}|_{1}^{2}-|u^{k-1}|_{1}^{2})
+\frac{h^{2}}{12}(\|v^{k+1}\|^{2}-\|v^{k-1}\|^{2})
-\frac{h^{4}}{144}(|v^{k+1}|_{1}^{2}-|v^{k-1}|_{1}^{2})\right]=0,\nonumber\\
 & \qquad\qquad\qquad\qquad\qquad\qquad1\leqslant k\leqslant N-1.
\end{align*}
Consequently
\begin{align*}
E(u^{k+1},u^{k})=E(u^{k},u^{k-1}),\quad 1\leqslant k\leqslant N-1.
\end{align*}
By the recursion, we have
\begin{align*}
E(u^{k+1},u^{k})=E(u^{1},u^{0}),\quad 1\leqslant k\leqslant N-1.
\end{align*}
\end{proof}
\begin{remark}
  \eqref{eq21} and \eqref{eq22} can be rewritten as
  \begin{align}
   &\;\frac{1}{2}(\|u^{k+1}\|^{2}+\|u^{k}\|^{2})+\frac{\mu}{2}\left[(|u^{k+1}|_{1}^{2}+|u^{k}|_{1}^{2})
+\frac{h^{2}}{12}(\|v^{k+1}\|^{2}+\|v^{k}\|^{2})
-\frac{h^{4}}{144}(|v^{k+1}|_{1}^{2}+|v^{k}|_{1}^{2})\right]\nonumber\\
&\;+\nu\tau\left(|u^{\frac{1}{2}}|_{1}^{2}+\frac{h^{2}}{12}\|v^{\frac{1}{2}}\|^{2}
-\frac{h^{4}}{144}|v^{\frac{1}{2}}|_{1}^{2}\right)
+2\nu\tau\sum_{l=1}^{k}\left(|u^{\bar{l}}|_{1}^{2}+\frac{h^{2}}{12}\|v^{\bar{l}}\|^{2}-\frac{ h^{4}}{144}|v^{\bar{l}}|_{1}^{2}\right)\nonumber\\
=&\;\|u^{0}\|^{2}+\mu|u^{0}|_{1}^{2}+\frac{\mu h^{2}}{12}\|v^{0}\|^{2}-\frac{\mu h^{4}}{144}|v^{0}|_{1}^{2},
\quad 0\leqslant k\leqslant N-1.\label{E}
  \end{align}
\end{remark}
\begin{remark}
  Combining \eqref{eq28} with \eqref{E}, we have
  \begin{align*}
    \|u^k\|\leqslant 2\left(\|u^0\|+\mu|u^{0}|_{1}^{2}+\frac{\mu h^{2}}{12}\|v^{0}\|^{2}-\frac{\mu h^{2}}{144}|v^{0}|_{1}^{2}\right), \quad 1 \leqslant k \leqslant N.
  \end{align*}
\end{remark}

\section{Uniqueness}
\label{Sec:5}
\setcounter{equation}{0}
\begin{theorem}
The finite difference scheme \eqref{eq17}--\eqref{eq20a} is uniquely solvable.
\end{theorem}
\begin{proof}
From \eqref{eq19}--\eqref{eq20a}, it is easy to know  that $u^{0}$ and $v^{0}$ have been determined.
From \eqref{eq17} and \eqref{eq19}, a linear system of equations about $u^1$ and $v^1$ can be obtained with respect to the first level.
Now we consider its homogenous linear system of equations
\begin{align}
&\frac{1}{\tau}u_{i}^{1}-\frac{\mu}{\tau}v_{i}^{1}+\frac{\gamma}{2}\psi(u^{0},u^{1})_{i}-\frac{\gamma h^{2}}{4}\psi(v^{0},u^{1})_{i}+\frac{\kappa}{2}\Delta_{x}u_{i}^{1}-\frac{\kappa h^{2}}{12}\Delta_{x}v_{i}^{1}-\frac{\nu}{2}v_{i}^{1}=0,\quad 1\leqslant i\leqslant M,\label{eq38}\\
&v_{i}^{1}=\delta_{x}^{2}u_{i}^{1}-\frac{h^{2}}{12}\delta_{x}^{2}v_{i}^{1},\quad 1\leqslant i\leqslant M.\label{eq39}
\end{align}
Taking the inner product of \eqref{eq38} with $u^{1}$, and combining Lemma \ref{lemma2} with \eqref{eq39}, we have
\begin{align}
\frac{1}{\tau}\|u^{1}\|^{2}-\frac{\mu}{\tau}(v^{1},u^{1})-\frac{\kappa h^{2}}{12}(\Delta_{x}v^{1},u^{1})-\frac{\nu}{2}(v^{1},u^{1})=0.\label{eq40}
\end{align}
Applying \eqref{lem4-2} in Lemma \ref{lemma4}, we have
\begin{align}
&(v^{1},u^{1})\leqslant -|u^{1}|_{1}^{2}-\frac{h^{2}}{18}\|v^{1}\|^{2},\label{eq41}\\
&(\Delta_{x}v^{1},u^{1})=0.\label{eq42}
\end{align}
Substituting \eqref{eq41}--\eqref{eq42} into \eqref{eq40} and a calculation shows that
\begin{align*}
\frac{1}{\tau}\|u^{1}\|^{2}+\left(\frac{\mu}{\tau}+\frac{\nu}{2}\right)\cdot\left(|u^{1}|_{1}^{2}
+\frac{h^{2}}{18}\|v^{1}\|^{2}\right)\leqslant 0.
\end{align*}
Thus, it holds that
\begin{align*}
\|u^{1}\|=0,\quad \|v^{1}\|=0.
\end{align*}
Therefore, \eqref{eq38} and \eqref{eq39} only allow zero solutions,
which implies that \eqref{eq17} and \eqref{eq19} determine $u^{1}$, $v^{1}$ uniquely.

Now we suppose that $u^{k-1}$, $u^{k}$, $v^{k-1}$, $v^{k}$ have been determined.
From \eqref{eq18}--\eqref{eq19}, a linear system of equations with respect to $u^{k+1}$ and $v^{k+1}$ is obtained.
Now we consider the homogenous system of equations as follows
\begin{align}
&\frac{1}{2\tau}u_{i}^{k+1}-\frac{\mu}{2\tau}v_{i}^{k+1}+\frac{\gamma}{2}\psi(u^{k},u^{k+1})_{i}-\frac{\gamma h^{2}}{4}\psi(v^{k},u^{k+1})_{i}+\frac{\kappa}{2}\Delta_{x}u_{i}^{k+1}-\frac{\kappa h^{2}}{12}\Delta_{x}v_{i}^{k+1}-\frac{\nu}{2}v_{i}^{k+1}=0,\nonumber\\
&\qquad\qquad\qquad\qquad\qquad\qquad\qquad\qquad 1\leqslant i\leqslant M,\label{eq43}\\
&v_{i}^{k+1}=\delta_{x}^{2}u_{i}^{k+1}-\frac{h^{2}}{12}\delta_{x}^{2}v_{i}^{k+1},\quad 1\leqslant i\leqslant M.\label{eq44}
\end{align}
Taking the inner product of \eqref{eq43} with $u^{k+1}$ and applying Lemma \ref{lemma2} and \eqref{eq44}, we have
\begin{align}
\frac{1}{2\tau}\|u^{k+1}\|^{2}-\left(\frac{\mu}{2\tau}+\frac{\nu}{2}\right)(v^{k+1},u^{k+1})-\frac{\kappa h^{2}}{12}(\Delta_{x}v^{k+1},u^{k+1})=0.\label{eq45}
\end{align}
Combining \eqref{lem4-2} in Lemma \ref{lemma4} and noticing $S^k=0$, we have
\begin{align}
&(v^{k+1},u^{k+1})\leqslant-|u^{k+1}|_{1}^{2}-\frac{h^{2}}{18}\|v^{k+1}\|^{2},\label{eq46}\\
&(\Delta_{x}v^{k+1},u^{k+1})=0.\label{eq47}
\end{align}
Substituting \eqref{eq46}--\eqref{eq47} into \eqref{eq45}, we have
\begin{align*}
\frac{1}{2\tau}\|u^{k+1}\|^{2}+\left(\frac{\mu}{2\tau}+\frac{\nu}{2}\right)
\cdot\left(|u^{k+1}|_{1}^{2}+\frac{h^{2}}{18}\|v^{k+1}\|^{2}\right)\leqslant0.
\end{align*}
Then it holds that
\begin{align*}
\|u^{k+1}\|=0,\quad \|v^{k+1}\|=0.
\end{align*}
Therefore, \eqref{eq43} and \eqref{eq44} only allow zero solutions,
which implies that \eqref{eq18}--\eqref{eq19} determine $u^{k+1}$ and $v^{k+1}$ uniquely. By the mathematical induction, this completes the proof.
\end{proof}

\section{Convergence and Stability}\label{Sec:6}
\subsection{Convergence}
\setcounter{equation}{0}
\begin{theorem}[Convergence]\label{thm6.1}
Suppose $\{U_{i}^{k}$, $V_{i}^{k}\,|\, 1\leqslant i\leqslant M,\;0\leqslant k\leqslant N\}$ is the solution of
\eqref{eq10}, \eqref{eq12}, \eqref{eq14}, \eqref{eq16}, \eqref{eq16a}
$\{u_{i}^{k}$, $v_{i}^{k}\,|\, 1\leqslant i\leqslant M,\;0\leqslant k\leqslant N\}$
is the solution of \eqref{eq17}--\eqref{eq20a}. Denote
\begin{align*}
e_{i}^{k}=U_{i}^{k}-u_{i}^{k},\quad f_{i}^{k}=V_{i}^{k}-v_{i}^{k},\quad 1\leqslant i\leqslant M,\; 0\leqslant k\leqslant N,
\end{align*}
then there exist positive constants $h_0$, $\tau_0$, such that when
$h\leqslant h_0$, $\tau\leqslant \tau_0$ and $\tau^{2}+h^{4}\leqslant 1/c_{4}$, we have the error estimate
\begin{align}
|e^{k}|_{1}\leqslant c_{4}(\tau^{2}+h^{4}), \quad 0\leqslant k\leqslant N,\label{Err}
\end{align}
where
\begin{align*}
c_{4}=\max\left\{\sqrt{\frac{c_{5}}{\mu}},\sqrt{2c_{10}}\right\},
\end{align*}
with
\begin{align*}
c_{5}=&\;\left[\frac{27}{4}\left(\frac{\nu\tau_{0}}{2}-\mu\right)^{2}+\frac{3\kappa^{2}h_{0}^{2}\tau_{0}^{2}}{16}
+\frac{3\gamma^{2}c_{0}^{2}\tau_{0}^{2}h_{0}^2(h_{0}+1)^2}{4}+\frac{\mu h_{0}^{2}}{16}+\frac{\nu h_{0}^{2}\tau_{0}}{32}
+\frac{3}{2}\left(\mu+\frac{\nu\tau_{0}}{2}\right)^{2}\right.\nonumber\\
&\;\left.+\frac{\kappa^{2}h_{0}^{4}\tau_{0}^{2}}{576\mu}+\frac{\kappa^{2}h_{0}^{2}\tau_{0}^{2}}{24}\right]Lc_{3}^{2}+\frac{3Lc_{1}^{2}}{2},\\
c_{6}=&\;\frac{5\gamma^{2}(Lc_{0}+\sqrt{L})^{2}}{8}
+\frac{5\gamma^{2}}{8}\left(\frac{Lc_{0}h_{0}^{2}}{2}+\frac{3\sqrt{L}}{2}\sqrt{8
+2Lc_{3}^{2}h_{0}^{10}}\right)^{2}+\frac{5\kappa^{2}}{2}+\frac{10\kappa^{2}}{9}+\frac{5\gamma^{2}(Lc_{0}+c_{0})^{2}}{18},\\
c_{7}=&\;5\gamma^{2}(h_{0}c_{0}+c_{0})^{2}+\frac{3\kappa^{2}h_{0}^{2}}{16\mu}+\frac{\kappa^{2}h_{0}^{2}}{48\mu}
+\frac{\kappa^{2}h_{0}^{4}}{72}+\frac{\nu^{2}h_{0}^{2}}{8},\\
c_{8}=&\;\left(\frac{3\mu h_{0}^{2}}{32}+\frac{5\mu^{2}}{2}+\frac{1}{2}+\frac{5\kappa^{2}h_{0}^{2}}{18}+\frac{5\nu^{2}}{2}\right)
Lc_{3}^{2}+\frac{5Lc_{2}^{2}}{2},\\
c_{9}=&\;\max\{c_{6},c_{7},c_{8}\},\\
c_{10}=&\;\exp\left(\frac{6Tc_{9}}{\nu}\right)\cdot\left(\frac{5c_{5}}{4\mu}+\frac{3h_{0}^{2}Lc_{3}^{2}}{8}+\frac{1}{2}\right).
\end{align*}
\end{theorem}
\begin{proof}
Subtracting \eqref{eq10}, \eqref{eq12}, \eqref{eq14}, \eqref{eq16}, \eqref{eq16a} from \eqref{eq17}--\eqref{eq20a}, the error system is written as
\begin{align}
&\delta_{t}e_{i}^{\frac{1}{2}}-\mu\delta_{t}f_{i}^{\frac{1}{2}}+\gamma\psi(u^{0},e^{\frac{1}{2}})_{i}
-\frac{\gamma h^{2}}{2}[\psi(V^{0},U^{\frac{1}{2}})_{i}-\psi(v^{0},u^{\frac{1}{2}})_{i}]+\kappa\Delta_{x}e_{i}^{\frac{1}{2}}
\nonumber\\ & -\frac{\kappa h^{2}}{6}\Delta_{x}f_{i}^{\frac{1}{2}}-\nu f_{i}^{\frac{1}{2}}=Q_{i}^0,\quad 1\leqslant i\leqslant M,\label{eq48}\\
&\Delta_{t}e_{i}^{k}-\mu\Delta_{t}f_{i}^{k}+\gamma[\psi(U^{k},U^{\bar{k}})_{i}-\psi(u^{k},u^{\bar{k}})_{i}]
-\frac{\gamma h^{2}}{2}[\psi(V^{k},U^{\bar{k}})_{i}-\psi(v^{k},u^{\bar{k}})_{i}]\nonumber\\
& +\kappa\Delta_{x}e_{i}^{\bar{k}}
-\frac{\kappa h^{2}}{6}\Delta_{x}f_{i}^{\bar{k}}-\nu f_{i}^{\bar{k}}=Q_{i}^{k},\quad 1\leqslant i\leqslant M,\; 1\leqslant k\leqslant N-1,\label{eq49}\\
&f_{i}^{k}=\delta_{x}^{2}e_{i}^{k}-\frac{h^{2}}{12}\delta_{x}^{2}f_{i}^{k}+R_{i}^{k},\quad 1\leqslant i\leqslant M,\; 0\leqslant k\leqslant N,\label{eq50}\\
&e_{i}^{0}=0,\quad 1\leqslant i\leqslant M,\label{eq51}\\
&e_{i}^{k}=e_{i+M}^{k}, \quad f_{i}^{k}=f_{i+M}^{k},\quad 1\leqslant i\leqslant M,\; 0\leqslant k\leqslant N. \label{eq51a}
\end{align}
Denote
\begin{align}
F^{k}=\frac{1}{2}\left[(|e^{k}|_{1}^{2}+|e^{k-1}|_{1}^{2})+\frac{h^{2}}{12}(\|f^{k}\|^{2}+\|f^{k-1}\|^{2})
-\frac{h^{4}}{144}(|f^{k}|_{1}^{2}+|f^{k-1}|_{1}^{2})\right],\quad 1\leqslant k\leqslant N.\label{F}
\end{align}
From \eqref{eq3}, we have
\begin{align}
&|U^{k}|_{1}\leqslant\sqrt{L}c_{0},\quad \|U^{k}\|_{\infty}\leqslant c_{0},\quad 0\leqslant k\leqslant N,\label{ass1}\\
&\|V^{k}\|\leqslant\sqrt{L}c_{0},\quad \|V^{k}\|_{\infty}\leqslant c_{0},\quad
|V^{k}|_{1}\leqslant\sqrt{L}c_{0},\quad 0\leqslant k\leqslant N.\label{ass2}
\end{align}
Taking the inner product of \eqref{eq50} with $f^{k}$, we have
\begin{align*}
\|f^{k}\|^{2}=\;&(\delta_{x}^{2}e^{k},f^{k})-\frac{h^{2}}{12}(\delta_{x}^{2}f^{k},f^{k})+(R^{k},f^{k})\nonumber\\
\leqslant\;&\|\delta_{x}^{2}e^{k}\|\cdot\|f^{k}\|+\frac{h^{2}}{12}|f^{k}|_{1}^{2}+\|R^{k}\|\cdot\|f^{k}\|\nonumber\\
\leqslant\;&\frac{1}{6}\|f^{k}\|^{2}+\frac{3}{2}\|\delta_{x}^{2}e^{k}\|^{2}+\frac{1}{3}\|f^{k}\|^{2}
+\frac{1}{6}\|f^{k}\|^{2}+\frac{3}{2}\|R^{k}\|^{2}\nonumber\\
\leqslant\;&\frac{2}{3}\|f^{k}\|^{2}+\frac{6}{h^{2}}|e^{k}|_{1}^{2}+\frac{3}{2}\|R^{k}\|^{2},\quad 0\leqslant k\leqslant N.
\end{align*}
Thus, we have
\begin{align}
\|f^{k}\|^{2}\leqslant\frac{18}{h^{2}}|e^{k}|_{1}^{2}+\frac{9}{2}\|R^{k}\|^{2},\quad 0\leqslant k\leqslant N.\label{fk}
\end{align}
Taking the inner product of \eqref{eq48} with $\delta_{t}e^{\frac{1}{2}}$, we have
\begin{align}
&\|\delta_{t}e^{\frac{1}{2}}\|^{2}-\mu(\delta_{t}f^{\frac{1}{2}},\delta_{t}e^{\frac{1}{2}})
+\gamma(\psi(u^{0},e^{\frac{1}{2}}),\delta_{t}e^{\frac{1}{2}})-\frac{\gamma h^{2}}{2}(\psi(V^{0},U^{\frac{1}{2}})-\psi(v^{0},u^{\frac{1}{2}}),\delta_{t}e^{\frac{1}{2}})\nonumber\\
&+\kappa(\Delta_{x}e^{\frac{1}{2}},\delta_{t}e^{\frac{1}{2}})-\frac{\kappa h^{2}}{6}(\Delta_{x}f^{\frac{1}{2}},\delta_{t}e^{\frac{1}{2}})-\nu(f^{\frac{1}{2}},\delta_{t}e^{\frac{1}{2}})
=(Q^0,\delta_{t}e^{\frac{1}{2}}).\label{eq52}
\end{align}
From \eqref{eq51}, we have
\begin{align}
&\|\delta_{t}e^{\frac{1}{2}}\|^{2}=\frac{1}{\tau^{2}}\|e^{1}\|^{2},\label{eq53}\\
&(Q^0,\delta_{t}e^{\frac{1}{2}})=\frac{1}{\tau}(Q^0,e^{1}).\label{eq54}
\end{align}
Applying \eqref{lem4-1} in Lemma \ref{lemma4}, we have
\begin{align}
&(\delta_{t}f^{\frac{1}{2}},\delta_{t}e^{\frac{1}{2}})=\frac{1}{\tau^{2}}(f^{1},e^{1})-\frac{1}{\tau^{2}}(f^{0},e^{1})\nonumber\\
=\;&\frac{1}{\tau^{2}}\left[-|e^{1}|_{1}^{2}-\frac{h^{2}}{12}\|f^{1}\|^{2}+\frac{h^{4}}{144}|f^{1}|_{1}^{2}
+\frac{h^{2}}{12}(f^{1},R^{1})
+(R^{1},e^{1})\right]-\frac{1}{\tau^{2}}(f^{0},e^{1}).\label{eq55}
\end{align}
According to the definition of $\psi(u,v)_{i}$ and applying Lemma \ref{lemma2}, we have
\begin{align}
&(\psi(u^{0},e^{\frac{1}{2}}),\delta_{t}e^{\frac{1}{2}})=\frac{1}{2\tau}(\psi(u^{0},e^{1}),e^{1})=0.\label{eq56}
\end{align}
Moreover, combining \eqref{ass1} and Lemma \ref{lemma5}, we have
\begin{align}
&(\psi(V^{0},U^{\frac{1}{2}})-\psi(v^{0},u^{\frac{1}{2}}),\delta_{t}e^{\frac{1}{2}})\nonumber\\
=\;&(\psi(V^{0},e^{\frac{1}{2}})+\psi(f^{0},U^{\frac{1}{2}})-\psi(f^{0},e^{\frac{1}{2}}),\delta_{t}e^{\frac{1}{2}})
\nonumber\\
=\;&\frac{1}{2\tau}(\psi(V^{0},e^{1}),e^{1})+\frac{1}{\tau}(\psi(f^{0},U^{\frac{1}{2}}),e^{1})
-\frac{1}{2\tau}(\psi(f^{0},e^{1}),e^{1})\nonumber\\
=\;&\frac{1}{\tau}(\psi(f^{0},U^{\frac{1}{2}}),e^{1})\nonumber\\
=\;&\frac{h}{3\tau}\sum_{i=1}^{M}\left[f_i^0\Delta_xU_i^{\frac{1}{2}} + \Delta_x(f^0U^{\frac{1}{2}})_i\right] \cdot e_{i}^{1} \nonumber\\
=\; & \frac{h}{3\tau} \sum_{i=1}^M\left(2f_i^0 \cdot \Delta_xU_i^{\frac{1}{2}} +\frac{1}{2}U_{i-1}^{\frac{1}{2}}\delta_x f_{i-\frac{1}{2}}^0 + \frac{1}{2}U_{i+1}^{\frac{1}{2}}\delta_x f_{i+\frac{1}{2}}^0\right)\cdot e_i^1 \nonumber\\
\leqslant\;&\frac{c_{0}}{3\tau}\left(2+\frac{2}{h}\right)\|f^{0}\|\cdot\|e^{1}\|.\label{eq57}
\end{align}
Noticing \eqref{eq51} and applying Lemma \ref{lemma2}, we have
\begin{align}
&(\Delta_{x}e^{\frac{1}{2}},\delta_{t}e^{\frac{1}{2}})=\frac{1}{2\tau}(\Delta_{x}e^{1},e^{1})=0.\label{eq58}
\end{align}
Applying \eqref{lem4-1} and \eqref{lem4-3} in Lemma \ref{lemma4}, we have
\begin{align}
(\Delta_{x}f^{\frac{1}{2}},\delta_{t}e^{\frac{1}{2}})
=\;&\frac{1}{2\tau}(\Delta_{x}f^{1},e^{1})+\frac{1}{2\tau}(\Delta_{x}f^{0},e^{1})\nonumber\\
=\;&\frac{1}{2\tau}\left[\frac{h^{2}}{12}(\Delta_{x}f^{1},R^{1})+(\Delta_{x}R^{1},e^{1})\right]
+\frac{1}{2\tau}(\Delta_{x}f^{0},e^{1})\label{eq59}
\end{align}
and
\begin{align}
&(f^{\frac{1}{2}},\delta_{t}e^{\frac{1}{2}})\nonumber\\
=\;&\frac{1}{2\tau}(f^{1},e^{1})+\frac{1}{2\tau}(f^{0},e^{1})\nonumber\\
=\;&\frac{1}{2\tau}\left[-|e^{1}|_{1}^{2}-\frac{h^{2}}{12}\|f^{1}\|^{2}+\frac{h^{4}}{144}|f^{1}|_{1}^{2}
+\frac{h^{2}}{12}(f^{1},R^{1})
+(R^{1},e^{1})\right]+\frac{1}{2\tau}(f^{0},e^{1}).\label{eq60}
\end{align}
Substituting \eqref{eq53}--\eqref{eq60} into \eqref{eq52}, we have
\begin{align*}
\|e^{1}\|^{2}
\leqslant\;&\mu\left[-|e^{1}|_{1}^{2}-\frac{h^{2}}{12}\|f^{1}\|^{2}+\frac{h^{4}}{144}|f^{1}|_{1}^{2}
+\frac{h^{2}}{12}(f^{1},R^{1})+(R^{1},e^{1})-(f^{0},e^{1})\right]+\tau(Q^0,e^{1})\nonumber\\
&+\frac{\gamma h^{2}}{2}\cdot \frac{c_{0}\tau}{3}\left(2+\frac{2}{h}\right)\|f^{0}\|\cdot\|e^{1}\|
+\frac{\kappa h^{2}\tau}{6}\left[\frac{h^{2}}{24}(\Delta_{x}f^{1},R^{1})
+\frac{1}{2}(\Delta_{x}R^{1},e^{1})
+\frac{1}{2}(\Delta_{x}f^{0},e^{1})\right]\nonumber\\
&+\frac{\nu\tau}{2}\left[-|e^{1}|_{1}^{2}-\frac{h^{2}}{12}\|f^{1}\|^{2}+\frac{h^{4}}{144}|f^{1}|_{1}^{2}
+\frac{h^{2}}{12}(f^{1},R^{1})
+(R^{1},e^{1})+(f^{0},e^{1})\right]\nonumber\\
=\;&-\left(\mu+\frac{\nu\tau}{2}\right)|e^{1}|_{1}^{2}
-\left(\frac{\mu h^{2}}{12}+\frac{\nu h^{2}\tau}{24}\right)\|f^{1}\|^{2}
+\left(\frac{\mu h^{4}}{144}+\frac{\nu h^{4}\tau}{288}\right)|f^{1}|_{1}^{2}+\tau(Q^0,e^{1})\nonumber\\
&+\left(\frac{\mu h^{2}}{12}+\frac{\nu h^{2}\tau}{24}\right)\cdot(f^{1},R^{1})
+\left(\mu+\frac{\nu\tau}{2}\right)\cdot(R^{1},e^{1})+\left(\frac{\nu\tau}{2}-\mu\right)\cdot(f^{0},e^{1})\nonumber\\
&+\frac{\kappa h^{4}\tau}{144}\cdot(\Delta_{x}f^{1},R^{1})
+\frac{\kappa h^{2}\tau}{12}\cdot(\Delta_{x}R^{1},e^{1})+\frac{\kappa h^{2}\tau}{12}\cdot(\Delta_{x}f^{0},e^{1})
+\frac{\gamma c_{0}\tau h(h+1)}{3}\cdot\|f^{0}\|\cdot\|e^{1}\|\nonumber\\
\leqslant\;&-\left(\mu+\frac{\nu\tau}{2}\right)|e^{1}|_{1}^{2}-\left(\frac{\mu h^{2}}{12}+\frac{\nu h^{2}\tau}{24}\right)\|f^{1}\|^{2}
+\left(\frac{\mu h^{2}}{36}+\frac{\nu h^{2}\tau}{72}\right)\|f^{1}\|^{2}
+\frac{1}{6}\|e^{1}\|^{2}\nonumber\\
&+\frac{3\tau^{2}}{2}\|Q^0\|^{2}+\left(\frac{\mu h^{2}}{36}+\frac{\nu h^{2}\tau}{72}\right)\|f^{1}\|^{2}+\frac{3}{4}\left(\frac{\mu h^{2}}{12}+\frac{\nu h^{2}\tau}{24}\right)\|R^{1}\|^{2}+\frac{1}{6}\|e^{1}\|^{2}\nonumber\\
&+\frac{3}{2}\left(\mu+\frac{\nu\tau}{2}\right)^{2}\|R^{1}\|^{2}+\frac{1}{6}\|e^{1}\|^{2}
+\frac{3}{2}\left(\frac{\nu\tau}{2}-\mu\right)^{2}\|f^{0}\|^{2}+\frac{h^{2}}{4}\left(\frac{\mu h^{2}}{36}+\frac{\nu h^{2}\tau}{72}\right)|f^{1}|_{1}^{2}\nonumber\\
&+\left(\frac{\kappa h^{4}\tau}{144}\right)^{2}\cdot\frac{1}{h^{2}\left(\frac{\mu h^{2}}{36}+\frac{\nu h^{2}\tau}{72}\right)}\|R^{1}\|^{2}+\frac{1}{6}\|e^{1}\|^{2}
+\frac{\kappa^{2}h^{4}\tau^{2}}{96}\|\Delta_{x}R^{1}\|^{2}+\frac{1}{6}\|e^{1}\|^{2}\nonumber\\
&+\frac{\kappa^{2}h^{4}\tau^{2}}{96}\|\Delta_{x}f^{0}\|^{2}+\frac{1}{6}\|e^{1}\|^{2}
+\frac{\gamma^{2}c_{0}^{2}\tau^{2}h^2(h+1)^2}{6}\|f^{0}\|^{2}\nonumber\\
\leqslant\;&-\left(\mu+\frac{\nu\tau}{2}\right)|e^{1}|_{1}^{2}+\|e^{1}\|^{2}
+\left[\frac{3}{2}\left(\frac{\nu\tau}{2}-\mu\right)^{2}+\frac{\kappa^{2}h^{2}\tau^{2}}{24}
+\frac{\gamma^{2}c_{0}^{2}\tau^{2}h^2(h+1)^2}{6}\right]\|f^{0}\|^{2}\nonumber\\
&+\frac{3\tau^{2}}{2}\|Q^0\|^{2}+\left[\frac{3}{4}\left(\frac{\mu h^{2}}{12}+\frac{\nu h^{2}\tau}{24}\right)
+\frac{3}{2}\left(\mu+\frac{\nu\tau}{2}\right)^{2}
+\frac{\kappa^{2}h^{4}\tau^{2}}{288(2\mu+\nu\tau)}+\frac{\kappa^{2}h^{2}\tau^{2}}{24}\right]\|R^{1}\|^{2},
\end{align*}
when $h\leqslant h_{0}$, $\tau\leqslant \tau_{0}$, we have
\begin{align}
\|e^{1}\|^{2}\leqslant\;&-\left(\mu+\frac{\nu\tau}{2}\right)|e^{1}|_{1}^{2}+\|e^{1}\|^{2}
+\left[\frac{3}{2}\left(\frac{\nu\tau_{0}}{2}-\mu\right)^{2}+\frac{\kappa^{2}h_{0}^{2}\tau_{0}^{2}}{24}
+\frac{\gamma^{2}c_{0}^{2}\tau_{0}^{2}h_{0}^2(h_{0}+1)^2}{6}\right]\|f^{0}\|^{2}\nonumber\\
&+\frac{3\tau^{2}}{2}\|Q^0\|^{2}+\left[\frac{3}{4}\left(\frac{\mu h_{0}^{2}}{12}+\frac{\nu h_{0}^{2}\tau_{0}}{24}\right)
+\frac{3}{2}\left(\mu+\frac{\nu\tau_{0}}{2}\right)^{2}
+\frac{\kappa^{2}h_{0}^{4}\tau_{0}^{2}}{576\mu}+\frac{\kappa^{2}h_{0}^{2}\tau_{0}^{2}}{24}\right]\|R^{1}\|^{2}.\label{eq61}
\end{align}
Taking $k=0$ in \eqref{fk}, substituting the result into \eqref{eq61} and using \eqref{eq11}, \eqref{eq15}, we have
\begin{align*}
&\;\left(\mu+\frac{\nu\tau}{2}\right)|e^{1}|_{1}^{2}\nonumber\\
\leqslant\;&\frac{9}{2}\left[\frac{3}{2}\left(\frac{\nu\tau_{0}}{2}-\mu\right)^{2}
+\frac{\kappa^{2}h_{0}^{2}\tau_{0}^{2}}{24}+\frac{\gamma^{2}c_{0}^{2}\tau_{0}^{2}h_{0}^2(h_{0}+1)^2}{6}\right]\|R^{0}\|^{2}
+\frac{3\tau^{2}}{2}\|Q^0\|^{2}\nonumber\\
&\;+\left[\frac{3}{4}\left(\frac{\mu h_{0}^{2}}{12}+\frac{\nu h_{0}^{2}\tau_{0}}{24}\right)
+\frac{3}{2}\left(\mu+\frac{\nu\tau_{0}}{2}\right)^{2}
+\frac{\kappa^{2}h_{0}^{4}\tau_{0}^{2}}{576\mu}+\frac{\kappa^{2}h_{0}^{2}\tau_{0}^{2}}{24}\right]\|R^{1}\|^{2}\nonumber\\
\leqslant\;&\left[\frac{27}{4}\left(\frac{\nu\tau_{0}}{2}-\mu\right)^{2}
+\frac{3\kappa^{2}h_{0}^{2}\tau_{0}^{2}}{16}
+\frac{3\gamma^{2}c_{0}^{2}\tau_{0}^{2}h_{0}^2(h_{0}+1)^2}{4}+\frac{\mu h_{0}^{2}}{16}+\frac{\nu h_{0}^{2}\tau_{0}}{32}
+\frac{3}{2}\left(\mu+\frac{\nu\tau_{0}}{2}\right)^{2}\right.\nonumber\\
&\;\left.+\frac{\kappa^{2}h_{0}^{4}\tau_{0}^{2}}{576\mu}+\frac{\kappa^{2}h_{0}^{2}\tau_{0}^{2}}{24}\right]Lc_{3}^{2}h^{8}
+\frac{3\tau^{2}Lc_{1}^{2}}{2}(\tau+h^{4})^{2}\nonumber\\
\leqslant\;& c_{5}(\tau^{2}+h^{4})^{2}.
\end{align*}
Rearranging the above term, we have
\begin{align}
|e^{1}|_{1}^{2}\leqslant \frac{c_{5}}{\mu+\frac{\nu\tau}{2}}(\tau^{2}+h^{4})^{2}\leqslant \frac{c_{5}}{\mu}(\tau^{2}+h^{4})^{2}.\label{eq63}
\end{align}
Consequently
\begin{align}
|e^{1}|_{1}\leqslant c_{4}(\tau^{2}+h^{4}).\label{Err1}
\end{align}

From \eqref{F}, \eqref{fk} and \eqref{eq63}, we have
\begin{align}
F^{1}=\;&\frac{1}{2}\left[|e^{1}|_{1}^{2}+\frac{h^{2}}{12}(\|f^{1}\|^{2}+\|f^{0}\|^{2})
-\frac{h^{4}}{144}(|f^{1}|_{1}^{2}+|f^{0}|_{1}^{2})\right]\nonumber\\
\leqslant\;&\frac{1}{2}¡¢|e^{1}|_{1}^{2}+\frac{h^{2}}{24}(\|f^{1}\|^{2}+\|f^{0}\|^{2})\nonumber\\
\leqslant\;&\frac{1}{2}|e^{1}|_{1}^{2}+\frac{h^{2}}{24}
\left(\frac{18}{h^{2}}|e^{1}|_{1}^{2}+\frac{9}{2}\|R^{1}\|^{2}+\frac{9}{2}\|R^{0}\|^{2}\right)\nonumber\\
\leqslant\;&\left(\frac{5c_{5}}{4\mu}+\frac{3h^{2}Lc_{3}^{2}}{8}\right)(\tau^{2}+h^{4})^{2}.\label{eqF1}
\end{align}

Taking the inner product of \eqref{eq49} with $\Delta_{t}e^{k}$, we have
\begin{align}
&\|\Delta_{t}e^{k}\|^{2}-\mu(\Delta_{t}f^{k},\Delta_{t}e^{k})+\gamma(\psi(U^{k},U^{\bar{k}})
-\psi(u^{k},u^{\bar{k}}),\Delta_{t}e^{k})-\frac{\gamma h^{2}}{2}(\psi(V^{k},U^{\bar{k}})
-\psi(v^{k},u^{\bar{k}}),\Delta_{t}e^{k})\nonumber\\&+\kappa(\Delta_{x}e^{\bar{k}},\Delta_{t}e^{k})
-\frac{\kappa h^{2}}{6}(\Delta_{x}f^{\bar{k}},\Delta_{t}e^{k})
-\nu(f^{\bar{k}},\Delta_{t}e^{k})=(Q^{k},\Delta_{t}e^{k}),\quad 1\leqslant k\leqslant N-1.\label{eq66}
\end{align}

Now we suppose that $|e^{k}|_{1}\leqslant c_{4}(\tau^{2}+h^{4})$ holds for $k=1,2,\cdots,l$ with $1\leqslant l\leqslant N-1$. When $(\tau^{2}+h^{4})\leqslant 1/c_{4}$, using \eqref{ass1}--\eqref{fk}, we have
\begin{align}
&\|f^{k}\|\leqslant3\sqrt{\frac{2}{h^{2}}+\frac{Lc_{3}^{2}h^{8}}{2}},\quad 1\leqslant k\leqslant l,\label{ass3}\\
&|u^{k}|_{1}\leqslant |U^{k}|_{1}+|e^{k}|_{1}\leqslant\sqrt{L}c_{0}+1,\quad 1\leqslant k\leqslant l,\label{ass4}\\
&\|u^{k}\|_{\infty}\leqslant\frac{\sqrt{L}}{2}|u^{k}|_{1}\leqslant\frac{\sqrt{L}}{2}\left(\sqrt{L}c_{0}+1\right),
\quad 1\leqslant k\leqslant l,\label{ass5}\\
&|v^{k}|_{1}\leqslant|V^{k}|_{1}+|f^{k}|_{1}\leqslant\sqrt{L}c_{0}+\frac{2}{h}\|f^{k}\|\leqslant\sqrt{L}c_{0}+3\sqrt{\frac{8}{h^{4}}
+2Lc_{3}^{2}h^{6}},\quad 1\leqslant k\leqslant l,\label{ass6}\\
&\|v^{k}\|_{\infty}\leqslant\frac{\sqrt{L}}{2}|v^{k}|_{1}\leqslant\frac{Lc_{0}}{2}+\frac{3\sqrt{L}}{2}\sqrt{\frac{8}{h^{4}}
+2Lc_{3}^{2}h^{6}},\quad 1\leqslant k\leqslant l.\label{ass7}
\end{align}
Using \eqref{eq50} and applying \eqref{lem4-1} in Lemma \ref{lemma4}, we have
\begin{align}
\;&(\Delta_{t}f^{k},\Delta_{t}e^{k})\notag \\
=\;&-|\Delta_{t}e^{k}|_{1}^{2}-\frac{h^{2}}{12}\|\Delta_{t}f^{k}\|^{2}
+\frac{h^{4}}{144}|\Delta_{t}f^{k}|_{1}^{2}+\frac{h^{2}}{12}(\Delta_{t}f^{k},\Delta_{t}R^{k})
+(\Delta_{t}R^{k},\Delta_{t}e^{k})\nonumber\\
\leqslant\;&-|\Delta_{t}e^{k}|_{1}^{2}-\frac{h^{2}}{12}\|\Delta_{t}f^{k}\|^{2}
+\frac{h^{4}}{144}|\Delta_{t}f^{k}|_{1}^{2}+\frac{h^{2}}{12}\|\Delta_{t}f^{k}\|\cdot\|\Delta_{t}R^{k}\|
+\|\Delta_{t}R^{k}\|\cdot\|\Delta_{t}e^{k}\|,\notag\\
&\qquad\qquad\qquad\qquad\qquad\qquad 1\leqslant k\leqslant l.\label{eq67}
\end{align}
Due to
\begin{align*}
&\psi(U^{k},U^{\bar{k}})_{i}-\psi(u^{k},u^{\bar{k}})_{i}\nonumber\\
=&\psi(U^{k},U^{\bar{k}})_{i}-\psi(U^{k}-e^{k},U^{\bar{k}}-e^{\bar{k}})_{i}\nonumber\\
=&\psi(u^{k},e^{\bar{k}})_{i}+\psi(e^{k},U^{\bar{k}})_{i}\nonumber\\
=&\frac{1}{3}\left[u_{i}^{k}\Delta_{x}e_{i}^{\bar{k}}+\Delta_{x}(u^{k}e^{\bar{k}})_{i}\right]
+\frac{1}{3}\left[e_{i}^{k}\Delta_{x}U_{i}^{\bar{k}}+\Delta_{x}(e^{k}U^{\bar{k}})_{i}\right].
\end{align*}
Applying Lemma \ref{lemma5}, we have
\begin{align}
&\psi(U^{k},U^{\bar{k}})_{i}-\psi(u^{k},u^{\bar{k}})_{i}\nonumber\\
=&\frac{1}{3}\left[u_{i}^{k}\Delta_{x}e_{i}^{\bar{k}}
+\frac{1}{2}\left(\delta_{x}e_{i+\frac{1}{2}}^{\bar{k}}\right)u_{i+1}^{k}
+\frac{1}{2}\left(\delta_{x}e_{i-\frac{1}{2}}^{\bar{k}}\right)u_{i-1}^{k}+e_{i}^{\bar{k}}\Delta_{x}u_{i}^{k}\right]\nonumber\\
&+\frac{1}{3}\left[e_{i}^{k}\Delta_{x}U_{i}^{\bar{k}}+\frac{1}{2}\left(\delta_{x}e_{i+\frac{1}{2}}^{k}\right)U_{i+1}^{\bar{k}}
+\frac{1}{2}\left(\delta_{x}e_{i-\frac{1}{2}}^{k}\right)U_{i-1}^{\bar{k}}+e_{i}^{k}\Delta_{x}U_{i}^{\bar{k}}\right].\label{eq68}
\end{align}
Combining \eqref{ass1}, \eqref{ass4}, \eqref{ass5} with \eqref{eq68}, we have
\begin{align}
&-\left(\psi(U^{k},U^{\bar{k}})-\psi(u^{k},u^{\bar{k}}),\Delta_{t}e^{k}\right)\nonumber\\
=&-\frac{h}{3}\sum_{i=1}^{M-1}\left[u_{i}^{k}\Delta_{x}e_{i}^{\bar{k}}
+\frac{1}{2}\left(\delta_{x}e_{i+\frac{1}{2}}^{\bar{k}}\right)u_{i+1}^{k}
+\frac{1}{2}\left(\delta_{x}e_{i-\frac{1}{2}}^{\bar{k}}\right)u_{i-1}^{k}+e_{i}^{\bar{k}}\Delta_{x}u_{i}^{k}\right]
\cdot\Delta_{t}e_{i}^{k}\nonumber\\
&-\frac{h}{3}\sum_{i=1}^{M-1}\left[e_{i}^{k}\Delta_{x}U_{i}^{\bar{k}}+\frac{1}{2}\left(\delta_{x}e_{i+\frac{1}{2}}^{k}\right)
U_{i+1}^{\bar{k}}+\frac{1}{2}\left(\delta_{x}e_{i-\frac{1}{2}}^{k}\right)U_{i-1}^{\bar{k}}+e_{i}^{k}\Delta_{x}U_{i}^{\bar{k}}\right]
\cdot\Delta_{t}e_{i}^{k}\nonumber\\
\leqslant&\frac{1}{3}\left(\|u^{k}\|_{\infty}\cdot|e^{\bar{k}}|_{1}+\frac{1}{2}|e^{\bar{k}}|_{1}\cdot\|u^{k}\|_{\infty}
+\frac{1}{2}|e^{\bar{k}}|_{1}\cdot\|u^{k}\|_{\infty}+\|e^{\bar{k}}\|_{\infty}\cdot|u^{k}|_{1}\right)\cdot\|\Delta_{t}e^{k}\|\nonumber\\
&+\frac{1}{3}\left(\|e^{k}\|_{\infty}\cdot|U^{\bar{k}}|_{1}+\frac{1}{2}|e^{k}|_{1}\cdot\|U^{\bar{k}}\|_{\infty}
+\frac{1}{2}|e^{k}|_{1}\cdot\|U^{\bar{k}}\|_{\infty}+\|e^{k}\|_{\infty}\cdot|U^{\bar{k}}|_{1}\right)\cdot\|\Delta_{t}e^{k}\|\nonumber\\
\leqslant&\frac{1}{3}\left[\frac{\sqrt{L}}{2}\left(\sqrt{L}c_{0}+1\right)\cdot|e^{\bar{k}}|_{1}
+\frac{\sqrt{L}}{2}\left(\sqrt{L}c_{0}+1\right)\cdot|e^{\bar{k}}|_{1}+\left(\sqrt{L}c_{0}+1\right)\cdot
\frac{\sqrt{L}}{2}|e^{\bar{k}}|_{1}\right]\cdot\|\Delta_{t}e^{k}\|\nonumber\\
&+\frac{1}{3}\left[\sqrt{L}c_{0}\cdot\frac{\sqrt{L}}{2}|e^{k}|_{1}+c_{0}\cdot|e^{k}|_{1}+\sqrt{L}c_{0}
\cdot\frac{\sqrt{L}}{2}|e^{k}|_{1}\right]\cdot\|\Delta_{t}e^{k}\|\nonumber\\
=&\frac{Lc_{0}+\sqrt{L}}{2}\cdot|e^{\bar{k}}|_{1}\cdot\|\Delta_{t}e^{k}\|+\frac{Lc_{0}+c_{0}}{3}\cdot|e^{k}|_{1}
\cdot\|\Delta_{t}e^{k}\|,\quad 1\leqslant k\leqslant l.\label{eq69}
\end{align}
Similarly, it is concluded that
\begin{align*}
&\;\psi(V^{k},U^{\bar{k}})_{i}-\psi(v^{k},u^{\bar{k}})_{i}\nonumber\\
=&\;\psi(v^{k},e^{\bar{k}})_{i}+\psi(f^{k},U^{\bar{k}})_{i}\nonumber\\
=&\;\frac{1}{3}\left[v_{i}^{k}\Delta_{x}e_{i}^{\bar{k}}
+\frac{1}{2}\left(\delta_{x}e_{i+\frac{1}{2}}^{\bar{k}}\right)v_{i+1}^{k}
+\frac{1}{2}\left(\delta_{x}e_{i-\frac{1}{2}}^{\bar{k}}\right)v_{i-1}^{k}+e_{i}^{\bar{k}}\Delta_{x}v_{i}^{k}\right]\nonumber\\
&\;+\frac{1}{3}\left[f_{i}^{k}\Delta_{x}U_{i}^{\bar{k}}+\frac{1}{2}\left(\delta_{x}f_{i+\frac{1}{2}}^{k}\right)U_{i+1}^{\bar{k}}
+\frac{1}{2}\left(\delta_{x}f_{i-\frac{1}{2}}^{k}\right)U_{i-1}^{\bar{k}}+f_{i}^{k}\Delta_{x}U_{i}^{\bar{k}}\right].
\end{align*}
Combining \eqref{ass1}, \eqref{ass6} and \eqref{ass7}, we have
\begin{align}
&\left(\psi(V^{k},U^{\bar{k}})-\psi(v^{k},u^{\bar{k}}),\Delta_{t}e^{k}\right)\nonumber\\
=&\frac{h}{3}\sum_{i=1}^{M-1}\left[v_{i}^{k}\Delta_{x}e_{i}^{\bar{k}}
+\frac{1}{2}\left(\delta_{x}e_{i+\frac{1}{2}}^{\bar{k}}\right)v_{i+1}^{k}
+\frac{1}{2}\left(\delta_{x}e_{i-\frac{1}{2}}^{\bar{k}}\right)v_{i-1}^{k}+e_{i}^{\bar{k}}\Delta_{x}v_{i}^{k}\right]
\cdot\Delta_{t}e_{i}^{k}\nonumber\\
&+\frac{h}{3}\sum_{i=1}^{M-1}\left[f_{i}^{k}\Delta_{x}U_{i}^{\bar{k}}+\frac{1}{2}\left(\delta_{x}f_{i+\frac{1}{2}}^{k}\right)U_{i+1}^{\bar{k}}
+\frac{1}{2}\left(\delta_{x}f_{i-\frac{1}{2}}^{k}\right)U_{i-1}^{\bar{k}}+f_{i}^{k}\Delta_{x}U_{i}^{\bar{k}}\right]
\cdot\Delta_{t}e_{i}^{k}\nonumber\\
\leqslant&\frac{1}{3}\left(\|v^{k}\|_{\infty}\cdot|e^{\bar{k}}|_{1}+\frac{1}{2}|e^{\bar{k}}|_{1}\cdot\|v^{k}\|_{\infty}
+\frac{1}{2}|e^{\bar{k}}|_{1}\cdot\|v^{k}\|_{\infty}+\|e^{\bar{k}}\|_{\infty}\cdot|v^{k}|_{1}\right)\cdot\|\Delta_{t}e^{k}\|\nonumber\\
&+\frac{1}{3}\left(\|f^{k}\|\cdot\|\Delta_{x}U^{\bar{k}}\|_{\infty}+\frac{1}{2}|f^{k}|_{1}\cdot\|U^{\bar{k}}\|_{\infty}
+\frac{1}{2}|f^{k}|_{1}\cdot\|U^{\bar{k}}\|_{\infty}+\|f^{k}\|\cdot\|\Delta_{x}U^{\bar{k}}\|_{\infty}\right)\cdot\|\Delta_{t}e^{k}\|\nonumber\\
\leqslant&\frac{1}{3}\left[2\left(\frac{Lc_{0}}{2}+\frac{3\sqrt{L}}{2}\sqrt{\frac{8}{h^{4}}
+2Lc_{3}^{2}h^{6}}\right)
+\frac{\sqrt{L}}{2}\left(\sqrt{L}c_{0}+3\sqrt{\frac{8}{h^{4}}
+2Lc_{3}^{2}h^{6}}\right)\right]\cdot|e^{\bar{k}}|_{1}
\cdot\|\Delta_{t}e^{k}\|\nonumber\\
&+\frac{1}{3}\left(2c_{0}+\frac{2}{h}\cdot c_{0}\right)\cdot\|f^{k}\|\cdot\|\Delta_{t}e^{k}\|\nonumber\\
=&\left(\frac{Lc_{0}}{2}+\frac{3\sqrt{L}}{2}\sqrt{\frac{8}{h^{4}}+2Lc_{3}^{2}h^{6}}\right)\cdot|e^{\bar{k}}|_{1}
\cdot\|\Delta_{t}e^{k}\|+\frac{2}{3}\left(c_{0}+\frac{c_{0}}{h}\right)\cdot\|f^{k}\|\cdot\|\Delta_{t}e^{k}\|,\quad 1\leqslant k\leqslant l.\label{eq69*}
\end{align}
In addition, applying Cauchy-Schwarz inequality, we have
\begin{align}
-(\Delta_{x}e^{\bar{k}},\Delta_{t}e^{k})\leqslant\|\Delta_{x}e^{\bar{k}}\|\cdot\|\Delta_{t}e^{k}\|,\quad 1\leqslant k\leqslant l.\label{eq70}
\end{align}
Moreover, it holds
\begin{align}
&(\Delta_{x}f^{\bar{k}},\Delta_{t}e^{k})\nonumber\\
=\;&\left(\Delta_{x}\left(\delta_{x}^{2}e^{\bar{k}}-\frac{h^{2}}{12}\delta_{x}^{2}f^{\bar{k}}+R^{\bar{k}}\right),
\Delta_{t}e^{k}\right)\nonumber\\
=\;&(\Delta_{x}(\delta_{x}^{2}e^{\bar{k}}),\Delta_{t}e^{k})-\frac{h^{2}}{12}(\Delta_{x}(\delta_{x}^{2}f^{\bar{k}}),
\Delta_{t}e^{k})+(\Delta_{x}R^{\bar{k}},\Delta_{t}e^{k})\nonumber\\
=\;&-(\Delta_{x}(\delta_{x}e^{\bar{k}}),\Delta_{t}(\delta_{x}e^{k}))-\frac{h^{2}}{12}(\Delta_{x}f^{\bar{k}},
\Delta_{t}(\delta_{x}^{2}e^{k}))+(\Delta_{x}R^{\bar{k}},\Delta_{t}e^{k})\nonumber\\
\leqslant\;&|\Delta_{x}e^{\bar{k}}|_{1}\cdot|\Delta_{t}e^{k}|_{1}-\frac{h^{2}}{12}\left(\Delta_{x}f^{\bar{k}},
\Delta_{t}\left(f^{k}+\frac{h^{2}}{12}\delta_{x}^{2}f^{k}-R^{k}\right)\right)+|R^{\bar{k}}|_{1}
\cdot\|\Delta_{t}e^{k}\|\nonumber\\
=\;&|\Delta_{x}e^{\bar{k}}|_{1}\cdot|\Delta_{t}e^{k}|_{1}-\frac{h^{2}}{12}(\Delta_{x}f^{\bar{k}},
\Delta_{t}f^{k})+\frac{h^{4}}{144}(\delta_{x}(\Delta_{x}f^{\bar{k}}),\delta_{x}(\Delta_{t}f^{k}))
+\frac{h^{2}}{12}(\Delta_{x}f^{\bar{k}},\Delta_{t}R^{k})\nonumber\\
&+|R^{\bar{k}}|_{1}\cdot\|\Delta_{t}e^{k}\|\nonumber\\
\leqslant\;&|\Delta_{x}e^{\bar{k}}|_{1}\cdot|\Delta_{t}e^{k}|_{1}+\frac{h^{2}}{12}|f^{\bar{k}}|_{1}\cdot
\|\Delta_{t}f^{k}\|+\frac{h^{4}}{144}|\Delta_{x}f^{\bar{k}}|_{1}\cdot|\Delta_{t}f^{k}|_{1}
+\frac{h^{2}}{12}|f^{\bar{k}}|_{1}\cdot\|\Delta_{t}R^{k}\|\nonumber\\
&+|R^{\bar{k}}|_{1}\cdot\|\Delta_{t}e^{k}\|,\quad 1\leqslant k\leqslant l\label{eq71}
\end{align}
and
\begin{align}
&(f^{\bar{k}},\Delta_{t}e^{k})\nonumber\\
=\;&\left(\delta_{x}^{2}e^{\bar{k}}-\frac{h^{2}}{12}\delta_{x}^{2}f^{\bar{k}}+R^{\bar{k}},\Delta_{t}e^{k}\right)
\nonumber\\
=\;&(\delta_{x}^{2}e^{\bar{k}},\Delta_{t}e^{k})-\frac{h^{2}}{12}(\delta_{x}^{2}f^{\bar{k}},\Delta_{t}e^{k})
+(R^{\bar{k}},\Delta_{t}e^{k})\nonumber\\
=\;&-(\delta_{x}e^{\bar{k}},\Delta_{t}(\delta_{x}e^{k}))-\frac{h^{2}}{12}(f^{\bar{k}},\Delta_{t}(\delta_{x}^{2}e^{k}))
+(R^{\bar{k}},\Delta_{t}e^{k})\nonumber\\
=\;&-\frac{1}{4\tau}(|e^{k+1}|_{1}^{2}-|e^{k-1}|_{1}^{2})-\frac{h^{2}}{12}\left(f^{\bar{k}},
\Delta_{t}\left(f^{k}+\frac{h^{2}}{12}\delta_{x}^{2}f^{k}-R^{k}\right)\right)
+(R^{\bar{k}},\Delta_{t}e^{k})\nonumber\\
=\;&-\frac{1}{4\tau}(|e^{k+1}|_{1}^{2}-|e^{k-1}|_{1}^{2})-\frac{h^{2}}{12}(f^{\bar{k}},
\Delta_{t}f^{k})-\frac{h^{4}}{144}(f^{\bar{k}},\Delta_{t}(\delta_{x}^{2}f^{k}))
+\frac{h^{2}}{12}(f^{\bar{k}},\Delta_{t}R^{k})\nonumber\\
&\;+(R^{\bar{k}},\Delta_{t}e^{k})\nonumber\\
=\;&-\frac{1}{4\tau}(|e^{k+1}|_{1}^{2}-|e^{k-1}|_{1}^{2})-\frac{h^{2}}{12}\cdot\frac{1}{4\tau}(\|f^{k+1}\|^{2}
-\|f^{k-1}\|^{2})+\frac{h^{4}}{144}(\delta_{x}f^{\bar{k}},\Delta_{t}(\delta_{x}f^{k}))\nonumber\\
&\;+\frac{h^{2}}{12}(f^{\bar{k}},\Delta_{t}R^{k})+(R^{\bar{k}},\Delta_{t}e^{k})\nonumber\\
\leqslant\;&-\frac{1}{4\tau}(|e^{k+1}|_{1}^{2}-|e^{k-1}|_{1}^{2})-\frac{h^{2}}{12}\cdot\frac{1}{4\tau}(\|f^{k+1}\|^{2}
-\|f^{k-1}\|^{2})+\frac{h^{4}}{144}\cdot\frac{1}{4\tau}(|f^{k+1}|_{1}^{2}-|f^{k-1}|_{1}^{2})\nonumber\\
&\;+\frac{h^{2}}{12}\|f^{\bar{k}}\|\cdot\|\Delta_{t}R^{k}\|+\|R^{\bar{k}}\|\cdot\|\Delta_{t}e^{k}\|,\quad 1\leqslant k\leqslant l.\label{eq72}
\end{align}
Substituting \eqref{eq67}, \eqref{eq69}--\eqref{eq72} into \eqref{eq66}, we have
\begin{align*}
&\;\|\Delta_{t}e^{k}\|^{2}\nonumber\\
\leqslant&\;\mu\left[-|\Delta_{t}e^{k}|_{1}^{2}-\frac{h^{2}}{12}\|\Delta_{t}f^{k}\|^{2}
+\frac{h^{4}}{144}|\Delta_{t}f^{k}|_{1}^{2}+\frac{h^{2}}{12}\|\Delta_{t}f^{k}\|
\cdot\|\Delta_{t}R^{k}\|+\|\Delta_{t}R^{k}\|\cdot\|\Delta_{t}e^{k}\|\right]\nonumber\\
&\;+\gamma\left[\frac{Lc_{0}+\sqrt{L}}{2}\cdot|e^{\bar{k}}|_{1}\cdot\|\Delta_{t}e^{k}\|+\frac{Lc_{0}+c_{0}}{3}\cdot|e^{k}|_{1}
\cdot\|\Delta_{t}e^{k}\|\right]
+\kappa\|\Delta_{x}e^{\bar{k}}\|\cdot\|\Delta_{t}e^{k}\|\nonumber\\
&\;+\frac{\gamma h^{2}}{2}\left[\left(\frac{Lc_{0}}{2}+\frac{3\sqrt{L}}{2}\sqrt{\frac{8}{h^{4}}
+2Lc_{3}^{2}h^{6}}\right)\cdot|e^{\bar{k}}|_{1}\cdot\|\Delta_{t}e^{k}\|+\frac{2}{3}\left(c_{0}+\frac{c_{0}}{h}\right)\cdot\|f^{k}\|\cdot\|\Delta_{t}e^{k}\|\right]\nonumber\\
&\;+\frac{\kappa h^{2}}{6}
\left[|\Delta_{x}e^{\bar{k}}|_{1}\cdot|\Delta_{t}e^{k}|_{1}
+\frac{h^{2}}{12}|f^{\bar{k}}|_{1}\cdot\|\Delta_{t}f^{k}\|+\frac{h^{4}}{144}|\Delta_{x}f^{\bar{k}}|_{1}\cdot|\Delta_{t}f^{k}|_{1}
+\frac{h^{2}}{12}|f^{\bar{k}}|_{1}\cdot\|\Delta_{t}R^{k}\|\right.\nonumber\\
&\;\left.+|R^{\bar{k}}|_{1}\cdot\|\Delta_{t}e^{k}\|\right]+\nu\left[-\frac{1}{4\tau}(|e^{k+1}|_{1}^{2}-|e^{k-1}|_{1}^{2})
-\frac{h^{2}}{12}\cdot\frac{1}{4\tau}(\|f^{k+1}\|^{2}-\|f^{k-1}\|^{2})\right.\nonumber\\
&\;\left.+\frac{h^{4}}{144}\cdot\frac{1}{4\tau}(|f^{k+1}|_{1}^{2}-|f^{k-1}|_{1}^{2})
+\frac{h^{2}}{12}\|f^{\bar{k}}\|\cdot\|\Delta_{t}R^{k}\|+\|R^{\bar{k}}\|
\cdot\|\Delta_{t}e^{k}\|\right]+\|Q^{k}\|\cdot\|\Delta_{t}e^{k}\|\nonumber\\
\leqslant\;&-\frac{\mu h^{2}}{12}\|\Delta_{t}f^{k}\|^{2}+\frac{\mu h^{2}}{36}
\|\Delta_{t}f^{k}\|^{2}+\frac{\mu h^{2}}{54}\|\Delta_{t}f^{k}\|^{2}
+\frac{3\mu h^{2}}{32}\|\Delta_{t}R^{k}\|^{2}+\frac{1}{10}
\|\Delta_{t}e^{k}\|^{2}+\frac{5\mu^{2}}{2}\|\Delta_{t}R^{k}\|^{2}\nonumber\\
&\;+\frac{1}{10}\|\Delta_{t}e^{k}\|^{2}+\frac{5\gamma^{2}(Lc_{0}+\sqrt{L})^{2}}{8}
|e^{\bar{k}}|_{1}^{2}+\frac{1}{10}\|\Delta_{t}e^{k}\|^{2}+\frac{5\gamma^{2}(Lc_{0}+c_{0})^{2}}{18}
|e^{k}|_{1}^{2}+\frac{1}{10}\|\Delta_{t}e^{k}\|^{2}\nonumber\\
&\;+\frac{5\kappa^{2}}{2}\|\Delta_{x}e^{\bar{k}}\|^{2}+\frac{1}{10}\|\Delta_{t}e^{k}\|^{2}
+\frac{5\gamma^{2}h^{4}}{8}\left(\frac{Lc_{0}}{2}+\frac{3\sqrt{L}}{2}\sqrt{\frac{8}{h^{4}}
+2Lc_{3}^{2}h^{6}}\right)^{2}|e^{\bar{k}}|_{1}^{2}+\frac{1}{10}\|\Delta_{t}e^{k}\|^{2}\nonumber\\
&\;+\frac{5\gamma^{2}h^{2}(hc_{0}+c_{0})^{2}}{18}\|f^{k}\|^{2}
+\frac{1}{10}\cdot\frac{h^{2}}{4}|\Delta_{t}e^{k}|_{1}^{2}
+\frac{5\kappa^{2}h^{2}}{18}|\Delta_{x}e^{\bar{k}}|^{2}
+\frac{\mu h^{2}}{54}\|\Delta_{t}f^{k}\|^{2}+\frac{\kappa^{2}h^{6}}{384\mu}|f^{\bar{k}}|_{1}^{2}\nonumber\\
&\;+\frac{\mu h^{2}}{54}\cdot\frac{h^{2}}{4}|\Delta_{t}f^{k}|_{1}^{2}
+\left(\frac{\kappa h^{6}}{864}\right)^{2}\cdot\frac{54}{\mu h^{4}}
|\Delta_{x}f^{\bar{k}}|_{1}^{2}+\left(\frac{\kappa h^{4}}{72}\right)^{2}|f^{\bar{k}}|^{2}+\frac{1}{4}\|\Delta_{t}R^{k}\|^{2}
+\frac{1}{10}\|\Delta_{t}e^{k}\|^{2}\nonumber\\
&\;+\frac{5\kappa^{2}h^{4}}{72}|R^{\bar{k}}|_{1}^{2}
-\frac{\nu}{4\tau}\left[(|e^{k+1}|_{1}^{2}-|e^{k-1}|_{1}^{2})
+\frac{h^{2}}{12}(\|f^{k+1}\|^{2}-\|f^{k-1}\|^{2})
-\frac{ h^{4}}{144}(|f^{k+1}|_{1}^{2}-|f^{k-1}|_{1}^{2})\right]\nonumber\\
&\;+\frac{\nu^{2}h^{4}}{144}\|f^{\bar{k}}\|^{2}+\frac{1}{4}\|\Delta_{t}R^{k}\|^{2}
+\frac{1}{10}\|\Delta_{t}e^{k}\|^{2}+\frac{5\nu^{2}}{2}\|R^{\bar{k}}\|^{2}
+\frac{1}{10}\|\Delta_{t}e^{k}\|^{2}+\frac{5}{2}\|Q^{k}\|^{2}\nonumber\\
\leqslant\;&\|\Delta_{t}e^{k}\|^{2}-\frac{\nu}{4\tau}\left[(|e^{k+1}|_{1}^{2}-|e^{k-1}|_{1}^{2})
+\frac{h^{2}}{12}(\|f^{k+1}\|^{2}-\|f^{k-1}\|^{2})
-\frac{h^{4}}{144}(|f^{k+1}|_{1}^{2}-|f^{k-1}|_{1}^{2})\right]\nonumber\\
&+\left[\frac{5\gamma^{2}(Lc_{0}+\sqrt{L})^{2}}{8}
+\frac{5\gamma^{2}}{8}\left(\frac{Lc_{0}h^{2}}{2}+\frac{3\sqrt{L}}{2}\sqrt{8
+2Lc_{3}^{2}h^{10}}\right)^{2}+\frac{5\kappa^{2}}{2}+\frac{10\kappa^{2}}{9}\right]|e^{\bar{k}}|_{1}^{2}\nonumber\\
&\;+\frac{5\gamma^{2}(Lc_{0}+c_{0})^{2}}{18}|e^{k}|_{1}^{2}
+\frac{5\gamma^{2}h^{2}(hc_{0}+c_{0})^{2}}{18}\|f^{k}\|^{2}+\left(\frac{\kappa^{2}h^{4}}{96\mu}
+\frac{\kappa^{2}h^{4}}{864\mu}+\frac{\kappa^{2}h^{6}}{1296}+\frac{\nu^{2}h^{4}}{144}\right)\|f^{\bar{k}}\|^{2}\nonumber\\
&\;+\left(\frac{3\mu h^{2}}{32}+\frac{5\mu^{2}}{2}+\frac{1}{2}\right)\|\Delta_{t}R^{k}\|^{2}
+\left(\frac{5\kappa^{2}h^{2}}{18}+\frac{5\nu^{2}}{2}\right)\|R^{\bar{k}}\|^{2}+\frac{5}{2}\|Q^{k}\|^{2}.
\end{align*}
Simplifying the formula, we have
\begin{align*}
&\frac{\nu}{2\tau}\left[\frac{1}{2}(|e^{k+1}|_{1}^{2}+|e^{k}|_{1}^{2})
+\frac{h^{2}}{24}(\|f^{k+1}\|^{2}+\|f^{k}\|^{2})
-\frac{h^{4}}{288}(|f^{k+1}|_{1}^{2}+|f^{k}|_{1}^{2})\right]\\
&-\frac{\nu}{2\tau}\left[\frac{1}{2}(|e^{k}|_{1}^{2}+|e^{k-1}|_{1}^{2})
+\frac{h^{2}}{24}(\|f^{k}\|^{2}+\|f^{k-1}\|^{2})
-\frac{h^{4}}{288}(|f^{k}|_{1}^{2}+|f^{k-1}|_{1}^{2})\right]\\
\leqslant\;&\left[\frac{5\gamma^{2}(Lc_{0}+\sqrt{L})^{2}}{8}
+\frac{5\gamma^{2}}{8}\left(\frac{Lc_{0}h^{2}}{2}+\frac{3\sqrt{L}}{2}\sqrt{8
+2Lc_{3}^{2}h^{10}}\right)^{2}+\frac{5\kappa^{2}}{2}+\frac{10\kappa^{2}}{9}\right]|e^{\bar{k}}|_{1}^{2}\\
&\;+\frac{5\gamma^{2}(Lc_{0}+c_{0})^{2}}{18}|e^{k}|_{1}^{2}
+\frac{5\gamma^{2}h^{2}(hc_{0}+c_{0})^{2}}{18}\|f^{k}\|^{2}+\left(\frac{\kappa^{2}h^{4}}{96\mu}
+\frac{\kappa^{2}h^{4}}{864\mu}+\frac{\kappa^{2}h^{6}}{1296}+\frac{\nu^{2}h^{4}}{144}\right)\|f^{\bar{k}}\|^{2}\\
&\;+\left(\frac{3\mu h^{2}}{32}+\frac{5\mu^{2}}{2}+\frac{1}{2}\right)\|\Delta_{t}R^{k}\|^{2}
+\left(\frac{5\kappa^{2}h^{2}}{18}+\frac{5\nu^{2}}{2}\right)\|R^{\bar{k}}\|^{2}+\frac{5}{2}\|Q^{k}\|^{2}\\
\leqslant\;&\left[\frac{5\gamma^{2}(Lc_{0}+\sqrt{L})^{2}}{8}
+\frac{5\gamma^{2}}{8}\left(\frac{Lc_{0}h^{2}}{2}+\frac{3\sqrt{L}}{2}\sqrt{8
+2Lc_{3}^{2}h^{10}}\right)^{2}+\frac{5\kappa^{2}}{2}+\frac{10\kappa^{2}}{9}\right]\frac{|e^{k+1}|_{1}^{2}+|e^{k-1}|_{1}^{2}}{2}\\
&\;+\frac{5\gamma^{2}(Lc_{0}+c_{0})^{2}}{18}|e^{k}|_{1}^{2}
+\frac{5\gamma^{2}h^{2}(hc_{0}+c_{0})^{2}}{18}\|f^{k}\|^{2}+\left(\frac{\kappa^{2}h^{4}}{96\mu}
+\frac{\kappa^{2}h^{4}}{864\mu}+\frac{\kappa^{2}h^{6}}{1296}+\frac{\nu^{2}h^{4}}{144}\right)\\
&\;\times\frac{\|f^{k+1}\|^{2}+\|f^{k-1}\|^{2}}{2}+\left(\frac{3\mu h^{2}}{32}+\frac{5\mu^{2}}{2}+\frac{1}{2}\right)\|\Delta_{t}R^{k}\|^{2}
+\left(\frac{5\kappa^{2}h^{2}}{18}+\frac{5\nu^{2}}{2}\right)\|R^{\bar{k}}\|^{2}+\frac{5}{2}\|Q^{k}\|^{2}\\
\leqslant\;&\left[\frac{5\gamma^{2}(Lc_{0}+\sqrt{L})^{2}}{8}
+\frac{5\gamma^{2}}{8}\left(\frac{Lc_{0}h^{2}}{2}+\frac{3\sqrt{L}}{2}\sqrt{8
+2Lc_{3}^{2}h^{10}}\right)^{2}+\frac{5\kappa^{2}}{2}+\frac{10\kappa^{2}}{9}+\frac{5\gamma^{2}(Lc_{0}+c_{0})^{2}}{18}\right]\\
&\;\times\left(\frac{|e^{k+1}|_{1}^{2}+|e^{k}|_{1}^{2}}{2}+\frac{|e^{k+1}|_{1}^{2}+|e^{k}|_{1}^{2}}{2}\right)
+\left[5\gamma^{2}(hc_{0}+c_{0})^{2}+\frac{3\kappa^{2}h^{2}}{16\mu}
+\frac{\kappa^{2}h^{2}}{48\mu}+\frac{\kappa^{2}h^{4}}{72}+\frac{\nu^{2}h^{2}}{8}\right]\\
&\;\times\left[\frac{h^{2}}{36}\left(\|f^{k+1}\|^{2}+\|f^{k}\|^{2}\right)
+\frac{h^{2}}{36}\left(\|f^{k}\|^{2}+\|f^{k-1}\|^{2}\right)\right]+\left[\left(\frac{3\mu h^{2}}{32}+\frac{5\mu^{2}}{2}+\frac{1}{2}+\frac{5\kappa^{2}h^{2}}{18}+\frac{5\nu^{2}}{2}\right)\right.\\
&\;\left.\times Lc_{3}^{2}+\frac{5Lc_{2}^{2}}{2}\right](\tau^{2}+h^{4})^{2},
\end{align*}
when $h\leqslant h_{0}$, $\tau\leqslant \tau_{0}$, we have
\begin{align}
&\frac{\nu}{2\tau}\left[\frac{1}{2}(|e^{k+1}|_{1}^{2}+|e^{k}|_{1}^{2})
+\frac{h^{2}}{24}(\|f^{k+1}\|^{2}+\|f^{k}\|^{2})
-\frac{h^{4}}{288}(|f^{k+1}|_{1}^{2}+|f^{k}|_{1}^{2})\right]\nonumber\\
&-\frac{\nu}{2\tau}\left[\frac{1}{2}(|e^{k}|_{1}^{2}+|e^{k-1}|_{1}^{2})
+\frac{h^{2}}{24}(\|f^{k}\|^{2}+\|f^{k-1}\|^{2})
-\frac{h^{4}}{288}(|f^{k}|_{1}^{2}+|f^{k-1}|_{1}^{2})\right]\nonumber\\
\leqslant\;&c_{6}\left(\frac{|e^{k+1}|_{1}^{2}+|e^{k}|_{1}^{2}}{2}+\frac{|e^{k}|_{1}^{2}+|e^{k-1}|_{1}^{2}}{2}\right)
+c_{7}\left[\frac{h^{2}}{36}(\|f^{k+1}\|^{2}+\|f^{k}\|^{2})+\frac{h^{2}}{36}(\|f^{k}\|^{2}+\|f^{k-1}\|^{2})\right]\nonumber\\
&+c_{8}(\tau^{2}+h^{4})^{2}\nonumber\\
\leqslant\;&c_{9}\left[\frac{|e^{k+1}|_{1}^{2}+|e^{k}|_{1}^{2}}{2}
+\frac{h^{2}}{36}(\|f^{k+1}\|^{2}+\|f^{k}\|^{2})\right]+c_{9}\left[\frac{|e^{k}|_{1}^{2}+|e^{k-1}|_{1}^{2}}{2}
+\frac{h^{2}}{36}(\|f^{k}\|^{2}+\|f^{k-1}\|^{2})\right]\nonumber\\
&+c_{9}(\tau^{2}+h^{4})^{2},\quad 1\leqslant k\leqslant l.\label{eq74}
\end{align}
Thanks to
\begin{align*}
\frac{h^{2}}{24}(\|f^{k+1}\|^{2}+\|f^{k}\|^{2})-\frac{h^{4}}{288}(|f^{k+1}|_{1}^{2}+|f^{k}|_{1}^{2})
\geqslant\;&\left(\frac{h^{2}}{24}-\frac{4}{h^{2}}\cdot\frac{h^{4}}{288}\right)(\|f^{k+1}\|^{2}+\|f^{k}\|^{2})\\
=\;&\frac{h^{2}}{36}(\|f^{k+1}\|^{2}+\|f^{k}\|^{2}),
\end{align*}
we have
\begin{align}
\frac{1}{2}(|e^{k}|_{1}^{2}+|e^{k-1}|_{1}^{2})+\frac{h^{2}}{36}(\|f^{k}\|^{2}+\|f^{k-1}\|^{2})\leqslant F^{k},\quad 1\leqslant k\leqslant N.\label{eq75}
\end{align}
Combining \eqref{eq74} with \eqref{eq75}, we have
\begin{align*}
\frac{\nu}{2\tau}(F^{k+1}-F^{k})\leqslant c_{9}(F^{k}+F^{k+1})+c_{9}(\tau^{2}+h^{4})^{2},\quad 1\leqslant k\leqslant l.
\end{align*}
According to the Gronwall inequality, when $2c_{9}\tau/\nu\leqslant 1/3$, we have
\begin{align*}
F^{k+1}\leqslant \exp\left(\frac{6Tc_{9}}{\nu}\right)\cdot\left[F^{1}+\frac{1}{2}(\tau^{2}+h^{4})^{2}\right],
\quad 1\leqslant k\leqslant l.
\end{align*}
From \eqref{eqF1}, when $h\leqslant h_{0}$, we have
\begin{align}
F^{k+1}\leqslant c_{10}(\tau^{2}+h^{4})^{2},\quad 1\leqslant k\leqslant l.\label{eq79}
\end{align}
A combination of \eqref{F}, \eqref{eq75} and \eqref{eq79}, we have
\begin{align*}
|e^{k+1}|_{1}\leqslant\sqrt{2F^{k+1}}\leqslant \sqrt{2c_{10}}(\tau^{2}+h^{4})\leqslant c_{4}(\tau^{2}+h^{4}),\quad 1\leqslant k\leqslant l.
\end{align*}
By the mathematical induction, we have
\begin{align}
|e^{k+1}|_{1}\leqslant c_{4}(\tau^{2}+h^{4}),\quad 1\leqslant k\leqslant N-1.\label{eq81}
\end{align}
Combining \eqref{Err1} with \eqref{eq81}, we have
\begin{align*}
|e^{k}|_{1}\leqslant c_{4}(\tau^{2}+h^{4}),\quad 0\leqslant k\leqslant N.
\end{align*}
This completes the proof.
\end{proof}
\begin{remark}
  \begin{align*}
     \|e^{k}\|_{\infty} \leqslant \frac{\sqrt{L}}{2} |e^{k}|_{1}\leqslant\frac{c_{4}\sqrt{L}}{2}(\tau^{2}+h^{4}),\quad 0\leqslant k\leqslant N.
  \end{align*}
\end{remark}
\subsection{Stability}
In the below, we will discuss the stability of the difference scheme \eqref{eq17}--\eqref{eq20a}.
 \begin{theorem}[Stability]
 Suppose $\{u_i^k, \ v_i^k\,|\, 1 \leqslant i \leqslant M,\; 0 \leqslant k \leqslant N\}$ is the solution of \eqref{eq17}--\eqref{eq20a}, $\{\hat{u}_i^k, \ \hat{v}_i^k\,|\, 1 \leqslant i \leqslant M,\; 0 \leqslant k \leqslant N\}$ is the solution of
\begin{align}
&\delta_{t}\hat{u}_{i}^{\frac{1}{2}}-\mu\delta_{t}\hat{v}_{i}^{\frac{1}{2}}+\gamma\left[\psi(\hat{u}^{0},\hat{u}^{\frac{1}{2}})_{i}
-\frac{h^{2}}{2}\psi(\hat{v}^{0},\hat{u}^{\frac{1}{2}})_{i}\right]+\kappa\left(\Delta_{x}\hat{u}_{i}^{\frac{1}{2}}
-\frac{h^{2}}{6}\Delta_{x}\hat{v}_{i}^{\frac{1}{2}}\right)-\nu \hat{v}_{i}^{\frac{1}{2}}=0,\nonumber\\
&\quad \qquad\qquad\qquad\qquad\qquad\qquad 1\leqslant i\leqslant M,\label{eq17bb}\\
&\Delta_{t}\hat{u}_{i}^{k}-\mu\Delta_{t}\hat{v}_{i}^{k}+\gamma\left[\psi(\hat{u}^{k},\hat{u}^{\bar{k}})_{i}
-\frac{h^{2}}{2}\psi(\hat{v}^{k},\hat{u}^{\bar{k}})_{i}\right]+\kappa\left(\Delta_{x}\hat{u}_{i}^{\bar{k}}
-\frac{h^{2}}{6}\Delta_{x}\hat{v}_{i}^{\bar{k}}\right)-\nu \hat{v}_{i}^{\bar{k}}=0,\nonumber\\ &\qquad\qquad\qquad\qquad\qquad\qquad\quad 1\leqslant i\leqslant M,\; 1\leqslant k\leqslant N-1,\label{eq18bb}\\
&\hat{v}_{i}^{k}=\delta_{x}^{2}\hat{u}_{i}^{k}-\frac{h^{2}}{12}\delta_{x}^{2}\hat{v}_{i}^{k},\quad 1\leqslant i\leqslant M,\; 0\leqslant k\leqslant N,\label{eq19bb}\\
&\hat{u}_{i}^{0}=\varphi(x_{i})+\phi^0(x_i),\quad 1\leqslant i\leqslant M,\label{eq20bb}\\
&\hat{u}_i^k = \hat{u}_{i+M}^k,\quad \hat{v}_i^k = \hat{v}_{i+M}^k,\quad 1\leqslant i\leqslant M,\; 0\leqslant k\leqslant N.\label{eq21bb}
\end{align}
Denote
\begin{align*}
\eta_i^k= \hat{u}_i^k-u_i^k,\quad \xi_i^k = \hat{v}_i^k-v_i^k,\quad 1\leqslant i\leqslant M,\;0\leqslant k\leqslant N,
\end{align*}
 then there exist positive constants $h_{0}$, $\tau_{0}$, such that when $h\leqslant h_{0}$, $\tau\leqslant\tau_{0}$,
 $c_{12}\tau\leqslant 1/4$,
 we have
 \begin{align*}
 |\eta^{k}|_{1} \leqslant c_{11}|\phi^0|_1,\quad 0\leqslant k\leqslant N,
 \end{align*}
 where
 \begin{align*}
 c_{11}=\max\left\{\sqrt{\frac{8c_{13}}{3}},\sqrt{2c_{17}}\right\},
 \end{align*}
with
\begin{align*}
c_{12}=&\;\frac{15\gamma^{2} L}{4\nu}(|\varphi|_{1}+|\phi^{0}|_{1})
^{2}+\frac{13\kappa^{2}}{6\nu}+\frac{5\kappa^{2}h_{0}^{2}}{36\mu\nu},\\
c_{13}=&\;\frac{7}{4}+\frac{15\gamma^{2}\tau_{0} L}{4\nu}(1+\sqrt{L}c_{0})^{2},\\
c_{14}=&\;\frac{15\gamma^{2}L}{4}(2+\sqrt{L}c_{0})^{2}+\frac{13\kappa^{2}}{6}
+\frac{3\gamma^{2}L}{8}(1+\sqrt{L}c_{0})^{2},\\
c_{15}=&\;\frac{27\gamma^{2}L}{4}(1+\sqrt{L}c_{0})^{2}+\frac{5\kappa^{2}h_{0}^{2}}{36\mu},\\
c_{16}=&\;\max\{c_{14},c_{15}\},\\
c_{17}=&\;\exp\left(\frac{6Tc_{16}}{\nu}\right)\cdot\frac{4c_{13}}{3}.
\end{align*}
  \end{theorem}
 Subtracting \eqref{eq17}--\eqref{eq20a} from \eqref{eq17bb}--\eqref{eq21bb}, we have
 \begin{align}
&\delta_{t}\eta_{i}^{\frac{1}{2}}-\mu\delta_{t}\xi_{i}^{\frac{1}{2}}+\gamma\left[\psi(\hat{u}^{0},\hat{u}^{\frac{1}{2}})_{i}
-\frac{h^{2}}{2}\psi(\hat{v}^{0},\hat{u}^{\frac{1}{2}})_{i}\right]-\gamma\left[\psi(u^{0},u^{\frac{1}{2}})_{i}
-\frac{h^{2}}{2}\psi(v^{0},u^{\frac{1}{2}})_{i}\right]\nonumber\\
&  +\kappa\left(\Delta_{x}\eta_{i}^{\frac{1}{2}}
-\frac{h^{2}}{6}\Delta_{x}\xi_{i}^{\frac{1}{2}}\right)-\nu \xi_{i}^{\frac{1}{2}}=0,\quad 1\leqslant i\leqslant M,\label{eq17bbb}\\
&\Delta_{t}\eta_{i}^{k}-\mu\Delta_{t}\xi_{i}^{k}+\gamma\left[\psi(\hat{u}^{k},\hat{u}^{\bar{k}})_{i}
-\frac{h^{2}}{2}\psi(\hat{v}^{k},\hat{u}^{\bar{k}})_{i}\right]-\gamma\left[\psi(u^{k},u^{\bar{k}})_{i}
-\frac{h^{2}}{2}\psi(v^{k},u^{\bar{k}})_{i}\right]\nonumber\\
 & +\kappa\left(\Delta_{x}\eta_{i}^{\bar{k}}
-\frac{h^{2}}{6}\Delta_{x}\xi_{i}^{\bar{k}}\right)-\nu \xi_{i}^{\bar{k}}=0,\quad 1\leqslant i\leqslant M,\; 1\leqslant k\leqslant N-1,\label{eq18bbb}\\
&\xi_{i}^{k}=\delta_{x}^{2}\eta_{i}^{k}-\frac{h^{2}}{12}\delta_{x}^{2}\xi_{i}^{k},\quad 1\leqslant i\leqslant M,\; 0\leqslant k\leqslant N,\label{eq19bbb}\\
&\eta_{i}^{0}=\phi^0(x_i),\quad 1\leqslant i\leqslant M,\label{eq20bbb}\\
&\eta_i^k = \eta_{i+M}^k,\quad \xi_i^k = \xi_{i+M}^k,\quad 1\leqslant i\leqslant M,\;0\leqslant k\leqslant N.\label{eq21bbb}
\end{align}
Similar to the proof of Theorem \ref{thm6.1}, we can obtain the stability with respect to the initial value.
  \begin{proof}
  Denote
\begin{align}
G^{k}=\frac{1}{2}\left[(|\eta^{k}|_{1}^{2}+|\eta^{k-1}|_{1}^{2})+\frac{h^{2}}{12}
(\|\xi^{k}\|^{2}+\|\xi^{k-1}\|^{2})-\frac{h^{4}}{144}(|\xi^{k}|_{1}^{2}+|\xi^{k-1}|
_{1}^{2})\right],\quad 1\leqslant k\leqslant N.\label{G}
\end{align}
Taking the inner product of \eqref{eq19bbb} with $\xi^{k}$, we have
\begin{align*}
\|\xi^{k}\|^{2}=&\;(\delta_{x}^{2}\eta^{k},\xi^{k})-\frac{h^{2}}{12}(\delta_{x}^{2}
\xi^{k},\xi^{k})\nonumber\\
\leqslant&\;\|\delta_{x}^{2}\eta^{k}\|\cdot\|\xi^{k}\|+\frac{h^{2}}{12}|\xi^{k}|_{1}^{2}\nonumber\\
\leqslant&\;\frac{3}{4}\|\delta_{x}^{2}\eta^{k}\|^{2}+\frac{1}{3}\|\xi^{k}\|^{2}
+\frac{1}{3}\|\xi^{k}\|^{2}\nonumber\\
\leqslant&\;\frac{3}{h^{2}}|\eta^{k}|_{1}^{2}+\frac{2}{3}\|\xi^{k}\|^{2},\quad 0\leqslant k\leqslant N.
\end{align*}
Thus, we have
\begin{align}
\|\xi^{k}\|^{2}\leqslant\frac{9}{h^{2}}|\eta^{k}|_{1}^{2},\quad 0\leqslant k\leqslant N.\label{xi}
\end{align}
Similarly, we have
\begin{align}
\|\hat{v}^{k}\|^{2}\leqslant\frac{9}{h^{2}}|\hat{u}^{k}|_{1}^{2},\quad 0\leqslant k\leqslant N.\label{v}
\end{align}
On the basis of Theorem $\ref{thm6.1}$ and \eqref{ass1}, when $h\leqslant h_{0}$, $\tau\leqslant\tau_{0}$, we have
\begin{align}
&|u^{k}|_{1}\leqslant|e^{k}|_{1}+|U^{k}|_{1}\leqslant 1+\sqrt{L}c_{0},\quad 0\leqslant k\leqslant N,\label{eq82}\\
&\|u^{k}\|_{\infty}\leqslant\frac{\sqrt{L}}{2}|u^{k}|_{1}\leqslant\frac{\sqrt{L}}{2}
(1+\sqrt{L}c_{0}),\quad 0\leqslant k\leqslant N,\label{eq83}\\
&\|\Delta_{x}u^{k}\|_{\infty}\leqslant\frac{\sqrt{L}}{2}|\Delta_{x}u^{k}|_{1}
\leqslant\frac{\sqrt{L}}{h}|u^{k}|_{1}\leqslant\frac{\sqrt{L}}{h}(1+\sqrt{L}c_{0}),
\quad 0\leqslant k\leqslant N.\label{eq84}
\end{align}
With the help of \eqref{eq20bb} and \eqref{v}, we have
\begin{align}
&|\hat{u}^{0}|_{1}=|\varphi+\psi^{0}|_{1}\leqslant|\varphi|_{1}+|\phi^{0}|_{1},\label{eq85}\\
&\|\hat{u}^{0}\|_{\infty}\leqslant\frac{\sqrt{L}}{2}|\hat{u}^{0}|_{1}\leqslant
\frac{\sqrt{L}}{2}(|\varphi|_{1}
+|\phi^{0}|_{1}),\label{eq86}\\
&|\hat{v}^{0}|_{1}\leqslant\frac{2}{h}\|\hat{v}^{0}\|\leqslant\frac{2}{h}\cdot
\frac{3}{h}|\hat{u}^{0}|_{1}\leqslant\frac{6}{h^{2}}(|\varphi|_{1}+|\phi^{0}|_{1}),\label{eq87}\\
&\|\hat{v}^{0}\|_{\infty}\leqslant\frac{\sqrt{L}}{2}|\hat{v}^{0}|_{1}\leqslant
\frac{3\sqrt{L}}{h^{2}}(|\varphi|_{1}+|\phi^{0}|_{1}).\label{eq88}
\end{align}
Taking the inner product of \eqref{eq17bbb} with $\delta_{t}\eta^{\frac{1}{2}}$, we have
\begin{align}
&\|\delta_{t}\eta^{\frac{1}{2}}\|^{2}-\mu(\delta_{t}\xi^{\frac{1}{2}},
\delta_{t}\eta^{\frac{1}{2}})+\gamma(\psi(\hat{u}^{0},\hat{u}^\frac{1}{2})
-\psi(u^{0},u^\frac{1}{2}),\delta_{t}\eta^{\frac{1}{2}})
-\frac{\gamma h^{2}}{2}(\psi(\hat{v}^{0},\hat{u}^\frac{1}{2})
-\psi(v^{0},u^\frac{1}{2}),\delta_{t}\eta^{\frac{1}{2}})\nonumber\\
&+\kappa(\Delta_{x}\eta^{\frac{1}{2}},\delta_{t}\eta^{\frac{1}{2}})-\frac{\kappa h^{2}}{6}(\Delta_{x}\xi^{\frac{1}{2}},\delta_{t}\eta^{\frac{1}{2}})
-\nu(\xi^{\frac{1}{2}},\delta_{t}\eta^{\frac{1}{2}})=0.\label{eqs1}
\end{align}
Applying \eqref{lem4-1} in Lemma $\ref{lemma4}$, we have
\begin{align}
(\delta_{t}\xi^{\frac{1}{2}},\delta_{t}\eta^{\frac{1}{2}})
=-|\delta_{t}\eta^{\frac{1}{2}}|_{1}^{2}-\frac{h^{2}}{12}\|\delta_{t}\xi^{\frac{1}{2}}\|^{2}
+\frac{h^{4}}{144}|\delta_{t}\xi^{\frac{1}{2}}|_{1}^{2}.\label{eqs2}
\end{align}
According to the definition of $\psi(u,v)_{i}$ and applying Lemma $\ref{lemma5}$, we have
\begin{align}
&\psi(\hat{u}^{0},\hat{u}^\frac{1}{2})_{i}-\psi(u^{0},u^\frac{1}{2})_{i}\nonumber\\
=&\;\psi(\eta^{0}+u^{0},\eta^{\frac{1}{2}}+u^\frac{1}{2})_{i}-\psi(u^{0},u^\frac{1}{2})_{i}\nonumber\\
=&\;\psi(\eta^{0},\eta^{\frac{1}{2}})_{i}+\psi(\eta^{0},u^{\frac{1}{2}})_{i}
+\psi(u^{0},\eta^{\frac{1}{2}})_{i}\nonumber\\
=&\;\psi(\hat{u}^{0},\eta^{\frac{1}{2}})_{i}+\psi(\eta^{0},u^{\frac{1}{2}})_{i}\nonumber\\
=&\;\frac{1}{3}\left[\hat{u}^{0}_{i}\Delta_{x}\eta^{\frac{1}{2}}_{i}
+\Delta_{x}(\hat{u}^{0}\eta^{\frac{1}{2}})_{i}\right]
+\frac{1}{3}\left[\eta^{0}_{i}\Delta_{x}u^{\frac{1}{2}}_{i}
+\Delta_{x}(\eta^{0}u^{\frac{1}{2}})_{i}\right]\nonumber\\
=&\;\frac{1}{3}\left[\hat{u}^{0}_{i}\Delta_{x}\eta^{\frac{1}{2}}_{i}
+\frac{1}{2}\left(\delta_{x}\eta_{i+\frac{1}{2}}^{\frac{1}{2}}\right)\hat{u}_{i+1}^{0}
+\frac{1}{2}\left(\delta_{x}\eta_{i-\frac{1}{2}}^{\frac{1}{2}}\right)\hat{u}_{i-1}^{0}
+\eta_{i}^{\frac{1}{2}}\Delta_{x}\hat{u}_{i}^{0}\right]\nonumber\\
&\;+\frac{1}{3}\left[2\eta^{0}_{i}\Delta_{x}u^{\frac{1}{2}}_{i}
+\frac{1}{2}\left(\delta_{x}\eta_{i+\frac{1}{2}}^{0}\right)u_{i+1}^{\frac{1}{2}}
+\frac{1}{2}\left(\delta_{x}\eta_{i-\frac{1}{2}}^{0}\right)u_{i-1}^{\frac{1}{2}}
\right].\label{eqs3}
\end{align}
From \eqref{eq82}--\eqref{eq88}, \eqref{eqs3}, we have
\begin{align}
&-(\psi(\hat{u}^{0},\hat{u}^\frac{1}{2})-\psi(u^{0},u^\frac{1}{2}),
\delta_{t}\eta^{\frac{1}{2}})\nonumber\\
=\;&-\frac{h}{3}\sum_{i=1}^{M}\left[\hat{u}^{0}_{i}\Delta_{x}\eta^{\frac{1}{2}}_{i}
+\frac{1}{2}\left(\delta_{x}\eta_{i+\frac{1}{2}}^{\frac{1}{2}}\right)\hat{u}_{i+1}^{0}
+\frac{1}{2}\left(\delta_{x}\eta_{i-\frac{1}{2}}^{\frac{1}{2}}\right)\hat{u}_{i-1}^{0}
+\eta_{i}^{\frac{1}{2}}\Delta_{x}\hat{u}_{i}^{0}\right]\delta_{t}\eta_{i}^{\frac{1}{2}}\nonumber\\
&\;-\frac{h}{3}\sum_{i=1}^{M}\left[2\eta^{0}_{i}\Delta_{x}u^{\frac{1}{2}}_{i}
+\frac{1}{2}\left(\delta_{x}\eta_{i+\frac{1}{2}}^{0}\right)u_{i+1}^{\frac{1}{2}}
+\frac{1}{2}\left(\delta_{x}\eta_{i-\frac{1}{2}}^{0}\right)u_{i-1}^{\frac{1}{2}}
\right]\delta_{t}\eta_{i}^{\frac{1}{2}}\nonumber\\
\leqslant&\;\frac{1}{3}\left(\|\hat{u}^{0}\|_{\infty}\cdot|\eta^{\frac{1}{2}}|_{1}
+|\eta^{\frac{1}{2}}|_{1}\cdot\|\hat{u}^{0}\|_{\infty}+\|\eta^{\frac{1}{2}}\|_{\infty}
\cdot|\hat{u}^{0}|_{1}\right)\cdot\|\delta_{t}\eta^{\frac{1}{2}}\|\nonumber\\
&\;+\frac{1}{3}\left(2\|\eta^{0}\|_{\infty}\cdot|u^{\frac{1}{2}}|_{1}+|\eta^{0}|_{1}
\cdot\|u^{\frac{1}{2}}\|_{\infty}\right)\cdot\|\delta_{t}\eta^{\frac{1}{2}}\|\nonumber\\
\leqslant&\;\left(\frac{2}{3}\|\hat{u}^{0}\|_{\infty}+\frac{\sqrt{L}}{6}|\hat{u}^{0}|_{1}\right)
\cdot|\eta^{\frac{1}{2}}|_{1}\cdot\|\delta_{t}\eta^{\frac{1}{2}}\|
+\left(\frac{\sqrt{L}}{3}|u^{\frac{1}{2}}|_{1}+\frac{1}{3}\|u^{\frac{1}{2}}\|_{\infty}\right)
\cdot|\eta^{0}|_{1}\cdot\|\delta_{t}\eta^{\frac{1}{2}}\|\nonumber\\
\leqslant&\;\left[\frac{2}{3}\cdot\frac{\sqrt{L}}{2}(|\varphi|_{1}+|\phi^{0}|_{1})
+\frac{\sqrt{L}}{6}(|\varphi|_{1}+|\phi^{0}|_{1})\right]\cdot|\eta^{\frac{1}{2}}|_{1}
\cdot\|\delta_{t}\eta^{\frac{1}{2}}\|\nonumber\\
&\;+\left[\frac{\sqrt{L}}{3}(1+\sqrt{L}c_{0})
+\frac{1}{3}\cdot\frac{\sqrt{L}}{2}(1+\sqrt{L}c_{0})\right]\cdot|\eta^{0}|_{1}\cdot
\|\delta_{t}\eta^{\frac{1}{2}}\|\nonumber\\
=&\;\frac{\sqrt{L}}{2}(|\varphi|_{1}+|\phi^{0}|_{1})\cdot|\eta^{\frac{1}{2}}|_{1}
\cdot\|\delta_{t}\eta^{\frac{1}{2}}\|+\frac{\sqrt{L}}{2}(1+\sqrt{L}c_{0})\cdot
|\eta^{0}|_{1}\cdot\|\delta_{t}\eta^{\frac{1}{2}}\|.\label{eqs4}
\end{align}
Similarly, it is concluded that
\begin{align*}
&\psi(\hat{v}^{0},\hat{u}^\frac{1}{2})_{i}-\psi(v^{0},u^\frac{1}{2})_{i}\nonumber\\
=&\;\frac{1}{3}\left[\hat{v}^{0}_{i}\Delta_{x}\eta^{\frac{1}{2}}_{i}
+\frac{1}{2}\left(\delta_{x}\eta_{i+\frac{1}{2}}^{\frac{1}{2}}\right)\hat{v}_{i+1}^{0}
+\frac{1}{2}\left(\delta_{x}\eta_{i-\frac{1}{2}}^{\frac{1}{2}}\right)\hat{v}_{i-1}^{0}
+\eta_{i}^{\frac{1}{2}}\Delta_{x}\hat{v}_{i}^{0}\right]\nonumber\\
&\;+\frac{1}{3}\left[2\xi^{0}_{i}\Delta_{x}u^{\frac{1}{2}}_{i}
+\frac{1}{2}\left(\delta_{x}\xi_{i+\frac{1}{2}}^{0}\right)u_{i+1}^{\frac{1}{2}}
+\frac{1}{2}\left(\delta_{x}\xi_{i-\frac{1}{2}}^{0}\right)u_{i-1}^{\frac{1}{2}}
\right].
\end{align*}
Thus we have
\begin{align}
&(\psi(\hat{v}^{0},\hat{u}^\frac{1}{2})-\psi(v^{0},u^\frac{1}{2}),
\delta_{t}\eta^{\frac{1}{2}})\nonumber\\
=&\;\frac{h}{3}\sum_{i=1}^{M}\left[\hat{v}^{0}_{i}\Delta_{x}\eta^{\frac{1}{2}}_{i}
+\frac{1}{2}\left(\delta_{x}\eta_{i+\frac{1}{2}}^{\frac{1}{2}}\right)\hat{v}_{i+1}^{0}
+\frac{1}{2}\left(\delta_{x}\eta_{i-\frac{1}{2}}^{\frac{1}{2}}\right)\hat{v}_{i-1}^{0}
+\eta_{i}^{\frac{1}{2}}\Delta_{x}\hat{v}_{i}^{0}\right]\delta_{t}\eta_{i}^{\frac{1}{2}}\nonumber\\
&\;+\frac{h}{3}\sum_{i=1}^{M}\left[2\xi^{0}_{i}\Delta_{x}u^{\frac{1}{2}}_{i}
+\frac{1}{2}\left(\delta_{x}\xi_{i+\frac{1}{2}}^{0}\right)u_{i+1}^{\frac{1}{2}}
+\frac{1}{2}\left(\delta_{x}\xi_{i-\frac{1}{2}}^{0}\right)u_{i-1}^{\frac{1}{2}}
\right]\delta_{t}\eta_{i}^{\frac{1}{2}}\nonumber\\
\leqslant&\;\frac{1}{3}\left(\|\hat{v}^{0}\|_{\infty}\cdot|\eta^{\frac{1}{2}}|_{1}
+|\eta^{\frac{1}{2}}|_{1}\cdot\|\hat{v}^{0}\|_{\infty}+\|\eta^{\frac{1}{2}}\|_{\infty}
\cdot|\hat{v}^{0}|_{1}\right)\cdot\|\delta_{t}\eta^{\frac{1}{2}}\|\nonumber\\
&\;+\frac{1}{3}\left(2\|\xi^{0}\|\cdot\|\Delta_{x}u^{\frac{1}{2}}\|_{\infty}
+|\xi^{0}|_{1}\cdot\|u^{\frac{1}{2}}\|_{\infty}\right)\cdot\|\delta_{t}
\eta^{\frac{1}{2}}\|\nonumber\\
\leqslant&\;\left(\frac{2}{3}\|\hat{v}^{0}\|_{\infty}+\frac{\sqrt{L}}{6}|\hat{v}^{0}|_{1}
\right)\cdot|\eta^{\frac{1}{2}}|_{1}\cdot\|\delta_{t}\eta^{\frac{1}{2}}\|
+\left(\frac{2}{3}\|\Delta_{x}u^{\frac{1}{2}}\|_{\infty}
+\frac{2}{3h}\|u^{\frac{1}{2}}\|_{\infty}\right)\cdot\|\xi^{0}\|\cdot
\|\delta_{t}\eta^{\frac{1}{2}}\|\nonumber\\
\leqslant&\;\left[\frac{2}{3}\cdot\frac{3\sqrt{L}}{h^{2}}(|\varphi|_{1}+|\phi^{0}|_{1})
+\frac{\sqrt{L}}{6}\cdot\frac{6}{h^{2}}(|\varphi|_{1}+|\phi^{0}|_{1})\right]\cdot
|\eta^{\frac{1}{2}}|_{1}\cdot\|\delta_{t}\eta^{\frac{1}{2}}\|\nonumber\\
&\;+\left[\frac{2}{3}\cdot\frac{\sqrt{L}}{h}(1+\sqrt{L}c_{0})+\frac{2}{3h}\cdot
\frac{\sqrt{L}}{2}(1+\sqrt{L}c_{0})\right]\cdot\|\xi^{0}\|\cdot
\|\delta_{t}\eta^{\frac{1}{2}}\|\nonumber\\
=&\;\frac{3\sqrt{L}}{h^{2}}(|\varphi|_{1}+|\phi^{0}|_{1})\cdot|\eta^{\frac{1}{2}}|_{1}
\cdot\|\delta_{t}\eta^{\frac{1}{2}}\|+\frac{\sqrt{L}}{h}(1+\sqrt{L}c_{0})\cdot\|\xi^{0}\|\cdot
\|\delta_{t}\eta^{\frac{1}{2}}\|.\label{eqs6}
\end{align}
In addition, applying Cauchy-Schwarz inequality, we have
\begin{align}
-(\Delta_{x}\eta^{\frac{1}{2}},\delta_{t}\eta^{\frac{1}{2}})
\leqslant\|\Delta_{x}\eta^{\frac{1}{2}}\|\cdot\|\delta_{t}\eta^{\frac{1}{2}}\|.\label{eqs7}
\end{align}
Moreover, it holds
\begin{align}
(\Delta_{x}\xi^{\frac{1}{2}},\delta_{t}\eta^{\frac{1}{2}})
\leqslant|\Delta_{x}\eta^{\frac{1}{2}}|_{1}\cdot|\delta_{t}\eta^{\frac{1}{2}}|_{1}
+\frac{h^{2}}{12}|\xi^{\frac{1}{2}}|_{1}\cdot\|\delta_{t}\xi^{\frac{1}{2}}\|
+\frac{h^{4}}{144}|\Delta_{x}\xi^{\frac{1}{2}}|_{1}\cdot|\delta_{t}\xi^{\frac{1}{2}}|_{1}.\label{eqs8}
\end{align}
Similar to the derivation of \eqref{eq25}, we have
\begin{align}
&\;(\xi^{\frac{1}{2}},\delta_{t}\eta^{\frac{1}{2}})
=-\frac{1}{2\tau}(|\eta^{1}|_{1}^{2}-|\eta^{0}|_{1}^{2})-\frac{h^{2}}{12}
\cdot\frac{1}{2\tau}(\|\xi^{1}\|^{2}-\|\xi^{0}\|^{2})+\frac{h^{4}}{144}
\cdot\frac{1}{2\tau}(|\xi^{1}|_{1}^{2}-|\xi^{0}|_{1}^{2}).\label{eqs9}
\end{align}
Substituting \eqref{eqs2}, \eqref{eqs4}--\eqref{eqs9}into
\eqref{eqs1}, we have
\begin{align}
&\|\delta_{t}\eta^{\frac{1}{2}}\|^{2}\nonumber\\
\leqslant&\;\mu\left(-|\delta_{t}\eta^{\frac{1}{2}}|_{1}^{2}-\frac{h^{2}}{12}
\|\delta_{t}\xi^{\frac{1}{2}}\|^{2}+\frac{h^{4}}{144}|\delta_{t}
\xi^{\frac{1}{2}}|_{1}^{2}\right)+\kappa\cdot\|\Delta_{x}\eta^{\frac{1}{2}}
\|\cdot\|\delta_{t}\eta^{\frac{1}{2}}\|\nonumber\\
&\;+\gamma\left[\frac{\sqrt{L}}{2}(|\varphi|_{1}+|\phi^{0}|_{1})\cdot
|\eta^{\frac{1}{2}}|_{1}\cdot\|\delta_{t}\eta^{\frac{1}{2}}\|+\frac{\sqrt{L}}{2}
(1+\sqrt{L}c_{0})\cdot|\eta^{0}|_{1}\cdot\|\delta_{t}\eta^{\frac{1}{2}}\|\right]\nonumber\\
&\;+\frac{\gamma h^{2}}{2}\left[\frac{3\sqrt{L}}{h^{2}}(|\varphi|_{1}+|\phi^{0}|_{1})
\cdot|\eta^{\frac{1}{2}}|_{1}\cdot\|\delta_{t}\eta^{\frac{1}{2}}\|+\frac{\sqrt{L}}{h}
(1+\sqrt{L}c_{0})\cdot\|\xi^{0}\|\cdot\|\delta_{t}\eta^{\frac{1}{2}}\|\right]\nonumber\\
&\;
+\frac{\kappa h^{2}}{6}\left(|\Delta_{x}\eta^{\frac{1}{2}}|_{1}\cdot|\delta_{t}
\eta^{\frac{1}{2}}|_{1}+\frac{h^{2}}{12}|\xi^{\frac{1}{2}}|_{1}\cdot
\|\delta_{t}\xi^{\frac{1}{2}}\|+\frac{h^{4}}{144}|\Delta_{x}\xi^{\frac{1}{2}}|_{1}
\cdot|\delta_{t}\xi^{\frac{1}{2}}|_{1}\right)\nonumber\\
&\;+\nu\left[-\frac{1}{2\tau}(|\eta^{1}|_{1}^{2}-|\eta^{0}|_{1}^{2})-\frac{h^{2}}{12}
\cdot\frac{1}{2\tau}(\|\xi^{1}\|^{2}-\|\xi^{0}\|^{2})+\frac{h^{4}}{144}
\cdot\frac{1}{2\tau}(|\xi^{1}|_{1}^{2}-|\xi^{0}|_{1}^{2})\right]\nonumber\\
\leqslant&\;-\frac{\mu h^{2}}{12}\|\delta_{t}\xi^{\frac{1}{2}}\|^{2}
+\frac{\mu h^{4}}{144}\cdot\frac{4}{h^{2}}\|\delta_{t}\xi^{\frac{1}{2}}\|^{2}
+\frac{1}{6}\|\delta_{t}\eta^{\frac{1}{2}}\|^{2}
+\frac{3\kappa^{2}}{2}\|\Delta_{x}\eta^{\frac{1}{2}}\|^{2}
+\frac{1}{6}\|\delta_{t}\eta^{\frac{1}{2}}\|^{2}\nonumber\\
&\;+\frac{3\gamma^{2}L}{8}(|\varphi|_{1}
+|\phi^{0}|_{1})^{2}|\eta^{\frac{1}{2}}|_{1}^{2}+\frac{1}{6}\|\delta_{t}
\eta^{\frac{1}{2}}\|^{2}+\frac{3\gamma^{2}L}{8}(1+\sqrt{L}c_{0})^{2}|\eta^{0}|_{1}^{2}+\frac{1}{6}
\|\delta_{t}\eta^{\frac{1}{2}}\|^{2}\nonumber\\
&\;+\frac{27\gamma^{2}L}{8}(|\varphi|_{1}
+|\phi^{0}|_{1})^{2}|\eta^{\frac{1}{2}}|_{1}^{2}+\frac{1}{6}\|\delta_{t}
\eta^{\frac{1}{2}}\|^{2}+\frac{3\gamma^{2}h^{2}L}{8}(1+\sqrt{L}c_{0})^{2}
\|\xi^{0}\|^{2}+\frac{1}{6}\cdot\frac{h^{2}}{4}|\delta_{t}\eta^{\frac{1}{2}}|_{1}^{2}\nonumber\\
&\;+\frac{\kappa^{2}h^{2}}{6}|\Delta_{x}\eta^{\frac{1}{2}}|_{1}^{2}
+\frac{\mu h^{2}}{36}\|\delta_{t}\xi^{\frac{1}{2}}\|^{2}
+\frac{\kappa^{2}h^{6}}{576\mu}|\xi^{\frac{1}{2}}|_{1}^{2}
+\frac{\mu h^{2}}{36}\cdot\frac{h^{2}}{4}|\delta_{t}\xi^{\frac{1}{2}}|_{1}^{2}
+\frac{\kappa^{2}h^{8}}{144^{2}\mu}|\Delta_{x}\xi^{\frac{1}{2}}|_{1}^{2}\nonumber\\
&\;-\frac{\nu}{2\tau}\left[(|\eta^{1}|_{1}^{2}-|\eta^{0}|_{1}^{2})+\frac{h^{2}}{12}
(\|\xi^{1}\|^{2}-\|\xi^{0}\|^{2})-\frac{h^{4}}{144}
(|\xi^{1}|_{1}^{2}-|\xi^{0}|_{1}^{2})\right]\nonumber\\
\leqslant&\;\|\delta_{t}\eta^{\frac{1}{2}}\|^{2}+\left[\frac{3\gamma^{2}L}{8}(|\varphi|_{1}
+|\phi^{0}|_{1})^{2}+\frac{27\gamma^{2}L}{8}(|\varphi|_{1}+|\phi^{0}|_{1})^{2}+\frac{3\kappa^{2}}{2}
+\frac{2\kappa^{2}}{3}\right]|\eta^{\frac{1}{2}}|_{1}^{2}
\nonumber\\
&\; +\frac{3\gamma^{2}L}{8}(1+\sqrt{L}c_{0})^{2}|\eta^{0}|_{1}^{2}
+\left(\frac{\kappa^{2}h^{4}}{144\mu}+\frac{\kappa^{2}h^{4}}{1296\mu}\right)
\|\xi^{\frac{1}{2}}\|^{2}+\frac{3\gamma^{2}h^{2}L}{8}
(1+\sqrt{L}c_{0})^{2}\|\xi^{0}\|^{2}\nonumber\\
&\;-\frac{\nu}{2\tau}\left[(|\eta^{1}|_{1}^{2}-|\eta^{0}|_{1}^{2})+\frac{h^{2}}{12}
(\|\xi^{1}\|^{2}-\|\xi^{0}\|^{2})-\frac{h^{4}}{144}
(|\xi^{1}|_{1}^{2}-|\xi^{0}|_{1}^{2})\right]\nonumber\\
\leqslant&\;\|\delta_{t}\eta^{\frac{1}{2}}\|^{2}+\left[\frac{15\gamma^{2}L}{4}
(|\varphi|_{1}+|\phi^{0}|_{1})^{2}+\frac{13\kappa^{2}}{6}\right]\cdot
\frac{|\eta^{1}|_{1}^{2}+|\eta^{0}|_{1}^{2}}{2}+\frac{3\gamma^{2}L}{8}
(1+\sqrt{L}c_{0})^{2}|\eta^{0}|_{1}^{2}\nonumber\\
&\;+\frac{5\kappa^{2}h^{4}}{648\mu}\cdot\frac{\|\xi^{1}\|^{2}+\|\xi^{0}\|^{2}}{2}
+\frac{3\gamma^{2}h^{2}L}{8}(1+\sqrt{L}c_{0})^{2}\|\xi^{0}\|^{2}\nonumber\\
&\;-\frac{\nu}{2\tau}\left[(|\eta^{1}|_{1}^{2}-|\eta^{0}|_{1}^{2})+\frac{h^{2}}{12}
(\|\xi^{1}\|^{2}-\|\xi^{0}\|^{2})-\frac{h^{4}}{144}
(|\xi^{1}|_{1}^{2}-|\xi^{0}|_{1}^{2})\right].\label{eqs10}
\end{align}
Simplifying the formula \eqref{eqs10}, we have
\begin{align*}
&\frac{1}{2}\left[(|\eta^{1}|_{1}^{2}+|\eta^{0}|_{1}^{2})+\frac{h^{2}}{12}
(\|\xi^{1}\|^{2}+\|\xi^{0}\|^{2})-\frac{h^{4}}{144}
(|\xi^{1}|_{1}^{2}+|\xi^{0}|_{1}^{2})\right]\\
\leqslant&\;\left(|\eta^{0}|_{1}^{2}+\frac{h^{2}}{12}\|\xi^{0}\|^{2}
-\frac{h^{4}}{144}|\xi^{0}|_{1}^{2}\right)+\left[\frac{15\gamma^{2}\tau L}{4\nu}
(|\varphi|_{1}+|\phi^{0}|_{1})^{2}+\frac{13\kappa^{2}\tau}{6\nu}\right]\cdot
\frac{|\eta^{1}|_{1}^{2}+|\eta^{0}|_{1}^{2}}{2}\nonumber\\
&\;+\frac{3\gamma^{2}\tau L}{8\nu}(1+\sqrt{L}c_{0})^{2}|\eta^{0}|_{1}^{2}
+\frac{5\kappa^{2}h^{4}\tau}{648\mu\nu}\cdot\frac{\|\xi^{1}\|^{2}+\|\xi^{0}\|^{2}}{2}
+\frac{3\gamma^{2}h^{2}\tau L}{8\nu}(1+\sqrt{L}c_{0})^{2}\|\xi^{0}\|^{2}\\
\leqslant&\;\left[\frac{15\gamma^{2}\tau L}{4\nu}(|\varphi|_{1}+|\phi^{0}|_{1})
^{2}+\frac{13\kappa^{2}\tau}{6\nu}\right]\cdot\frac{|\eta^{1}|_{1}^{2}
+|\eta^{0}|_{1}^{2}}{2}+\frac{5\kappa^{2}h^{2}\tau}{36\mu\nu}\cdot\frac{h^{2}}{36}
(\|\xi^{1}\|^{2}+\|\xi^{0}\|^{2})\nonumber\\
&\;+\left[1+\frac{3\gamma^{2}\tau L}{8\nu}(1+\sqrt{L}c_{0})^{2}\right]
\cdot|\eta^{0}|_{1}^{2}+\left[\left(\frac{h^{2}}{12}-\frac{h^{4}}{144}\cdot
\frac{6}{L^{2}}\right)+\frac{3\gamma^{2}h^{2}\tau L}{8\nu}(1+\sqrt{L}c_{0})^{2}\right]\cdot\|\xi^{0}\|^{2}\nonumber\\
\leqslant&\;\left[\frac{15\gamma^{2}\tau L}{4\nu}(|\varphi|_{1}+|\phi^{0}|_{1})
^{2}+\frac{13\kappa^{2}\tau}{6\nu}\right]\cdot\frac{|\eta^{1}|_{1}^{2}
+|\eta^{0}|_{1}^{2}}{2}+\frac{5\kappa^{2}h^{2}\tau}{36\mu\nu}\cdot\frac{h^{2}}{36}
(\|\xi^{1}\|^{2}+\|\xi^{0}\|^{2})\nonumber\\
&\;+\left[1+\frac{3\gamma^{2}\tau L}{8\nu}(1+\sqrt{L}c_{0})^{2}+\frac{9}{h^{2}}\left(\frac{h^{2}}{12}-\frac{h^{4}}{144}\cdot
\frac{6}{L^{2}}\right)+\frac{27\gamma^{2}\tau L}{8\nu}(1+\sqrt{L}c_{0})^{2}\right]
\cdot|\eta^{0}|_{1}^{2}\nonumber\\
\leqslant&\;c_{12}\tau
\cdot\frac{1}{2}\left[(|\eta^{1}|_{1}^{2}+|\eta^{0}|_{1}^{2})+\frac{h^{2}}{12}
(\|\xi^{1}\|^{2}+\|\xi^{0}\|^{2})
-\frac{h^{4}}{144}(|\xi^{1}|_{1}^{2}+|\xi^{0}|_{1}^{2})\right]\nonumber\\
&\;+\left[\frac{7}{4}+\frac{15\gamma^{2}\tau L}{4\nu}(1+\sqrt{L}c_{0})^{2}\right]
\cdot|\eta^{0}|_{1}^{2},
\end{align*}
when $h\leqslant h_{0}$, $\tau\leqslant \tau_{0}$ and $c_{12}\tau\leqslant 1/4$, we have
\begin{align*}
&\frac{1}{2}\left[(|\eta^{1}|_{1}^{2}+|\eta^{0}|_{1}^{2})+\frac{h^{2}}{12}
(\|\xi^{1}\|^{2}+\|\xi^{0}\|^{2})-\frac{h^{4}}{144}
(|\xi^{1}|_{1}^{2}+|\xi^{0}|_{1}^{2})\right]\\
\leqslant&\;\frac{1}{4}\cdot\frac{1}{2}\left[(|\eta^{1}|_{1}^{2}+|\eta^{0}|_{1}^{2})
+\frac{h^{2}}{12}(\|\xi^{1}\|^{2}+\|\xi^{0}\|^{2})-\frac{h^{4}}{144}
(|\xi^{1}|_{1}^{2}+|\xi^{0}|_{1}^{2})\right]+c_{13}|\eta^{0}|_{1}^{2}.
\end{align*}
Therefore, we have
\begin{align}
G^{1}\leqslant \frac{4c_{13}}{3}|\phi^{0}|_{1}^{2}.\label{eqs11}
\end{align}
Moreover
\begin{align}
|\eta^{1}|_{1}\leqslant \sqrt{\frac{8c_{13}}{3}}|\phi^{0}|_{1}.\label{eqs12}
\end{align}
Taking the inner product of \eqref{eq18bbb} with $\Delta_{t}\eta^{k}$, we have
\begin{align}
&\|\Delta_{t}\eta^{k}\|^{2}-\mu(\Delta_{t}\xi^{k},\Delta_{t}\eta^{k})
+\gamma(\psi(\hat{u}^{k},\hat{u}^{\bar{k}})-\psi(u^{k},u^{\bar{k}}),
\Delta_{t}\eta^{k})-\frac{\gamma h^{2}}{2}(\psi(\hat{v}^{k},
\hat{u}^{\bar{k}})-\psi(v^{k},u^{\bar{k}}),\Delta_{t}\eta^{k})\nonumber\\
&+\kappa(\Delta_{x}\eta^{\bar{k}},\Delta_{t}\eta^{k})-\frac{\kappa h^{2}}{6}
(\Delta_{x}\xi^{\bar{k}},\Delta_{t}\eta^{k})-\nu(\xi^{\bar{k}},
\Delta_{t}\eta^{k})=0,\quad 1\leqslant k\leqslant N-1.\label{eqs13}
\end{align}Now we suppose that $|\eta^{k}|_{1}\leqslant c_{11}|\phi^{0}|_{1}$ holds for $k=1,2,\cdots,l$ with $1\leqslant l\leqslant N-1$. When $c_{11}|\phi^{0}|_{1}\leqslant 1$, we have
\begin{align}
&|\hat{u}^{k}|_{1}\leqslant|\eta^{k}|_{1}+|u^{k}|_{1}\leqslant 2+\sqrt{L}c_{0},\quad 1\leqslant k\leqslant l,\label{eq89}\\
&\|\hat{u}^{k}\|_{\infty}\leqslant\frac{\sqrt{L}}{2}|\hat{u}^{k}|_{1}\leqslant
\frac{\sqrt{L}}{2}(2+\sqrt{L}c_{0}),\quad 1\leqslant k\leqslant l,\label{eq90}\\
&|\hat{v}^{k}|_{1}\leqslant\frac{2}{h}\|\hat{v}^{k}\|\leqslant\frac{2}{h}
\cdot\frac{3}{h}|\hat{u}^{k}|_{1}\leqslant\frac{6}{h^{2}}(2+\sqrt{L}c_{0}),\quad 1\leqslant k\leqslant l,\label{eq91}\\
&\|\hat{v}^{k}\|_{\infty}\leqslant\frac{\sqrt{L}}{2}|\hat{v}^{k}|_{1}\leqslant
\frac{3\sqrt{L}}{h^{2}}(2+\sqrt{L}c_{0}),\quad 1\leqslant k\leqslant l.\label{eq92}
\end{align}
Applying \eqref{lem4-1} in Lemma $\ref{lemma4}$, we have
\begin{align}
(\Delta_{t}\xi^{k},\Delta_{t}\eta^{k})
=-|\Delta_{t}\eta^{k}|_{1}^{2}-\frac{h^{2}}{12}\|\Delta_{t}\xi^{k}\|^{2}
+\frac{h^{4}}{144}|\Delta_{t}\xi^{k}|_{1}^{2},\quad 1\leqslant k\leqslant l.\label{eqs13}
\end{align}
According to the definition of $\psi(u,v)_{i}$ and applying Lemma $\ref{lemma5}$, we have
\begin{align}
&\psi(\hat{u}^{k},\hat{u}^{\bar{k}})_{i}-\psi(u^{k},u^{\bar{k}})_{i}\nonumber\\
=&\;\frac{1}{3}\left[\hat{u}^{k}_{i}\Delta_{x}\eta^{\bar{k}}_{i}
+\frac{1}{2}\left(\delta_{x}\eta_{i+\frac{1}{2}}^{\bar{k}}\right)\hat{u}_{i+1}^{k}
+\frac{1}{2}\left(\delta_{x}\eta_{i-\frac{1}{2}}^{\bar{k}}\right)\hat{u}_{i-1}^{k}
+\eta_{i}^{\bar{k}}\Delta_{x}\hat{u}_{i}^{k}\right]\nonumber\\
&\;+\frac{1}{3}\left[2\eta^{k}_{i}\Delta_{x}u^{\bar{k}}_{i}
+\frac{1}{2}\left(\delta_{x}\eta_{i+\frac{1}{2}}^{k}\right)u_{i+1}^{\bar{k}}
+\frac{1}{2}\left(\delta_{x}\eta_{i-\frac{1}{2}}^{k}\right)u_{i-1}^{\bar{k}}
\right].\label{eqs14}
\end{align}
According to \eqref{eqs14}, we have
\begin{align}
&-(\psi(\hat{u}^{k},\hat{u}^{\bar{k}})-\psi(u^{k},u^{\bar{k}}),
\Delta_{t}\eta^{\bar{k}})\nonumber\\
=\;&-\frac{h}{3}\sum_{i=1}^{M}\left[\hat{u}^{k}_{i}\Delta_{x}\eta^{\bar{k}}_{i}
+\frac{1}{2}\left(\delta_{x}\eta_{i+\frac{1}{2}}^{\bar{k}}\right)\hat{u}_{i+1}^{k}
+\frac{1}{2}\left(\delta_{x}\eta_{i-\frac{1}{2}}^{\bar{k}}\right)\hat{u}_{i-1}^{k}
+\eta_{i}^{\bar{k}}\Delta_{x}\hat{u}_{i}^{k}\right]\Delta_{t}\eta_{i}^{k}\nonumber\\
&\;-\frac{h}{3}\sum_{i=1}^{M}\left[2\eta^{k}_{i}\Delta_{x}u^{\bar{k}}_{i}
+\frac{1}{2}\left(\delta_{x}\eta_{i+\frac{1}{2}}^{k}\right)u_{i+1}^{\bar{k}}
+\frac{1}{2}\left(\delta_{x}\eta_{i-\frac{1}{2}}^{k}\right)u_{i-1}^{\bar{k}}
\right]\Delta_{t}\eta_{i}^{k}\nonumber\\
\leqslant&\;\frac{1}{3}\left(\|\hat{u}^{k}\|_{\infty}\cdot|\eta^{\bar{k}}|_{1}
+|\eta^{\bar{k}}|_{1}\cdot\|\hat{u}^{k}\|_{\infty}+\|\eta^{\bar{k}}\|_{\infty}
\cdot|\hat{u}^{k}|_{1}\right)\cdot\|\Delta_{t}\eta^{k}\|\nonumber\\
&\;+\frac{1}{3}\left(2\|\eta^{k}\|_{\infty}\cdot|u^{\bar{k}}|_{1}+|\eta^{k}|_{1}
\cdot\|u^{\bar{k}}\|_{\infty}\right)\cdot\|\Delta_{t}\eta^{k}\|\nonumber\\
\leqslant&\;\left(\frac{2}{3}\|\hat{u}^{k}\|_{\infty}+\frac{\sqrt{L}}{6}|\hat{u}^{k}|_{1}\right)
\cdot|\eta^{\bar{k}}|_{1}\cdot\|\Delta_{t}\eta^{k}\|
+\left(\frac{\sqrt{L}}{3}|u^{\bar{k}}|_{1}+\frac{1}{3}\|u^{\bar{k}}\|_{\infty}\right)
\cdot|\eta^{k}|_{1}\cdot\|\Delta_{t}\eta^{k}\|\nonumber\\
\leqslant&\;\left[\frac{2}{3}\cdot\frac{\sqrt{L}}{2}(2+\sqrt{L}c_{0})
+\frac{\sqrt{L}}{6}(2+\sqrt{L}c_{0})\right]\cdot|\eta^{\bar{k}}|_{1}\cdot
\|\Delta_{t}\eta^{k}\|\nonumber\\
&\;+\left[\frac{\sqrt{L}}{3}(1+\sqrt{L}c_{0})+\frac{1}{3}\cdot
\frac{\sqrt{L}}{2}(1+\sqrt{L}c_{0})\right]\cdot|\eta^{k}|_{1}\cdot
\|\Delta_{t}\eta^{k}\|\nonumber\\
=&\;\frac{\sqrt{L}}{2}(2+\sqrt{L}c_{0})\cdot|\eta^{\bar{k}}|_{1}\cdot
\|\Delta_{t}\eta^{k}\|+\frac{\sqrt{L}}{2}(1+\sqrt{L}c_{0})\cdot|\eta^{k}|_{1}\cdot
\|\Delta_{t}\eta^{k}\|,\quad 1\leqslant k\leqslant l.\label{eqs15}
\end{align}
Similarly, it is concluded that
\begin{align*}
&\psi(\hat{v}^{k},\hat{u}^{\bar{k}})_{i}-\psi(v^{k},u^{\bar{k}})_{i}\nonumber\\
=&\;\frac{1}{3}\left[\hat{v}^{k}_{i}\Delta_{x}\eta^{\bar{k}}_{i}
+\frac{1}{2}\left(\delta_{x}\eta_{i+\frac{1}{2}}^{\bar{k}}\right)\hat{v}_{i+1}^{k}
+\frac{1}{2}\left(\delta_{x}\eta_{i-\frac{1}{2}}^{\bar{k}}\right)\hat{v}_{i-1}^{k}
+\eta_{i}^{\bar{k}}\Delta_{x}\hat{v}_{i}^{k}\right]\nonumber\\
&\;+\frac{1}{3}\left[2\xi^{k}_{i}\Delta_{x}u^{\bar{k}}_{i}
+\frac{1}{2}\left(\delta_{x}\xi_{i+\frac{1}{2}}^{k}\right)u_{i+1}^{\bar{k}}
+\frac{1}{2}\left(\delta_{x}\xi_{i-\frac{1}{2}}^{k}\right)u_{i-1}^{\bar{k}}
\right].
\end{align*}
Thus we have
\begin{align}
&(\psi(\hat{v}^{k},\hat{u}^{\bar{k}})-\psi(v^{k},u^{\bar{k}}),
\delta_{t}\eta^{\bar{k}})\nonumber\\
=&\;\frac{h}{3}\sum_{i=1}^{M}\left[\hat{v}^{k}_{i}\Delta_{x}\eta^{\bar{k}}_{i}
+\frac{1}{2}\left(\delta_{x}\eta_{i+\frac{1}{2}}^{\bar{k}}\right)\hat{v}_{i+1}^{k}
+\frac{1}{2}\left(\delta_{x}\eta_{i-\frac{1}{2}}^{\bar{k}}\right)\hat{v}_{i-1}^{k}
+\eta_{i}^{\bar{k}}\Delta_{x}\hat{v}_{i}^{k}\right]\Delta_{t}\eta_{i}^{k}\nonumber\\
&\;+\frac{h}{3}\sum_{i=1}^{M}\left[2\xi^{k}_{i}\Delta_{x}u^{\bar{k}}_{i}
+\frac{1}{2}\left(\delta_{x}\xi_{i+\frac{1}{2}}^{k}\right)u_{i+1}^{\bar{k}}
+\frac{1}{2}\left(\delta_{x}\xi_{i-\frac{1}{2}}^{k}\right)u_{i-1}^{\bar{k}}
\right]\Delta_{t}\eta_{i}^{k}\nonumber\\
\leqslant&\;\frac{1}{3}\left(\|\hat{v}^{k}\|_{\infty}\cdot|\eta^{\bar{k}}|_{1}
+|\eta^{\bar{k}}|_{1}\cdot\|\hat{v}^{k}\|_{\infty}+\|\eta^{\bar{k}}\|_{\infty}
\cdot|\hat{v}^{k}|_{1}\right)\cdot\|\Delta_{t}\eta^{k}\|\nonumber\\
&\;+\frac{1}{3}\left(2\|\xi^{k}\|\cdot\|\Delta_{x}u^{\bar{k}}\|_{\infty}
+|\xi^{k}|_{1}\cdot\|u^{\bar{k}}\|_{\infty}\right)\cdot\|\Delta_{t}
\eta^{k}\|\nonumber\\
\leqslant&\;\left(\frac{2}{3}\|\hat{v}^{k}\|_{\infty}+\frac{\sqrt{L}}{6}|\hat{v}^{k}|_{1}
\right)\cdot|\eta^{\bar{k}}|_{1}\cdot\|\Delta_{t}\eta^{k}\|
+\left(\frac{2}{3}\|\Delta_{x}u^{\bar{k}}\|_{\infty}
+\frac{2}{3h}\|u^{\bar{k}}\|_{\infty}\right)\cdot\|\xi^{k}\|\cdot
\|\Delta_{t}\eta^{k}\|\nonumber\\
\leqslant&\;\left[\frac{2}{3}\cdot\frac{3\sqrt{L}}{h^{2}}(2+\sqrt{L}c_{0})
+\frac{\sqrt{L}}{6}\cdot\frac{6}{h^{2}}(2+\sqrt{L}c_{0})\right]\cdot
|\eta^{\bar{k}}|_{1}\cdot\|\Delta_{t}\eta^{k}\|\nonumber\\
&\;+\left[\frac{2}{3}\cdot\frac{\sqrt{L}}{h}(1+\sqrt{L}c_{0})+\frac{2}{3h}\cdot
\frac{\sqrt{L}}{2}(1+\sqrt{L}c_{0})\right]\cdot\|\xi^{k}\|\cdot
\|\Delta_{t}\eta^{k}\|\nonumber\\
=&\;\frac{3\sqrt{L}}{h^{2}}(2+\sqrt{L}c_{0})\cdot|\eta^{\bar{k}}|_{1}\cdot
\|\Delta_{t}\eta^{k}\|+\frac{\sqrt{L}}{h}(1+\sqrt{L}c_{0})\cdot\|\xi^{k}\|\cdot
\|\Delta_{t}\eta^{k}\|,\quad 1\leqslant k\leqslant l.\label{eqs17}
\end{align}
Moreover, similar to \eqref{eq71}--\eqref{eq72}, it holds
\begin{align}
&-(\Delta_{x}\eta^{\bar{k}},\Delta_{t}\eta^{k})\leqslant \|\Delta_{x}\eta^{\bar{k}}\|\cdot\|\Delta_{t}\eta^{k}\|,\quad 1\leqslant k\leqslant l,\label{eqs18}\\
&(\Delta_{x}\xi^{\bar{k}},\Delta_{t}\eta^{k})\leqslant|\Delta_{x}\eta^{\bar{k}}|_{1}
\cdot|\Delta_{t}\eta^{k}|_{1}+\frac{h^{2}}{12}|\xi^{\bar{k}}|_{1}
\cdot\|\Delta_{t}\xi^{k}\|+\frac{h^{4}}{144}|\Delta_{x}\xi^{\bar{k}}|_{1}
\cdot|\Delta_{t}\xi^{k}|_{1},\quad 1\leqslant k\leqslant l,\label{eqs19}\\
&(\xi^{\bar{k}},\Delta_{t}\eta^{k})=-\frac{1}{4\tau}\left[(|\eta^{k+1}|_{1}^{2}
-|\eta^{k-1}|_{1}^{2})+\frac{h^{2}}{12}(\|\xi^{k+1}\|^{2}
-\|\xi^{k-1}\|^{2})-\frac{h^{4}}{144}(|\xi^{k+1}|_{1}^{2}
-|\xi^{k-1}|_{1}^{2})\right],\nonumber\\
&\qquad\qquad\qquad\qquad\qquad\qquad\qquad\qquad 1\leqslant k\leqslant l.\label{eqs20}
\end{align}
Substituting \eqref{eqs13}, \eqref{eqs15}--\eqref{eqs20} into \eqref{eqs13}, we have
\begin{align}
&\;\|\Delta_{t}\eta^{k}\|^{2}\nonumber\\
\leqslant&\;\mu\left(-|\Delta_{t}\eta^{k}|_{1}^{2}-\frac{h^{2}}{12}\|\Delta_{t}\xi^{k}\|^{2}
+\frac{h^{4}}{144}|\Delta_{t}\xi^{k}|_{1}^{2}\right)\nonumber\\
&\;+\gamma\left[\frac{\sqrt{L}}{2}(2+\sqrt{L}c_{0})\cdot|\eta^{\bar{k}}|_{1}\cdot
\|\Delta_{t}\eta^{k}\|+\frac{\sqrt{L}}{2}(1+\sqrt{L}c_{0})
\times|\eta^{k}|_{1}\cdot\|\Delta_{t}\eta^{k}\|\right]\nonumber\\
&\;+\frac{\gamma h^{2}}{2}\left[\frac{3\sqrt{L}}{h^{2}}(2+\sqrt{L}c_{0})\cdot|\eta^{\bar{k}}|_{1}\cdot
\|\Delta_{t}\eta^{k}\|+\frac{\sqrt{L}}{h}(1+\sqrt{L}c_{0})\cdot\|\xi^{k}\|\cdot
\|\Delta_{t}\eta^{k}\|\right]\nonumber\\
&\;+\kappa\|\Delta_{x}\eta^{\bar{k}}\|\cdot\|\Delta_{t}\eta^{k}\|
+\frac{\kappa h^{2}}{6}\left(|\Delta_{x}\eta^{\bar{k}}|_{1}
\cdot|\Delta_{t}\eta^{k}|_{1}+\frac{h^{2}}{12}|\xi^{\bar{k}}|_{1}
\cdot\|\Delta_{t}\xi^{k}\|+\frac{h^{4}}{144}|\Delta_{x}\xi^{\bar{k}}|_{1}
\cdot|\Delta_{t}\xi^{k}|_{1}\right)\nonumber\\
&\;-\frac{1}{4\tau}\left[(|\eta^{k+1}|_{1}^{2}
-|\eta^{k-1}|_{1}^{2})+\frac{h^{2}}{12}(\|\xi^{k+1}\|^{2}
-\|\xi^{k-1}\|^{2})-\frac{h^{4}}{144}(|\xi^{k+1}|_{1}^{2}
-|\xi^{k-1}|_{1}^{2})\right]\nonumber\\
\leqslant&\;-\frac{\mu h^{2}}{12}\|\Delta_{t}\xi^{k}\|^{2}+\frac{\mu h^{4}}
{144}\cdot\frac{h^{2}}{4}\|\Delta_{t}\xi^{k}\|^{2}+\frac{1}{6}\|\Delta_{t}\eta^{k}\|^{2}
+\frac{3\gamma^{2}L}{8}(2+\sqrt{L}c_{0})^{2}|\eta^{\bar{k}}|_{1}^{2}
+\frac{1}{6}\|\Delta_{t}\eta^{k}\|^{2}\nonumber\\
&\;+\frac{3\gamma^{2}L}{8}(1+\sqrt{L}c_{0})^{2}|\eta^{k}|_{1}^{2}
+\frac{1}{6}\|\Delta_{t}\eta^{k}\|^{2}+\frac{27\gamma^{2}L}{8}(2+\sqrt{L}c_{0})^{2}
\cdot|\eta^{\bar{k}}|_{1}^{2}+\frac{1}{6}\|\Delta_{t}\eta^{k}\|^{2}\nonumber\\
&\;+\frac{3\gamma^{2}h^{2}L}{8}(1+\sqrt{L}c_{0})^{2}\|\xi^{k}\|^{2}
+\frac{1}{6}\|\Delta_{t}\eta^{k}\|^{2}+\frac{3\kappa^{2}}{2}\|\Delta_{x}\eta^{\bar{k}}\|^{2}
+\frac{1}{6}\cdot\frac{h^{2}}{4}|\Delta_{t}\eta^{k}|_{1}^{2}+\frac{\kappa^{2}h^{2}}{6}
|\Delta_{x}\eta^{\bar{k}}|_{1}^{2}\nonumber\\
&\;+\frac{\mu h^{2}}{36}\|\Delta_{t}\xi^{k}\|^{2}+\frac{\kappa^{2}h^{6}}{576\mu}
|\xi^{\bar{k}}|_{1}^{2}+\frac{\mu h^{2}}{36}\cdot\frac{h^{2}}{4}
|\Delta_{t}\xi^{k}|_{1}^{2}+\frac{\kappa^{2}h^{8}}{144^{2}\mu}
|\Delta_{x}\xi^{\bar{k}}|_{1}^{2}\nonumber\\
&\;-\frac{1}{4\tau}\left[(|\eta^{k+1}|_{1}^{2}
-|\eta^{k-1}|_{1}^{2})+\frac{h^{2}}{12}(\|\xi^{k+1}\|^{2}
-\|\xi^{k-1}\|^{2})-\frac{h^{4}}{144}(|\xi^{k+1}|_{1}^{2}
-|\xi^{k-1}|_{1}^{2})\right]\nonumber\\
\leqslant&\;\|\Delta_{t}\eta^{k}\|^{2}+\left[\frac{3\gamma^{2}L}{8}(2+\sqrt{L}c_{0})^{2}
+\frac{27\gamma^{2}L}{8}(2+\sqrt{L}c_{0})^{2}
+\frac{3\kappa^{2}}{2}+\frac{2\kappa^{2}}{3}\right]|\eta^{\bar{k}}|_{1}^{2}\nonumber\\
&\;
+\frac{3\gamma^{2}L}{8}(1+\sqrt{L}c_{0})^{2}|\eta^{k}|_{1}^{2}+\frac{3\gamma^{2}h^{2}L}{8}(1+\sqrt{L}c_{0})^{2}
\|\xi^{k}\|^{2}+\left(\frac{\kappa^{2}h^{4}}{144\mu}+\frac{\kappa^{2}h^{4}}{1296\mu}\right)
\|\xi^{\bar{k}}\|^{2}\nonumber\\
&\;-\frac{1}{4\tau}\left[(|\eta^{k+1}|_{1}^{2}
-|\eta^{k-1}|_{1}^{2})+\frac{h^{2}}{12}(\|\xi^{k+1}\|^{2}
-\|\xi^{k-1}\|^{2})-\frac{h^{4}}{144}(|\xi^{k+1}|_{1}^{2}
-|\xi^{k-1}|_{1}^{2})\right]\nonumber\\
\leqslant&\;\|\Delta_{t}\eta^{k}\|^{2}+\left[\frac{15\gamma^{2}L}{4}(2+\sqrt{L}c_{0})^{2}
+\frac{13\kappa^{2}}{6}\right]\cdot\frac{|\eta^{k+1}|_{1}^{2}+|\eta^{k-1}|_{1}^{2}}{2}
+\frac{3\gamma^{2}L}{8}(1+\sqrt{L}c_{0})^{2}\cdot|\eta^{k}|_{1}^{2}\nonumber\\
&\;+\frac{3\gamma^{2}h^{2}L}{8}(1+\sqrt{L}c_{0})^{2}\cdot
\|\xi^{k}\|^{2}+\frac{5\kappa^{2}h^{4}}{648\mu}\cdot\frac{\|\xi^{k+1}\|^{2}
+\|\xi^{k-1}\|^{2}}{2}\nonumber\\
&\;-\frac{1}{4\tau}\left[(|\eta^{k+1}|_{1}^{2}
-|\eta^{k-1}|_{1}^{2})+\frac{h^{2}}{12}(\|\xi^{k+1}\|^{2}
-\|\xi^{k-1}\|^{2})-\frac{h^{4}}{144}(|\xi^{k+1}|_{1}^{2}
-|\xi^{k-1}|_{1}^{2})\right],\quad 1\leqslant k\leqslant l.\label{eqs21}
\end{align}
Simplifying the formula \eqref{eqs21}, we have
\begin{align*}
&\frac{\nu}{2\tau}\left[\frac{1}{2}(|\eta^{k+1}|_{1}^{2}+|\eta^{k}|_{1}^{2})+\frac{h^{2}}{24}
(\|\xi^{k+1}\|^{2}+\|\xi^{k}\|^{2})-\frac{h^{4}}{288}(|\xi^{k+1}|_{1}^{2}+|\xi^{k}|
_{1}^{2})\right]\nonumber\\
&-\frac{\nu}{2\tau}\left[\frac{1}{2}(|\eta^{k}|_{1}^{2}+|\eta^{k-1}|_{1}^{2})
+\frac{h^{2}}{24}(\|\xi^{k}\|^{2}+\|\xi^{k-1}\|^{2})-\frac{h^{4}}{288}(|\xi^{k}|_{1}^{2}
+|\xi^{k-1}|_{1}^{2})\right]\nonumber\\
\leqslant&\;\left[\frac{15\gamma^{2}L}{4}(2+\sqrt{L}c_{0})^{2}
+\frac{13\kappa^{2}}{6}\right]\cdot\frac{|\eta^{k+1}|_{1}^{2}+|\eta^{k-1}|_{1}^{2}}{2}
+\frac{3\gamma^{2}L}{8}(1+\sqrt{L}c_{0})^{2}\cdot|\eta^{k}|_{1}^{2}\nonumber\\
&\;+\frac{3\gamma^{2}h^{2}L}{8}(1+\sqrt{L}c_{0})^{2}\cdot
\|\xi^{k}\|^{2}+\frac{5\kappa^{2}h^{4}}{648\mu}\cdot\frac{\|\xi^{k+1}\|^{2}
+\|\xi^{k-1}\|^{2}}{2}\nonumber\\
\leqslant&\;\left[\frac{15\gamma^{2}L}{4}(2+\sqrt{L}c_{0})^{2}
+\frac{13\kappa^{2}}{6}+\frac{3\gamma^{2}L}{8}(1+\sqrt{L}c_{0})^{2}\right]\cdot
\left(\frac{|\eta^{k+1}|_{1}^{2}+|\eta^{k}|_{1}^{2}}{2}+\frac{|\eta^{k}|_{1}^{2}
+|\eta^{k-1}|_{1}^{2}}{2}\right)\nonumber\\
&\;+\left[\frac{18}{h^{2}}\cdot\frac{3\gamma^{2}h^{2}L}{8}(1+\sqrt{L}c_{0})^{2}
+\frac{18}{h^{2}}\cdot\frac{5\kappa^{2}h^{4}}{648\mu}\right]\cdot
\left[\frac{h^{2}}{36}\left(\|\xi^{k+1}\|^{2}+\|\xi^{k}\|^{2}\right)
+\frac{h^{2}}{36}\left(\|\xi^{k}\|^{2}+\|\xi^{k-1}\|^{2}\right)\right],\\
=&\;\left[\frac{15\gamma^{2}L}{4}(2+\sqrt{L}c_{0})^{2}
+\frac{13\kappa^{2}}{6}+\frac{3\gamma^{2}L}{8}(1+\sqrt{L}c_{0})^{2}\right]\cdot
\left(\frac{|\eta^{k+1}|_{1}^{2}+|\eta^{k}|_{1}^{2}}{2}+\frac{|\eta^{k}|_{1}^{2}
+|\eta^{k-1}|_{1}^{2}}{2}\right)\nonumber\\
&\;+\left[\frac{27\gamma^{2}L}{4}(1+\sqrt{L}c_{0})^{2}
+\frac{5\kappa^{2}h^{2}}{36\mu}\right]\cdot
\left[\frac{h^{2}}{36}\left(\|\xi^{k+1}\|^{2}+\|\xi^{k}\|^{2}\right)
+\frac{h^{2}}{36}\left(\|\xi^{k}\|^{2}+\|\xi^{k-1}\|^{2}\right)\right],\\
&\qquad\qquad\qquad\qquad\qquad\qquad\qquad\qquad\qquad 1\leqslant k\leqslant l.
\end{align*}
when $h\leqslant h_{0}$, $\tau\leqslant\tau_{0}$, we have
\begin{align*}
&\frac{\nu}{2\tau}\left[\frac{1}{2}(|\eta^{k+1}|_{1}^{2}+|\eta^{k}|_{1}^{2})+\frac{h^{2}}{24}
(\|\xi^{k+1}\|^{2}+\|\xi^{k}\|^{2})-\frac{h^{4}}{288}(|\xi^{k+1}|_{1}^{2}+|\xi^{k}|
_{1}^{2})\right]\nonumber\\
&-\frac{\nu}{2\tau}\left[\frac{1}{2}(|\eta^{k}|_{1}^{2}+|\eta^{k-1}|_{1}^{2})
+\frac{h^{2}}{24}(\|\xi^{k}\|^{2}+\|\xi^{k-1}\|^{2})-\frac{h^{4}}{288}(|\xi^{k}|_{1}^{2}
+|\xi^{k-1}|_{1}^{2})\right]\nonumber\\
\leqslant&\;c_{14}\left(\frac{|\eta^{k+1}|_{1}^{2}+|\eta^{k}|_{1}^{2}}{2}+\frac{|\eta^{k}|_{1}^{2}
+|\eta^{k-1}|_{1}^{2}}{2}\right)+c_{15}\left[\frac{h^{2}}{36}\left(\|\xi^{k+1}\|^{2}
+\|\xi^{k}\|^{2}\right)+\frac{h^{2}}{36}\left(\|\xi^{k}\|^{2}+\|\xi^{k-1}\|^{2}\right)\right]\nonumber\\
\leqslant&\;c_{16}\left[\frac{|\eta^{k+1}|_{1}^{2}+|\eta^{k}|_{1}^{2}}{2}+\frac{h^{2}}{36}
\left(\|\xi^{k+1}\|^{2}+\|\xi^{k}\|^{2}\right)\right]+c_{16}\left[\frac{|\eta^{k}|_{1}^{2}
+|\eta^{k-1}|_{1}^{2}}{2}+\frac{h^{2}}{36}\left(\|\xi^{k}\|^{2}+\|\xi^{k-1}\|^{2}
\right)\right],\\
&\qquad\qquad\qquad\qquad\qquad\qquad\qquad\qquad\qquad 1\leqslant k\leqslant l.
\end{align*}
Therefore, we have
\begin{align*}
\frac{\nu}{2\tau}(G^{k+1}-G^{k})\leqslant c_{16}(G^{k+1}+G^{k}),\quad 1\leqslant k\leqslant l.
\end{align*}
According to the Gronwall inequality, when $2c_{16}\tau/\nu\leqslant 1/3$, applying \eqref{eqs11}, we have
\begin{align*}
G^{k+1}\leqslant\exp\left(\frac{6Tc_{16}}{\nu}\right)G^{1}\leqslant\exp
\left(\frac{6Tc_{16}}{\nu}\right)\cdot\frac{4c_{13}}{3}|\phi^{0}|_{1}^{2}= c_{17}|\phi^{0}|_{1}^{2},\quad 1\leqslant k\leqslant l.
\end{align*}
Thus we have
\begin{align*}
|\eta^{k+1}|_{1}\leqslant\sqrt{2c_{17}}|\phi^{0}|_{1},\quad 1\leqslant k\leqslant l.
\end{align*}
By the mathematical induction, we have
\begin{align}
|\eta^{k+1}|_{1}\leqslant c_{11}|\phi^{0}|_{1},\quad 1\leqslant k\leqslant N-1.\label{eqs22}
\end{align}
Combining \eqref{eqs22} with \eqref{eqs12}, we have
\begin{align*}
|\eta^{k}|_{1}\leqslant c_{11}|\phi^{0}|_{1},\quad 0\leqslant k\leqslant N.
\end{align*}
  \end{proof}
\section{Numerical Experiments}
\label{Sec:7}
\setcounter{equation}{0}
In the section, we will implement several numerical examples to verify the effectiveness of our scheme and the correctness of theoretical results.

When the exact solution is known, we define the discrete error in the $L^{\infty}$-norm as follows
\begin{align*}
{\rm E}_{\infty}(h,\tau)=\max_{1\leqslant i\leqslant M,\;0\leqslant k\leqslant N}|U_{i}^{k}-u_{i}^{k}|,
\end{align*}
where $U_{i}^{k}$ and $u_{i}^{k}$ represent the exact solution and the numerical solution, respectively.
Furthermore, denote the spatial and temporal convergence orders, respectively, as
\begin{align*}
\mathrm{Order_{\infty}^{h}}=\log_{2}\frac{{\rm E}_{\infty}(2h,\tau)}{{\rm E}_{\infty}(h,\tau)},\quad
\mathrm{Order_{\infty}^{\tau}}=\log_{2}\frac{{\rm E}_{\infty}(h,2\tau)}{{\rm E}_{\infty}(h,\tau)}.
\end{align*}

When the exact solution is unknown, we use the posterior error estimation to testify the convergence orders in
temporal direction and spatial direction, respectively.
For sufficient small $h$, we denote
\begin{align*}
{\rm F}_{\infty}(h,\tau)=\max\limits_{1\leqslant i\leqslant M, \; 0\leqslant k\leqslant N}|u_{i}^{k}(h,\tau)-u_{2i}^{k}(h/2,\tau)|,
\quad \mathrm{Order}_{\infty}^{h}=\log_{2}\left(\frac{{\rm F}_{\infty}(2h,\tau)}{{\rm F}_{\infty}(h,\tau)}\right),
\end{align*}
and for sufficient small $\tau$, we denote
 \begin{align*}
{\rm G}_{\infty}(h,\tau)=\max\limits_{0\leqslant i\leqslant M, \; 0\leqslant k\leqslant N}|u_{i}^{k}(h,\tau)-u_{i}^{2k}(h,\tau/2)|,\quad \mathrm{Order}_{\infty}^{\tau}=\log_{2}\left(\frac{{\rm G}_{\infty}(h,2\tau)}{{\rm G}_{\infty}(h,\tau)}\right).
\end{align*}

\begin{example}\label{example1}
We first consider the following {\rm BBMB} equation $($see {\rm \cite{OO2017}}$)$
\begin{align*}
&u_{t}-u_{xxt}+uu_{x}+u_{x}-u_{xx}=f(x,t),\quad 0<x<2,\;0<t\leqslant 1,
\end{align*}
where
\begin{align*}
f(x,t)=(1+2\pi^{2})e^{t}\sin\pi x+\frac{\pi}{2} e^{2t}\sin2\pi x+\pi e^{t}\cos\pi x.
\end{align*}
The initial condition is determined by the exact solution $u(x,t)=e^{t}\sin\pi x$ with the period $L=2$.
\end{example}

The numerical results are reported in Tables \ref{table1}--\ref{table2} and Figures \ref{fig:1}--\ref{fig:2}.

In Table \ref{table1}, we fix the temporal step-size $\tau=1/5000$, meanwhile,
reduce the spatial step-size $h$ half by half $(h=1/4$, $1/8$, $1/16$, $1/32$, $1/64)$.
As we can see, the spatial convergence order approaches to four order approximately,
which is consistent with our convergence results.

In Table \ref{table2}, we fix the spatial step-size $h=1/50$, meanwhile,
reduce the temporal step-size $\tau$ half by half $(\tau=1/20$, $1/40$, $1/80$, $1/160$, $1/320)$.
We observe that the temporal convergence order approaches to two order in maximum norm.

Compared our numerical results with those in \cite{MF2017} from Table \ref{table1} and Table \ref{table2},
we find our scheme is more efficient and accurate.

Moreover, in order to verify the stability of the difference scheme \eqref{eq17}--\eqref{eq20a},
we have drawn the stable error curves in Figure \ref{fig:1}. For each curve, we fixed different
temporal step-size $(\tau=1/8$, $1/16$, $1/32$, $1/64$, $1/128)$ by reducing the spatial step-size
$h$ half by half $(h=1/2$, $1/4$, $1/8$, $1/16$, $1/32$, $1/64)$. We observe that the spatial error in maximum norm
approaches to a fixed value since the numerical errors mainly come from the discretization in time,
which verifies the difference scheme \eqref{eq17}--\eqref{eq20a} is almost unconditional stable.
In Figure \ref{fig:2}, the numerical panorama for $u(x,t)$ and numerical profiles are displayed,
which further demonstrate the high accuracy of our scheme in practical simulation.

\begin{table}[tbh!]
\begin{center}
\renewcommand{\arraystretch}{1.12}
\tabcolsep 0pt \caption{Maximum norm errors behavior versus $h$-grid size reduction with the fixed
temporal step-size $\tau=1/5000$ in Example \ref{example1} }\label{table1}
\def\temptablewidth{0.9\textwidth}
\rule{\temptablewidth}{1pt}
{\footnotesize
\begin{tabular*}{\temptablewidth}{@{\extracolsep{\fill}}ccccc}
&\multicolumn{2}{c}{~~difference scheme \eqref{eq17}--\eqref{eq20a}}&\multicolumn{2}{c}{~~difference scheme in \cite{MF2017}}\\
\cline{2-3}\cline{4-5}
$\quad h$&$\quad  {\rm E}_{\infty}(h,\tau)$&${\rm Order}_{\infty}^h$ &$\quad  {\rm E}_{\infty}(h,\tau)$&${\rm Order}_{\infty}^h$\\
\hline
$\quad  1/ 4$   & $\quad  9.0677{\rm e}-3$  & $ *$          & $\quad  1.8968{\rm e}-2$   & $ *$       \\
$\quad  1/ 8$   & $\quad  5.9120{\rm e}-4$  & $ 3.9390$     & $\quad  1.3213{\rm e}-3$   & $3.8436$   \\
$\quad  1/ 16$  & $\quad  3.7491{\rm e}-5$  & $ 3.9790$     & $\quad  8.7856{\rm e}-5$   & $3.9107$   \\
$\quad  1/ 32$  & $\quad  2.3538{\rm e}-6$  & $ 3.9935$     & $\quad  5.5096{\rm e}-6$   & $3.9951$   \\
$\quad  1/ 64$  & $\quad  1.2326{\rm e}-7$  & $ 4.2552$     & $\quad  3.1792{\rm e}-7$   & $4.1152$
\end{tabular*}}
\rule{\temptablewidth}{1pt}
\end{center}
\end{table}

\begin{table}[tbh!]
\begin{center}
\renewcommand{\arraystretch}{1.12}
\tabcolsep 0pt \caption{Maximum norm errors behavior versus $\tau$-grid size reduction with the fixed spatial step-size $h=1/50$ in Example \ref{example1}  }\label{table2}
\def\temptablewidth{0.9\textwidth}
\rule{\temptablewidth}{1pt}
{\footnotesize
\begin{tabular*}{\temptablewidth}{@{\extracolsep{\fill}}ccccc}
&\multicolumn{2}{c}{~~difference scheme \eqref{eq17}--\eqref{eq20a}}&\multicolumn{2}{c}{~~difference scheme in \cite{MF2017}}\\
\cline{2-3}\cline{4-5}
$\quad \tau$&$\quad  {\rm E}_{\infty}(h,\tau)$&${\rm Order}_{\infty}^{\tau}$&$\quad  {\rm E}_{\infty}(h,\tau)$&${\rm Order}_{\infty}^{\tau}$\\
\hline
$\quad  1/ 20$  & $\quad  1.9486{\rm e}-3$  & $ *$          & $\quad  2.3772{\rm e}-3$   & $ *$   \\
$\quad  1/ 40$  & $\quad  4.8670{\rm e}-4$  & $ 2.0013 $    & $\quad  6.0794{\rm e}-4$   & $1.9673$   \\
$\quad  1/ 80$  & $\quad  1.2144{\rm e}-4$  & $ 2.0028 $    & $\quad  1.5342{\rm e}-4$   & $1.9865$   \\
$\quad  1/ 160$ & $\quad  3.0162{\rm e}-5$  & $ 2.0094 $    & $\quad  3.8242{\rm e}-5$   & $2.0042$  \\
$\quad  1/ 320$ & $\quad  7.3615{\rm e}-6$  & $ 2.0346 $    & $\quad  9.2928{\rm e}-6$   & $2.0410$
\end{tabular*}}
\rule{\temptablewidth}{1pt}
\end{center}
\end{table}

\begin{figure}[htbp!]
\centering
\includegraphics[width=0.7\textwidth]{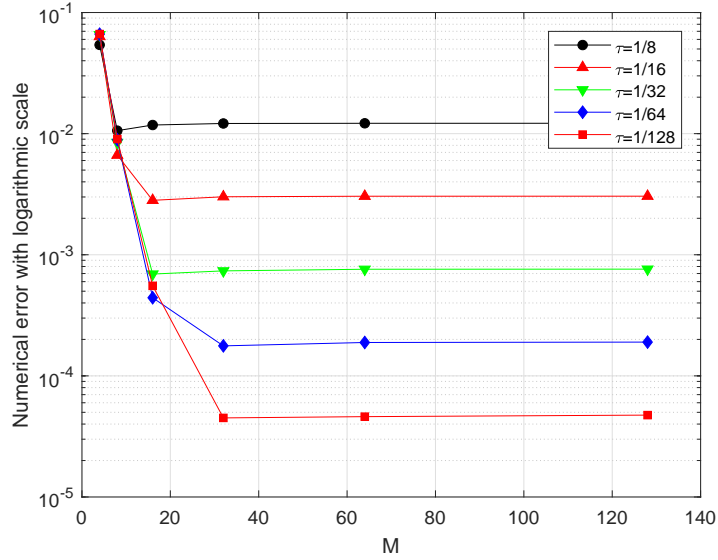}
 \caption{Numerical stability test chart} \label{fig:1}
\end{figure}

\begin{figure}[htbp]
 \subfigure[The numerical panorama for $u(x,t)$ ]{\centering
\includegraphics[width=0.5\textwidth]{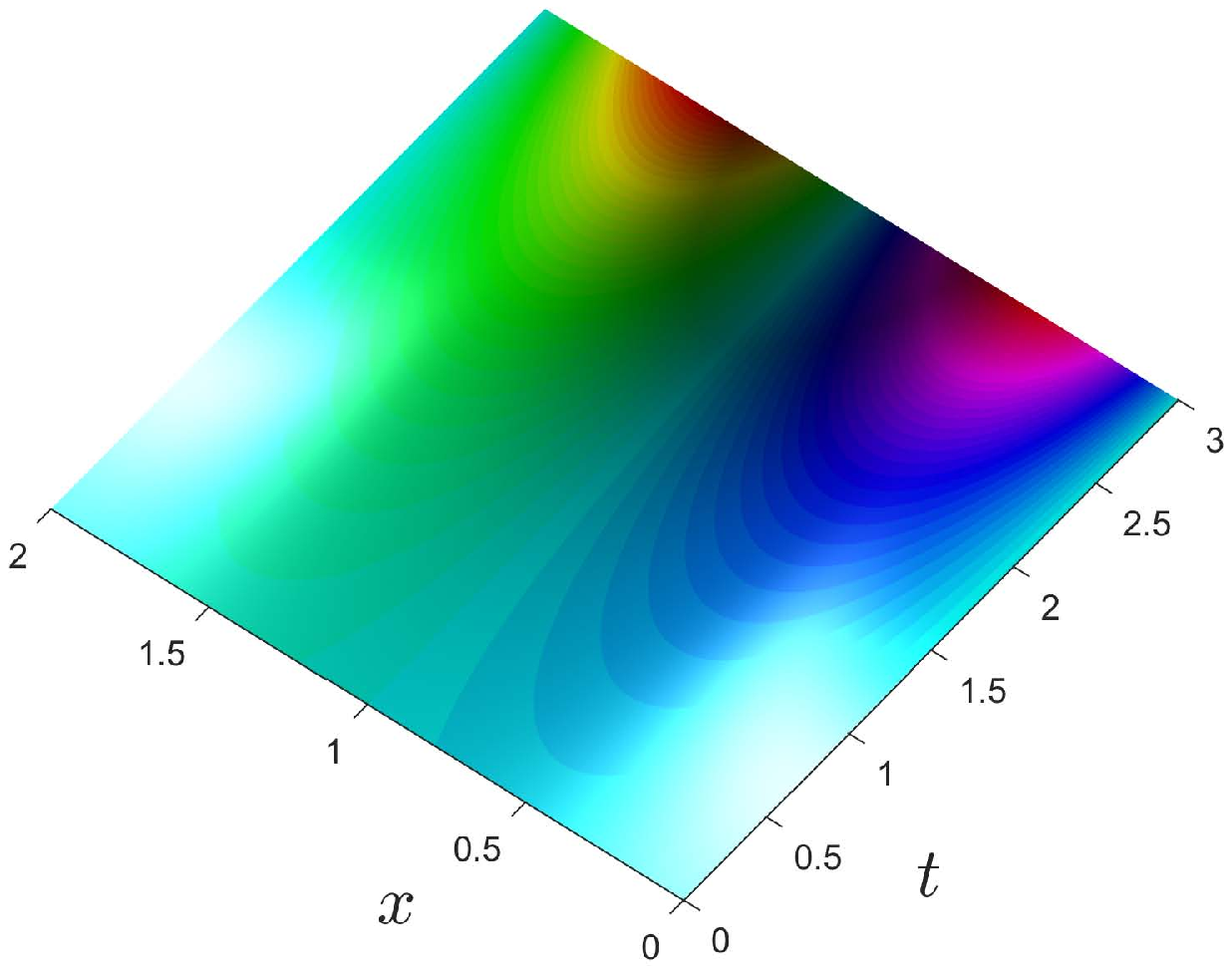}
}\subfigure[Numerical solution profiles]{\centering
\includegraphics[width=0.5\textwidth]{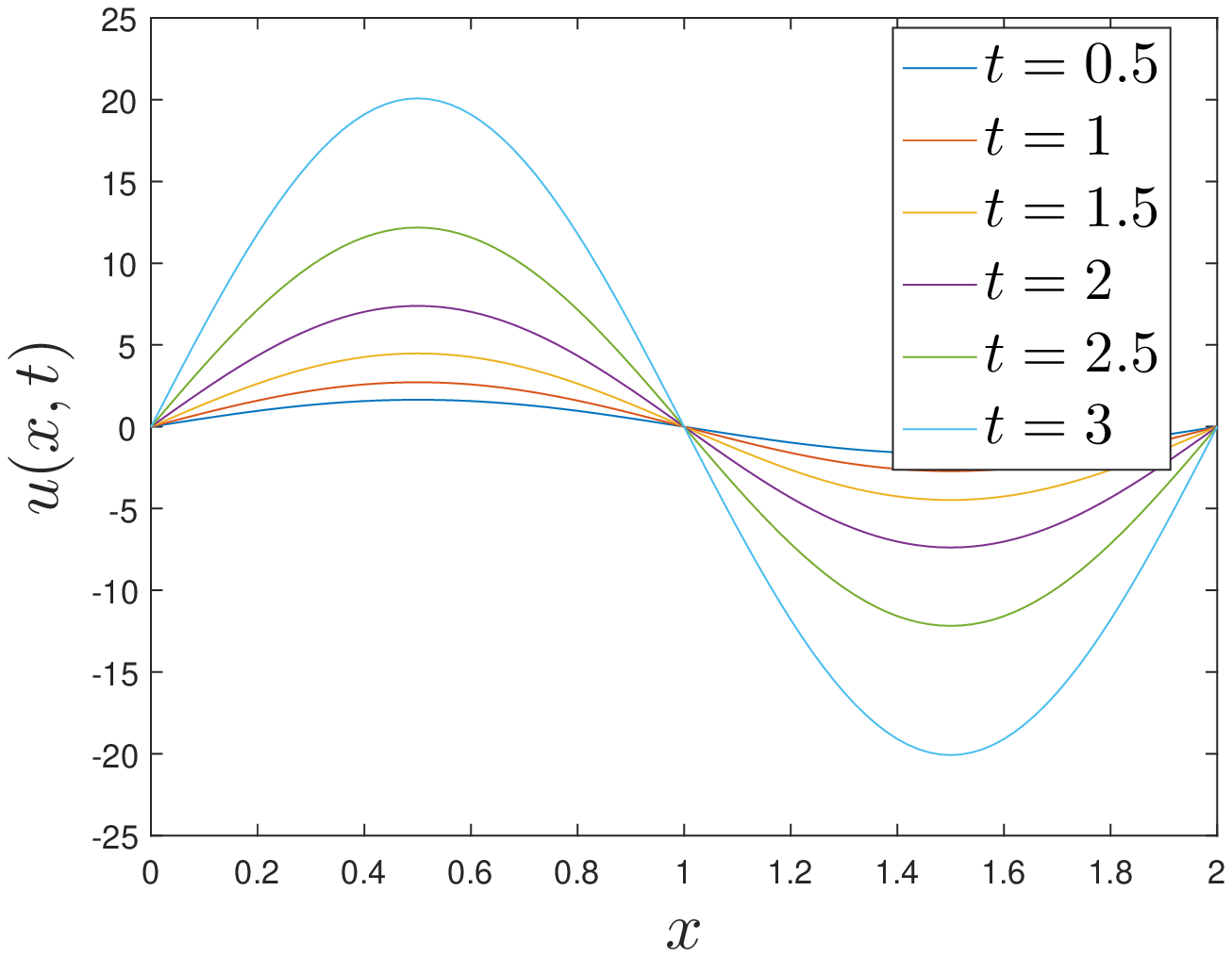}
}
\caption{(a) The numerical solution, (b) the solution profiles for $u(x,t)$ with $t=0.5$, $1$, $1.5$, $2$, $2.5$, $3$} \label{fig:2}
\end{figure}

\begin{example}\label{example2}
Then, we consider the following {\rm BBMB} equation
\begin{align*}
&u_{t}-\mu u_{xxt}+uu_{x}+u_{x}-\nu u_{xx}=0,\quad -25<x<25,\; 0<t\leqslant 1,
\end{align*}
with the initial condition
\begin{align*}
u(x,0)=\frac{1}{2}{\rm sech}^{2}\left(\frac{x}{4}\right),\quad -25\leqslant x\leqslant 25,
\end{align*}
where the exact solution is unknown and the period $L=50$.
\end{example}

The numerical results are showed in Tables \ref{table3}--\ref{table5} and Figure \ref{fig3} with $\mu=1$ and $\nu=1$ .

Firstly, we fix the temporal step-size $\tau=1/2000$, in the meantime, decrease
the spatial step-size $h$ half by half ($M=20,40,80,160,320,640$).
As we can see from Table \ref{table3}, the spatial convergence orders approach to fourth order for both schemes.
However, our scheme is more accurate than that in the reference \cite{MF2017}.

Next, we fix the spatial step-size $h=1/2$, and then reduce the temporal step-size $\tau$ half by half.
The maximum norm error and the temporal convergence orders are listed in Table \ref{table4}.
The temporal convergence order approaches to $\mathcal{O}(\tau^{2})$ approximately. However, the difference scheme in \cite{MF2017}
is less than two and the accuracy is far from enough. We further refine the spatial grid (fixed step size $h=1/100$) and decrease the temporal step-size $\tau$ half by half again in Table \ref{table4b}, though both schemes can achieve orders two, our scheme \eqref{eq17}--\eqref{eq20a} is still better than that in \cite{MF2017} with respect to the accuracy.
Combining Tables \ref{table4} and \ref{table4b}, we conclude that our scheme is more robust and stable that the scheme in \cite{MF2017}, which illustrates the superiority of our scheme.

To further verify the performance of the numerical scheme \eqref{eq17}--\eqref{eq20a} more rigorously,
we test the energy conservation invariants \eqref{E} with different $\mu$ and $\nu$.
The conservation invariants of $E^n$ at different time are demonstrated in Table \ref{table5}.
It is easy to see from Table \ref{table5} that the three-point four-order compact difference scheme can keep the
conservative invariant even for the very small parameters, which demonstrate that our numerical scheme is stable and robust.


\begin{table}[tbh!]
\begin{center}
\renewcommand{\arraystretch}{1.12}
\tabcolsep 0pt \caption{Maximum norm errors behavior versus $h$-grid size reduction with the fixed temporal step-size $\tau=1/2000$ in Example \ref{example2} }\label{table3}
\def\temptablewidth{0.9\textwidth}
\rule{\temptablewidth}{1pt}
{\footnotesize
\begin{tabular*}{\temptablewidth}{@{\extracolsep{\fill}}ccccc}
&\multicolumn{2}{c}{~~difference scheme \eqref{eq17}--\eqref{eq20a}}&\multicolumn{2}{c}{~~difference scheme in \cite{MF2017}}\\
\cline{2-3}\cline{4-5}
$\quad h$&$\quad  {\rm F}_{\infty}(h,\tau)$&${\rm Order}_{\infty}^h$&$\quad  {\rm F}_{\infty}(h,\tau)$&${\rm Order}_{\infty}^h$\\
\hline
$\quad  5/ 4$      & $\quad  4.2025{\rm e}-4$  & $ *$           &$\quad  8.1169{\rm e}-3$  &$*$         \\
$\quad  5/ 8$      & $\quad  3.2284{\rm e}-5$  & $ 3.7024$      &$\quad  7.6856{\rm e}-4$  &$3.4007$    \\
$\quad  5/ 16$     & $\quad  2.0457{\rm e}-6$  & $ 3.9801$      &$\quad  5.7840{\rm e}-5$  &$3.7320$   \\
$\quad  5/ 32$     & $\quad  1.2833{\rm e}-7$  & $ 3.9946$      &$\quad  3.8500{\rm e}-6$  &$3.9091$  \\
$\quad  5/ 64$     & $\quad  7.9715{\rm e}-9$  & $ 4.0089$      &$\quad  2.4304{\rm e}-7$  &$3.9856$
\end{tabular*}}
\rule{\temptablewidth}{1pt}
\end{center}
\end{table}

\begin{table}[tbh!]
\begin{center}
\renewcommand{\arraystretch}{1.12}
\tabcolsep 0pt \caption{Maximum norm errors behavior versus $\tau$-grid size reduction
with the fixed spatial step-size $h=1/2\;(M=100)$ in Example \ref{example2}}\label{table4}
\def\temptablewidth{0.9\textwidth}
\rule{\temptablewidth}{1pt}
{\footnotesize
\begin{tabular*}{\temptablewidth}{@{\extracolsep{\fill}}ccccc}
&\multicolumn{2}{c}{~~difference scheme \eqref{eq17}--\eqref{eq20a}}&\multicolumn{2}{c}{~~difference scheme in \cite{MF2017}}\\
\cline{2-3}\cline{4-5}
$\quad \tau$&$\quad  {\rm G}_{\infty}(h,\tau)$&${\rm Order}_{\infty}^{\tau}$&$\quad  {\rm G}_{\infty}(h,\tau)$&${\rm Order}_{\infty}^{\tau}$\\
\hline
$\quad  1/ 20$  & $\quad  2.1054{\rm e}-5$   & $*$              &  $\quad  3.8007{\rm e}-4$  & $*$   \\
$\quad  1/ 40$  & $\quad  5.4491{\rm e}-6$   & $ 1.9500$        &  $\quad  9.9663{\rm e}-5$  & $1.9311$   \\
$\quad  1/ 80$  & $\quad  1.3852{\rm e}-6$   & $ 1.9759$        &  $\quad  3.0017{\rm e}-5$  & $1.7313$   \\
$\quad  1/ 160$ & $\quad  3.4915{\rm e}-7$   & $ 1.9882$        &  $\quad  1.3680{\rm e}-5$  & $1.1337$   \\
$\quad  1/ 320$ & $\quad  8.7641{\rm e}-8$   & $ 1.9942$        &  $\quad  1.0438{\rm e}-5$  & $0.3903$
\end{tabular*}}
\rule{\temptablewidth}{1pt}
\end{center}
\end{table}

\begin{table}[tbh!]
\begin{center}
\renewcommand{\arraystretch}{1.12}
\tabcolsep 0pt \caption{Maximum norm errors behavior versus $\tau$-grid size reduction
with the refined spatial step-size $h=1/100\;(M=5000)$ in Example \ref{example2}}\label{table4b}
\def\temptablewidth{0.9\textwidth}
\rule{\temptablewidth}{1pt}
{\footnotesize
\begin{tabular*}{\temptablewidth}{@{\extracolsep{\fill}}ccccc}
&\multicolumn{2}{c}{~~difference scheme \eqref{eq17}--\eqref{eq20a}}&\multicolumn{2}{c}{~~difference scheme in \cite{MF2017}}\\
\cline{2-3}\cline{4-5}
$\quad \tau$&$\quad  {\rm G}_{\infty}(h,\tau)$&${\rm Order}_{\infty}^{\tau}$&$\quad  {\rm G}_{\infty}(h,\tau)$&${\rm Order}_{\infty}^{\tau}$\\
\hline
$\quad  1/ 10$  & $\quad  3.0686e-4$   & $*$              &  $\quad  1.4755{\rm e}-3$  & $*$   \\
$\quad  1/ 20$  & $\quad  7.8967e-5$   & $ 1.9583$        &  $\quad  3.7422{\rm e}-4$  & $1.9793$   \\
$\quad  1/ 40$  & $\quad  2.1227e-5$   & $ 1.8954$        &  $\quad  9.4215{\rm e}-5$  & $1.9898$   \\
$\quad  1/ 80$  & $\quad  5.4878e-6$   & $ 1.9516$        &  $\quad  2.3636{\rm e}-5$  & $1.9950$   \\
$\quad  1/ 160$ & $\quad  1.3945e-6$   & $ 1.9765$        &  $\quad  5.9192{\rm e}-6$  & $1.9975$
\end{tabular*}}
\rule{\temptablewidth}{1pt}
\end{center}
\end{table}

\begin{table}[tbh!]
\begin{center}
\renewcommand{\arraystretch}{1.12}
\tabcolsep 0pt \caption{Numerical invariants of $E^n$ at time $t$ with $h=1/5$ and $\tau=1/256$ in Example \ref{example2}\label{table5}}
\def\temptablewidth{0.9\textwidth}
\rule{\temptablewidth}{1pt}
{\footnotesize
\begin{tabular*}{\temptablewidth}{@{\extracolsep{\fill}}ccccc}
\quad$t\qquad$  &$(\mu,\nu) = (100,1)\qquad$  &$(\mu,\nu) = (1,1)\qquad$ &$ (\mu,\nu) = (0.01,0.01)\qquad$ &$ (\mu,\nu) = (0.0001,0.0001)$  \qquad \\
\hline
 \quad$0\qquad$   &$7.999997216956726 \qquad$&$1.399999972059210\qquad$ & $1.333999999610235\qquad$ &$1.333339999885745\qquad$    \\
 \quad$2\qquad$   &$7.999997216861070 \qquad$&$1.399999972053378\qquad$ & $1.333999999610103\qquad$ &$1.333339999885731\qquad$    \\
 \quad$4\qquad$   &$7.999997216774900 \qquad$&$1.399999972048419\qquad$ & $1.333999999609973\qquad$ &$1.333339999885733\qquad$    \\
 \quad$6\qquad$   &$7.999997216690209 \qquad$&$1.399999972044242\qquad$ & $1.333999999609848\qquad$ &$1.333339999885742\qquad$    \\
 \quad$8\qquad$   &$7.999997216644139 \qquad$&$1.399999972040209\qquad$ & $1.333999999609709\qquad$ &$1.333339999885743\qquad$   \\
\end{tabular*}}
\rule{\temptablewidth}{1pt}
\end{center}
\end{table}

\begin{figure}[htbp]
 \subfigure[The numerical panorama for $u(x,t)$ ]{\centering
\includegraphics[width=0.5\textwidth]{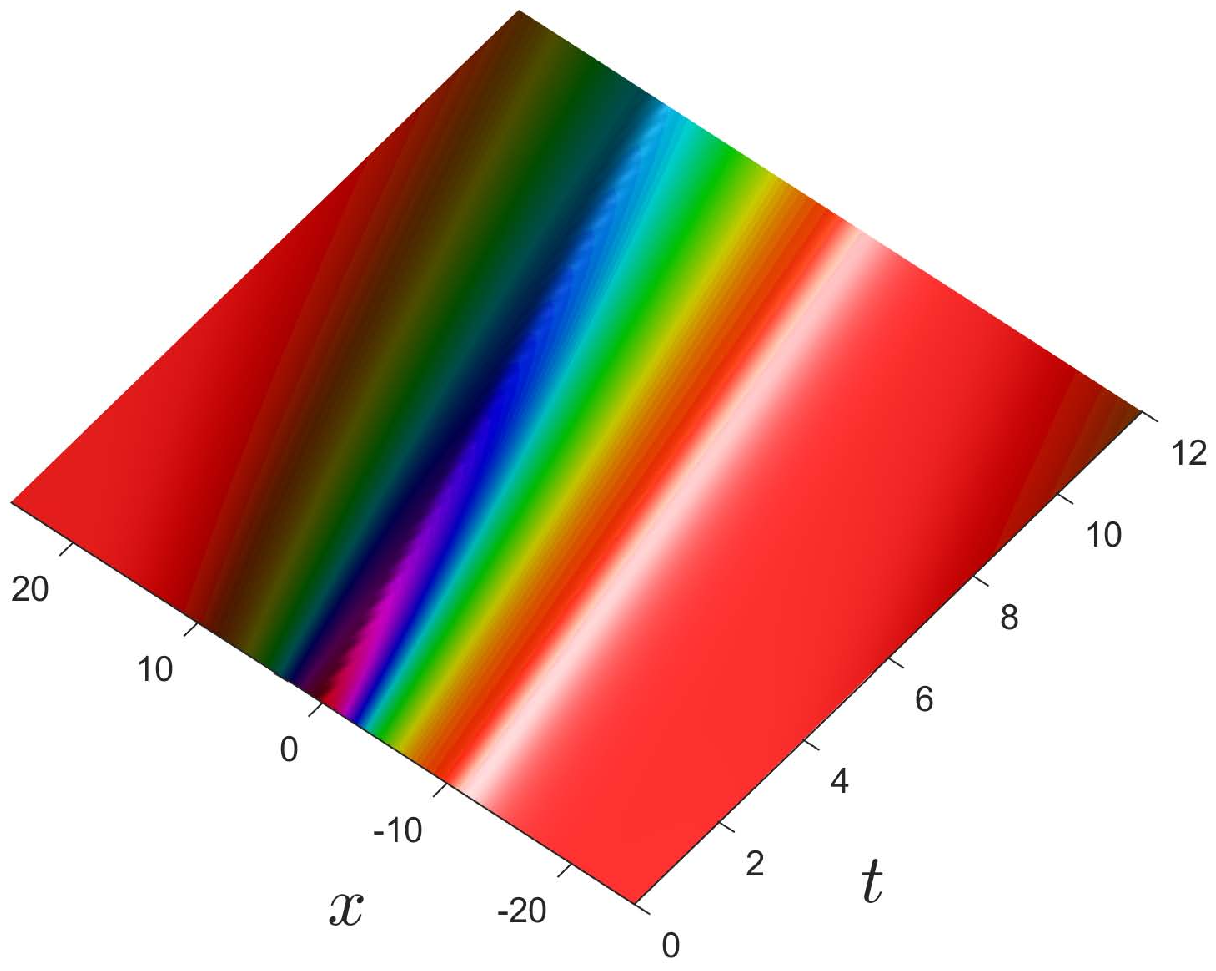}
}\subfigure[Numerical solution profiles]{\centering
\includegraphics[width=0.5\textwidth]{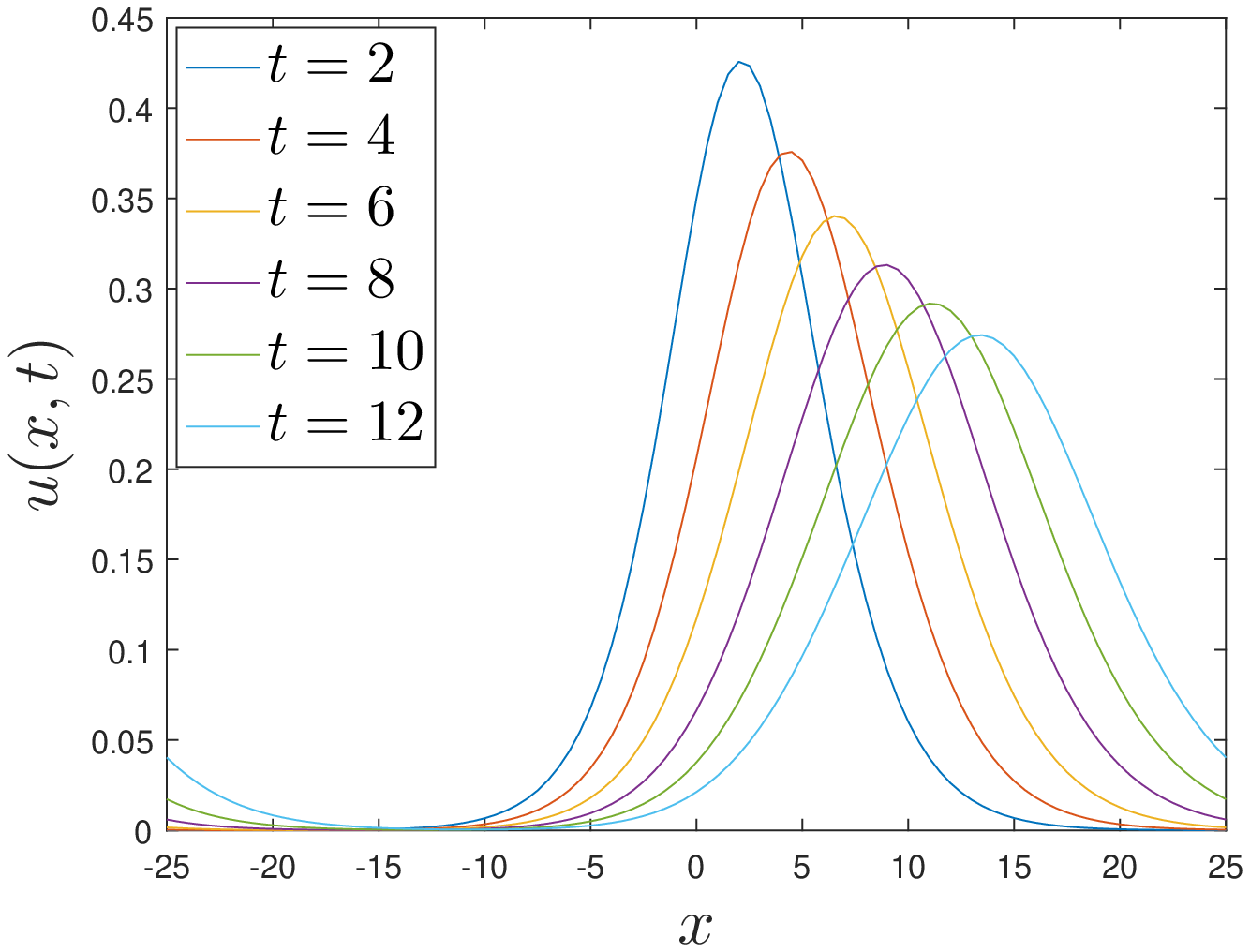}
}
\caption{(a) The numerical solution $t=12$, (b) the solution profiles for $u(x,t)$ with $t=2$, $4$, $6$, $8$, $10$, $12$} \label{fig3}
\end{figure}

\begin{example}\label{example3}
  Finally, we consider a nonlinear {\rm BBMB} equation
   \begin{align*}
&u_{t}-\mu u_{xxt}+\gamma uu_{x}+\kappa u_{x}-\nu u_{xx}+ F'(u)=0,\quad x_l<x<x_r,\;0<t\leqslant T, \\
&u(x,0)=\phi(x),\quad x_l\leqslant x\leqslant x_r,
\end{align*}
where
   $F(u) = 1/4\cdot (1-u^2)^2$, $x_l = -50$, $x_r=50$, $\mu=\gamma=\kappa = \nu =1$. The initial condition is $\phi(x)=\frac{\sqrt{6}}{3}{\rm sech}^2\left(\frac{x}{3}\right)$.
\end{example}
Since the above problem is nonlinear, we use Newton linearized technique (see \cite{ZLZ2020}) for practical implementation.
In order to demonstrate the superiority of the present scheme, we compare it with the numerical result in \cite{ZLZ2020}
with the period boundary condition. The corresponding convergence orders in spatial direction and temporal direction are reported in Tables \ref{table7} and \ref{table8}, respectively. We see from Table \ref{table7} that the numerical errors are better than those in \cite{ZLZ2020} along with the spatial direction.
According to the results in Tables \ref{table7} and \ref{table8}, we know that the convergence orders are two in time and four in space for difference scheme \eqref{eq17}--\eqref{eq20a}, which are consistent with our theoretical results.

The numerical surfaces and the numerical curves are simulated by difference scheme \eqref{eq17}--\eqref{eq20a} in Figures \ref{fig4}-\ref{fig5}.
 We see that the present scheme is much more accurate than that in \cite{ZLZ2020} and clearly depicts the evolutionary process of the solution.

\begin{table}[tbh!]
\begin{center}
\renewcommand{\arraystretch}{1.12}
\tabcolsep 0pt \caption{Maximum norm errors behavior versus $h$-grid size reduction with the fixed temporal step-size $\tau=1/100$ in Example \ref{example3} }\label{table7}
\def\temptablewidth{0.9\textwidth}
\rule{\temptablewidth}{1pt}
{\footnotesize
\begin{tabular*}{\temptablewidth}{@{\extracolsep{\fill}}ccccc}
&\multicolumn{2}{c}{~~difference scheme \eqref{eq17}--\eqref{eq20a}}&\multicolumn{2}{c}{~~difference scheme in \cite{ZLZ2020}}\\
\cline{2-3}\cline{4-5}
$\quad h$&$\quad  {\rm F}_{\infty}(h,\tau)$&${\rm Order}_{\infty}^h$&$\quad  {\rm F}_{\infty}(h,\tau)$&${\rm Order}_{\infty}^h$\\
\hline
$\quad  1/10$      & $\quad  1.7767{\rm e}-5$  & $ *$           &$\quad  5.2399{\rm e}-4$  &$*$         \\
$\quad  1/20$      & $\quad  1.1166{\rm e}-6$  & $ 3.9920$      &$\quad  1.3105{\rm e}-4$  &$1.9994$    \\
$\quad  1/40$      & $\quad  6.9998{\rm e}-8$  & $ 3.9957$      &$\quad  3.2769{\rm e}-5$  &$1.9997$   \\
$\quad  1/80$      & $\quad  4.4432{\rm e}-9$  & $ 3.9776$      &$\quad  8.1942{\rm e}-6$  &$1.9996$  \\
$\quad  1/160$     & $\quad  3.0875{\rm e}-10$ & $ 3.8471$      &$\quad  2.0495{\rm e}-6$  &$1.9993$
\end{tabular*}}
\rule{\temptablewidth}{1pt}
\end{center}
\end{table}

\begin{table}[tbh!]
\begin{center}
\renewcommand{\arraystretch}{1.12}
\tabcolsep 0pt \caption{Maximum norm errors behavior versus $\tau$-grid size reduction
with the fixed spatial step-size $h=1/100\;(M=10000)$ in Example \ref{example3}}\label{table8}
\def\temptablewidth{0.9\textwidth}
\rule{\temptablewidth}{1pt}
{\footnotesize
\begin{tabular*}{\temptablewidth}{@{\extracolsep{\fill}}ccccc}
&\multicolumn{2}{c}{~~difference scheme \eqref{eq17}--\eqref{eq20a}}&\multicolumn{2}{c}{~~difference scheme in \cite{ZLZ2020}}\\
\cline{2-3}\cline{4-5}
$\quad \tau$&$\quad  {\rm G}_{\infty}(h,\tau)$&${\rm Order}_{\infty}^{\tau}$&$\quad  {\rm G}_{\infty}(h,\tau)$&${\rm Order}_{\infty}^{\tau}$\\
\hline
$\quad  1/ 10$  & $\quad  9.7617{\rm e}-3$   & $*$              &  $\quad  2.4014{\rm e}-3$  & $*$   \\
$\quad  1/ 20$  & $\quad  2.9675{\rm e}-3$   & $ 1.7179$        &  $\quad  6.7036{\rm e}-4$  & $1.8409$   \\
$\quad  1/ 40$  & $\quad  7.5462{\rm e}-4$   & $ 1.9754$        &  $\quad  1.7703{\rm e}-4$  & $1.9209$   \\
$\quad  1/ 80$  & $\quad  1.9088{\rm e}-4$   & $ 1.9831$        &  $\quad  4.5492{\rm e}-5$  & $1.9603$   \\
$\quad  1/ 160$ & $\quad  4.8043{\rm e}-5$   & $ 1.9903$        &  $\quad  1.1531{\rm e}-5$  & $1.9801$
\end{tabular*}}
\rule{\temptablewidth}{1pt}
\end{center}
\end{table}

\begin{figure}[htbp]
 \subfigure[The numerical panorama for $u(x,t)$ ]{\centering
\includegraphics[width=0.5\textwidth]{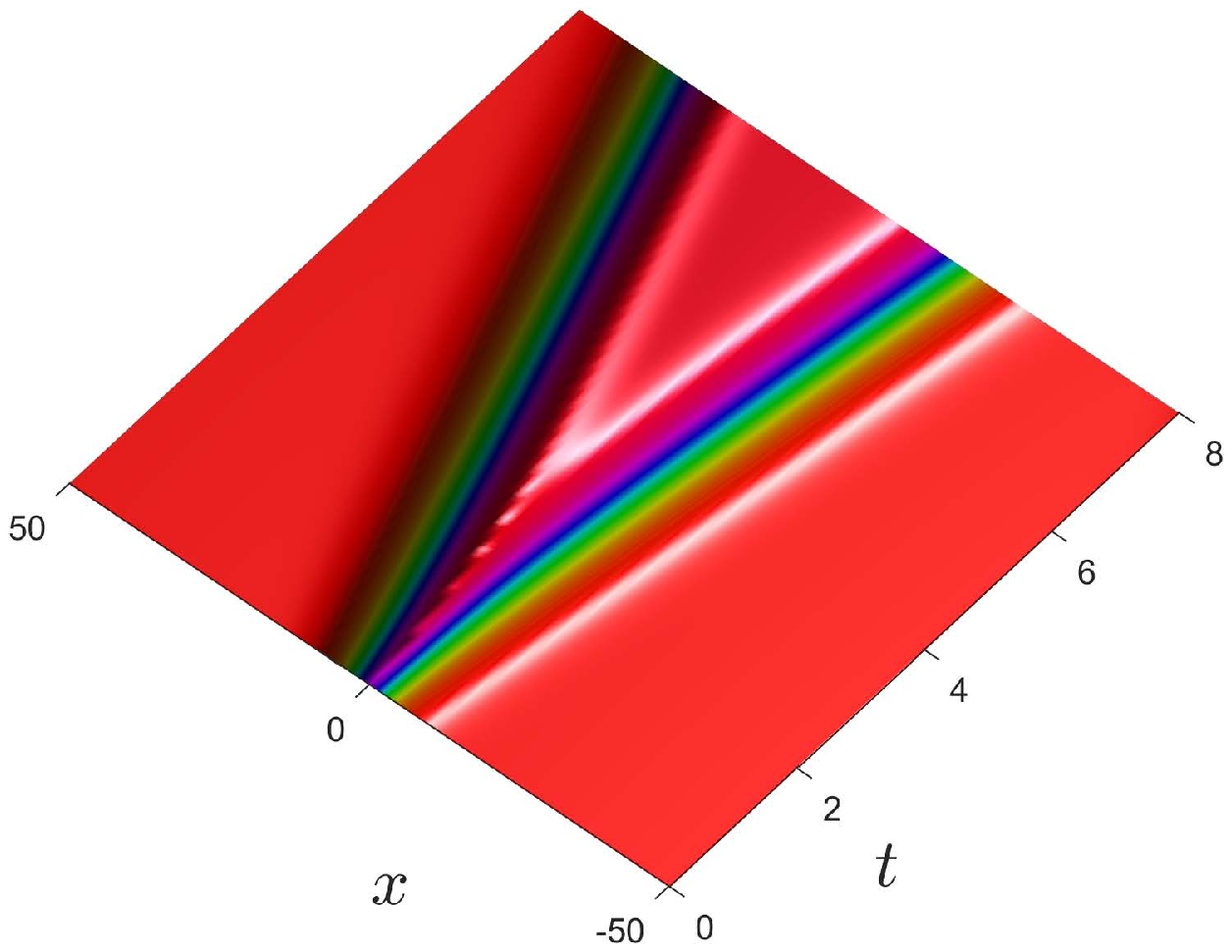}
}\subfigure[Numerical solution profiles]{\centering
\includegraphics[width=0.5\textwidth]{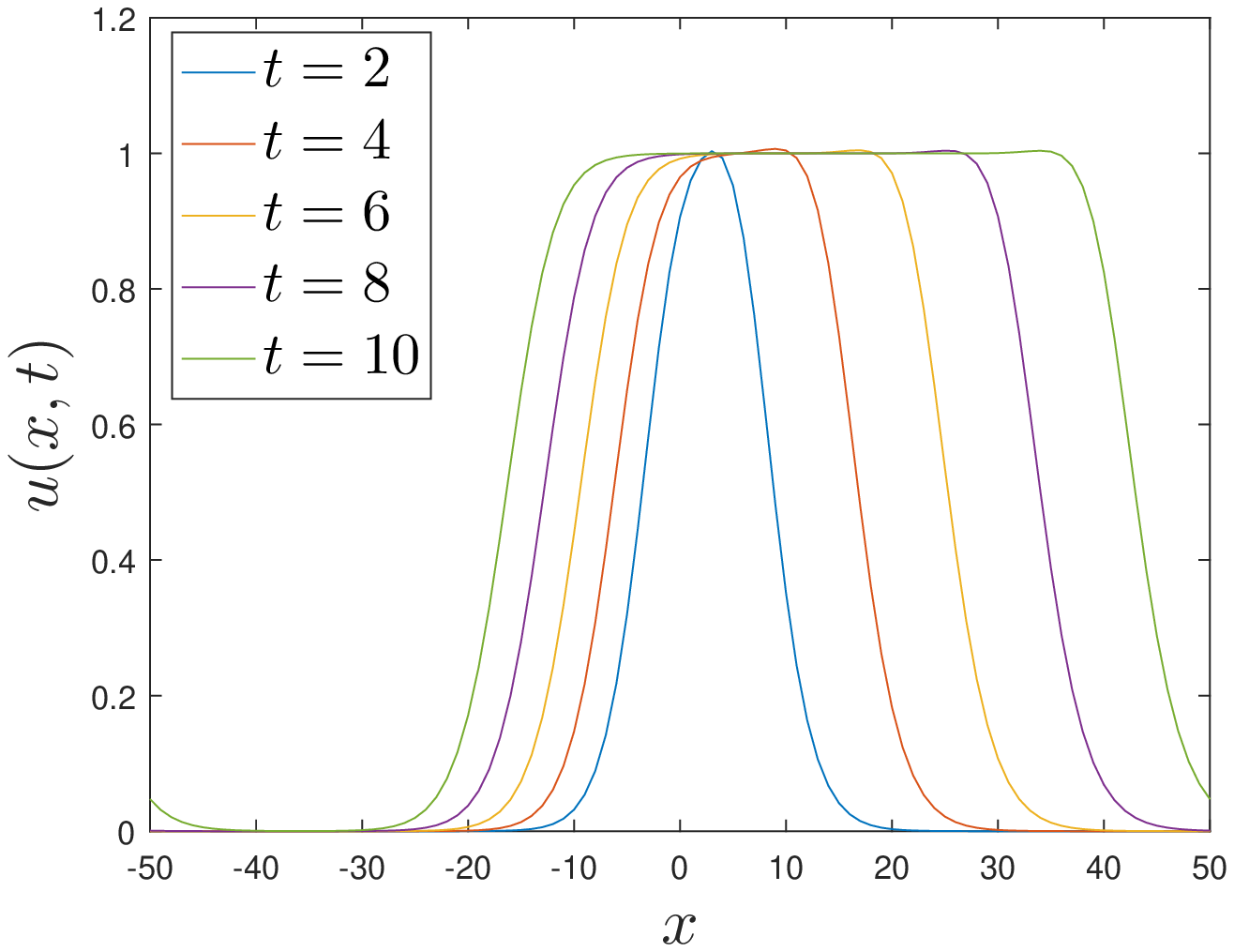}
}
\caption{(a) The numerical solution $t=8$, (b) the solution profiles for $u(x,t)$ with $t=2$, $4$, $6$, $8$, $10$} \label{fig4}
\end{figure}

\begin{figure}[htbp]
 \subfigure[The numerical panorama for $u(x,t)$ ]{\centering
\includegraphics[width=0.5\textwidth]{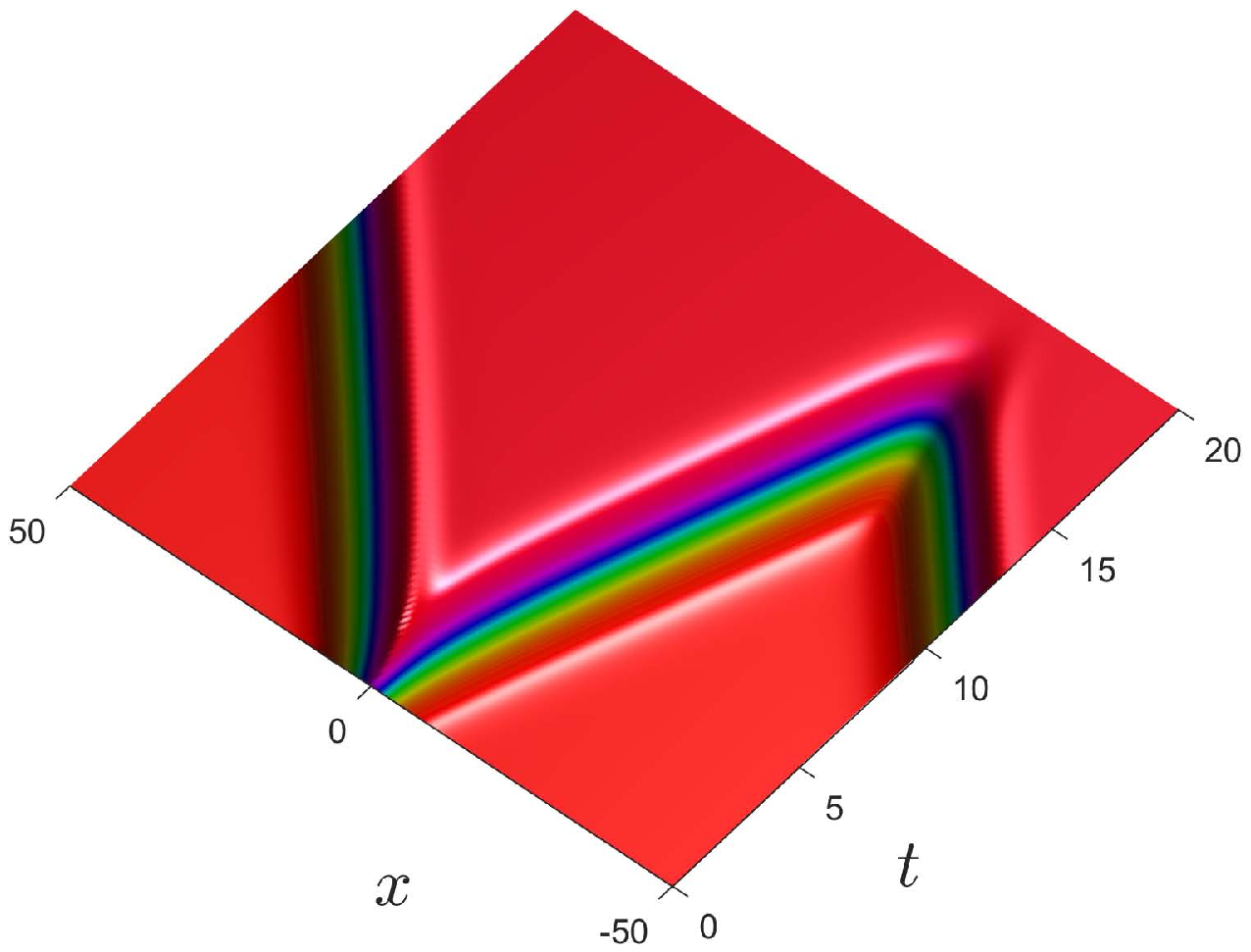}
}\subfigure[Numerical solution profiles]{\centering
\includegraphics[width=0.5\textwidth]{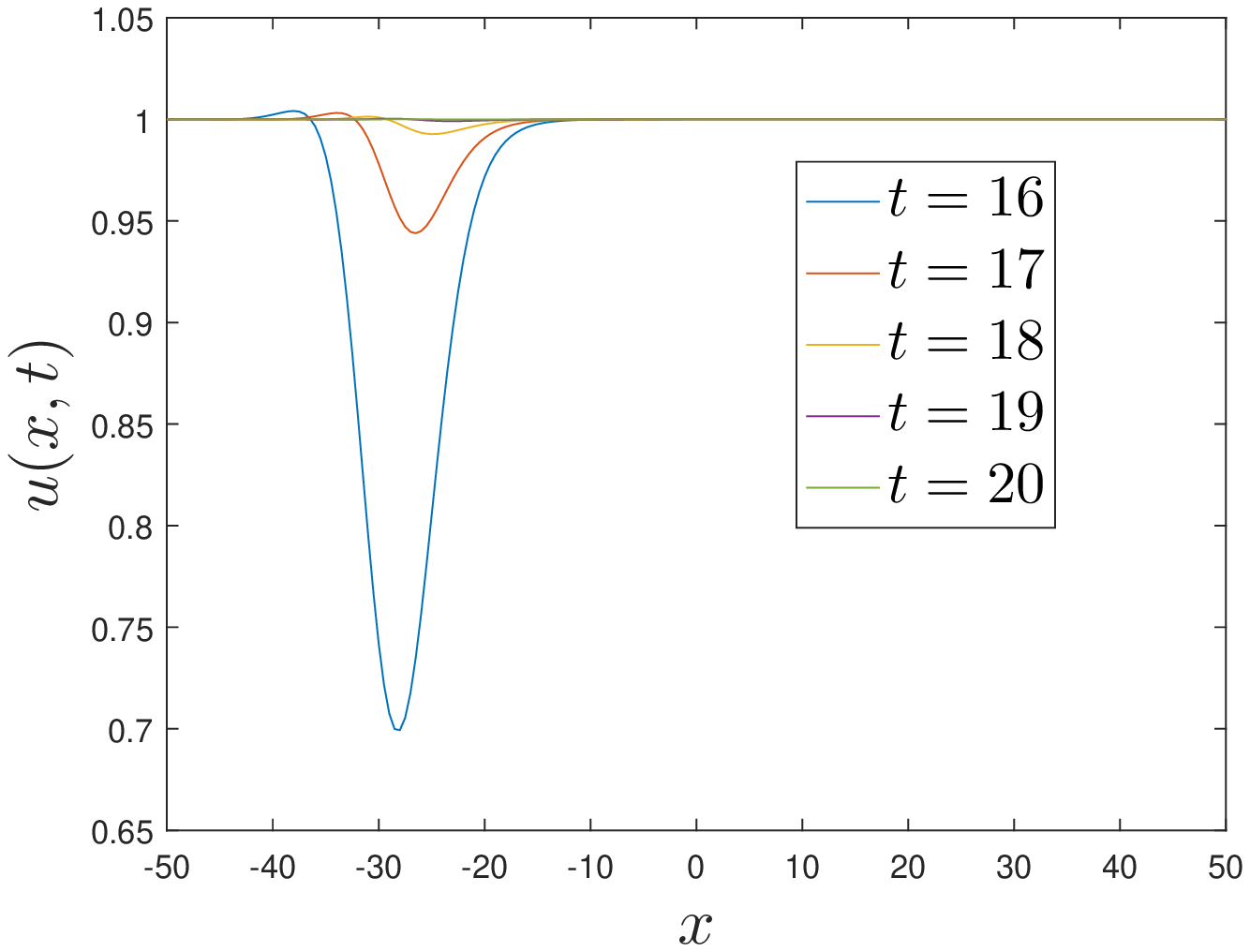}
}
\caption{(a) The numerical solution $t=20$, (b) the solution profiles for $u(x,t)$ with $t=16$, $17$, $18$, $19$, $20$} \label{fig5}
\end{figure}

\section{Conclusions}
\label{Sec:8}
\setcounter{equation}{0}
In the work, incorporating the reduction order method, a three-point four-order compact difference scheme and
a three-level linearized technique, we propose and analyze a linearized implicit, fourth-order compact scheme for the BBMB equation.
We have obtained the unique solvability, conservative invariant and boundedness.
Moreover, we have rigorously proved the maximum error estimation and the stability.
Compared presented scheme with those in the references, the novel fourth-order compact scheme reliably improve the computational accuracy.
Moreover, presented scheme can be extended to the BBMB equation with homogeneous boundary conditions without any difficulty.
In the future, extended our technique and idea to other nonlocal and nonlinear evolution
equations \cite{BH2017,Li2017,SZ2018,LV2019a,LV2019b,Wang2020} will be our on-going project.


\small{

}
\end{document}